\documentclass[english,reqno]{amsart}
\usepackage[pdftex]{graphicx}
\usepackage{graphicx,listings}
\usepackage{amsmath, appendix, ulem}
\usepackage{amsrefs}
\usepackage{amssymb, color}
\usepackage[english]{babel}
\usepackage{hyperref}
\usepackage[ansinew]{inputenc}
\usepackage[bookmarks]{}
\usepackage{enumerate}
\numberwithin{equation}{section}
\numberwithin{figure}{section}
\usepackage{tikz}
\usetikzlibrary{decorations.pathreplacing}
\usepackage[margin=1in]{geometry}

\usepackage{mathtools}  

\usepackage[font=footnotesize,labelfont=bf]{caption}
\usepackage{subcaption}

\usepackage{cases}            

\newcommand{\nn}{\nonumber}
\newcommand{\Ni}{\noindent}

\newcommand{\dr}{ \, {\rm d} r}
\newcommand{\dv}{ \, {\rm d} v}
\newcommand{\dt}{ \, {\rm d} t}
\newcommand{\dx}{ \, {\rm d} x}
\newcommand{\deta}{ \, {\rm d} \eta}
\newcommand{\dtheta}{ \, {\rm d} \theta}

\newcommand{\BigO}{{\mathcal{O}}}
\newcommand{\abs}[1]{\left\lvert#1\right\rvert}

\newcommand{\vpran}[1]{\left(#1\right)}
\newcommand{\Eps}{\varepsilon}

\newcommand{\F}{\mathcal{F}}
\newcommand{\Ord}{\mathcal{O}}
\newcommand{\R}{\mathbb{R}}
\newcommand{\N}{\mathbb{N}}
\newcommand{\eps}{\varepsilon}

\newcommand{\vs}{v^*}
\newcommand{\xs}{x^*}
\newcommand{\ys}{y^*}

\newcommand{\bvs}{\mathbf{v}^*}
\newcommand{\bxs}{\mathbf{x}^*}
\newcommand{\bv}{\mathbf{v}}
\newcommand{\bx}{\mathbf{x}}
\newcommand{\bxo}{\mathbf{x}_0}
\newcommand{\rs}{r^*}
\newcommand{\thetas}{\theta^*}
\newcommand{\ro}{r_0}
\newcommand{\thetao}{\theta_0}
\newcommand{\vt}{\tilde{v}}
\newcommand{\thetat}{\tilde{\theta}}
\newcommand{\rt}{\tilde{r}}
\newcommand{\tEndI}{\tau_1}
\newcommand{\tEndII}{\tau_2}
\newcommand{\tEndIII}{\tau_3}
\newcommand{\tEndIV}{\tau_4}
\newcommand{\tEndIVA}{\tau_{4A}}
\newcommand{\tEndIVB}{\tau_{4B}}
\newcommand{\As}{A^*}
\newcommand{\thetatwo}{\varphi^*}
\newcommand{\thetathree}{\tilde{\varphi}}
\newcommand{\thetatwos}{\omega^*}
\newcommand{\thetathrees}{\tilde{\omega}}

\newcommand{\leqs}{\leqslant}
\newcommand{\geqs}{\geqslant}

\newcommand{\RP}{(RP)}
\newcommand{\RN}{(RN)}

\newtheorem{theorem}{Theorem}

\newtheorem{remark}{Remark}
\newtheorem{assumption}{Assumption}

\renewcommand{\leq}{\leqs}
\renewcommand{\geq}{\geqs}

\begin{document}

\title{Small inertia regularization of an anisotropic aggregation model
}

\author[J.H.M.~Evers]{Joep H.M.~Evers}
\address{Department of Mathematics and Statistics, Dalhousie University, PO Box 15000, Halifax, NS B3H 4R2, Canada 
\\
\& Department of Mathematics, Simon Fraser University, 8888 University Dr., Burnaby, BC V5A 1S6, Canada}
\email{jevers@sfu.ca}

\author[R.C.~Fetecau]{Razvan C.~Fetecau}
\address{Department of Mathematics, Simon Fraser University, 8888 University Dr., Burnaby, BC V5A 1S6, Canada}
\email{van@sfu.ca}


\author[W.~Sun]{Weiran Sun}
\address{Department of Mathematics, Simon Fraser University, 8888 University Dr., Burnaby, BC V5A 1S6, Canada}
\email{weirans@sfu.ca}



\begin{abstract}

We consider an anisotropic first-order ODE aggregation model and its approximation by a second-order relaxation system. The relaxation model contains a small parameter $\eps$, which can be interpreted as inertia or response time. We examine rigorously the limit $\eps \to 0$ of solutions  to the relaxation system. Of major interest is how discontinuous (in velocities) solutions to the first-order model are captured in the zero-inertia limit. We find that near such discontinuities, solutions to the second-order model perform fast transitions within a time layer of size $\BigO(\eps^{2/3})$. We validate this scale with numerical simulations.



\end{abstract}

\maketitle

\section{Introduction}

The aim of the present paper is to provide a rigorous foundation for a certain mathematical model for self-collective behaviour. The model describes the evolution of positions~$x_i$ ($i=1,\dots,N$) of $N$ particles (individuals) in $\mathbb{R}^d$; in its simplest form it reads:
\begin{subequations} 
\label{eqn:fo-nobz} 
\begin{align}
    \frac{\dx_i}{\dt} &= v_i, \\
    v_i &= -\frac{1}{N} \displaystyle\sum_{j\neq i}\nabla_{x_i}K(|x_i-x_j|), \label{eqn:vi-nobz} 
\end{align} 
\end{subequations}
where $K$ is a radially symmetric  aggregation potential which models inter-individual social interactions.  The particular form of $K$ depends on the specific application, and it typically incorporates long-range attraction and short-range repulsion interactions between individuals. 

Models for self-collective behaviour have been of central interest lately, due to their diverse applications in a wide range of areas, such as the formation of  biological groups (fish schools, bird flocks, insect swarms) \cite{Camazine_etal},  robotics and space missions \cite{JiEgerstedt2007}, opinion formation \cite{MotschTadmor2014},  traffic and pedestrian flow \cite{Helbing95} and social networks \cite{Jackson2010}. 

Model \eqref{eqn:fo-nobz} and its continuum/macroscopic counterpart \cite{BodnarVelazquez2} have been hugely popular in the aggregation literature of the last decade. On one hand, model \eqref{eqn:fo-nobz} can capture a wide variety of  experimentally-observed self-collective or swarm behaviours \cite{Parish:Keshet, Ballerini_etal, Camazine_etal, Theraulaz:Deneubourg}, which gives practical grounds to the model. Aggregation patterns that can be achieved numerically with model  \eqref{eqn:fo-nobz} include uniform densities in a ball, uniform densities on a co-dimension one manifold (ring in 2D, sphere in 3D), annuli, soccer balls \cite{LeToBe2009, KoSuUmBe2011, BrechtUminsky2012, FeHuKo11,Balague_etalARMA}.  On the other hand, the continuum limit of model  \eqref{eqn:fo-nobz}  has attracted wide interest from analysts and a substantial amount of rigorous results have been established for the macroscopic model in recent years \cite{BodnarVelazquez2, BertozziLaurent, BertozziLaurentRosado, CarrDiFranFigLauSlep11}.

A fundamental assumption in model \eqref{eqn:fo-nobz} is that the interactions are {\em isotropic}, as they depend only on the distance between individuals (by radial symmetry, $K(x)=K(|x|)$). This assumption is not realistic, in particular for biological applications, where most species have a restricted zone of social perception, defined by the limitations of their perception ranges (e.g., a restricted field of vision) \cite{Seidl:Kaiser, Kunz:Hemelrijk2012}. Nevertheless, despite the extensive literature on model \eqref{eqn:fo-nobz},  a systematic study of its (more realistic) anisotropic extensions has only been addressed very recently \cite{Gulikers,EversIFAC, EversFetecauRyzhik}.  The primary goal of this paper is to consider the anisotropic model investigated in \cite{EversFetecauRyzhik} and set a rigorous framework for the well-posedness of its solutions.

Perception restrictions to model  \eqref{eqn:fo-nobz} can be introduced via weights $w_{ij}$ that limit the influence of individuals $j$ on the reference individual $i$ \cite{EversFetecauRyzhik}:
 \begin{equation}
 \label{eqn:v-weighted}
 v_i = -\frac{1}{N} \displaystyle\sum_{j\neq i}\nabla_{x_i}K(|x_i-x_j|) w_{ij}.
 \end{equation}
 In \cite{EversFetecauRyzhik} the weights are set to model sensorial restrictions due a limited field of vision. Specifically, given a reference individual located at $x_i$ moving with velocity $v_i$, the weights $w_{ij}$ were assumed to depend on the relative position $x_j-x_i$ of individual $j$ with respect to the current direction of motion $v_i$ of individual $i$. Mathematically, $w_{ij}$ are modelled in \cite{EversFetecauRyzhik} as
 \begin{equation}
 \label{eqn:weights}
 w_{ij}= g\left(\dfrac{x_i-x_j}{|x_i-x_j|}\cdot \dfrac{v_i}{|v_i|}\right),
 \end{equation}
with a function $g$ chosen so that $w_{ij}$ are largest when $j$ is right ahead of individual $i$ ($x_j-x_i$ is in the same direction of $v_i$) and lowest when $j$ is right behind individual $i$ ($x_j-x_i$ in the opposite direction of $v_i$). Note that the weights $w_{ij}$ are {\em not} symmetric, as in general, $w_{ij} \neq w_{ji}$.

Combining \eqref{eqn:v-weighted} and \eqref{eqn:weights} one arrives at the following anisotropic extension of \eqref{eqn:fo-nobz}, which is the main object of study of the present paper:
\begin{subequations}
\label{eqn:fo-bz}
\begin{align}
    \frac{\dx_i}{\dt} &= v_i, \\
    v_i &= -\frac{1}{N} \displaystyle\sum_{j\neq i}\nabla_{x_i}K(|x_i-x_j|)\,g\left(\dfrac{x_i-x_j}{|x_i-x_j|} \cdot \dfrac{v_i}{|v_i|}\right). \label{eqn:vi-bz}
  \end{align}
\end{subequations}

Mathematically, the major distinction between \eqref{eqn:fo-nobz} and \eqref{eqn:fo-bz} is that in the latter model the velocities $v_i$ are no longer explicitly given in terms of the spatial configuration $\{x_1,x_2,\dots,x_N\}$, but are defined instead through the {\em implicit} equation \eqref{eqn:vi-bz}. Among other difficulties, identified and discussed in \cite{EversFetecauRyzhik}, this subtle distinction leads to a {\em loss of smoothness} of solutions of model \eqref{eqn:fo-bz}, as roots $v_i$ of \eqref{eqn:vi-bz} evolve dynamically along with the spatial configuration $\{x_1, x_2, \dots, x_N\}$. Consequently, velocities have to be allowed to be discontinuous at certain jump times and in addition, a criterion for selecting (uniquely) the correct/physical jump has to be identified. 

The key idea in \cite{EversFetecauRyzhik} is to introduce a relaxation term in \eqref{eqn:vi-bz} and consider the following regularized system
\begin{subequations}
\label{eqn:so}
\begin{align}
    \frac{\dx_i}{\dt} &= v_i, \\
    \eps\dfrac{\dv_i}{\dt} &= -v_i -\dfrac1N\displaystyle\sum_{j\neq i}\nabla_{x_i}K(|x_i-x_j|)\,g\left(\dfrac{x_i-x_j}{|x_i-x_j|} \cdot \dfrac{v_i}{|v_i|} \right) \label{eqn:vi-so} .
      \end{align}
\end{subequations}
From the biological point of view, \eqref{eqn:so} introduces a small inertia or response time for the individuals. The idea of considering a second-order model such as \eqref{eqn:so} goes back in fact to the original  derivation of the isotropic model  \eqref{eqn:fo-nobz}. Indeed, model  \eqref{eqn:fo-nobz} was formally derived in \cite{BodnarVelazquez1} from Newton's second law of motion \eqref{eqn:so} (for isotropic interactions $ g \equiv 1$) by neglecting the inertia terms; this amounts to assume that individuals change their velocities instantaneously. The authors of \cite{EversFetecauRyzhik} bring back the original second-order model and argue that restoring a small inertia/response time $\eps$, and then passing $\eps \to 0$, is the correct mechanism of capturing the ``physically" relevant jump solutions of the first-order anisotropic model \eqref{eqn:fo-bz}. The limiting procedure is then illustrated numerically in \cite{EversFetecauRyzhik} for various jump scenarios in two dimensions.

The main goal of the present paper is to validate rigorously the zero inertia limit of solutions to \eqref{eqn:so}. We restrict our attention to the two dimensional case, which is the setup of all numerical simulations in \cite{EversFetecauRyzhik}. For as long as solutions of the first-order model \eqref{eqn:fo-bz} exist and are smooth, the limit $\eps \to 0$ of system \eqref{eqn:so} follows from a classical theorem of Tikhonov \cite{Tikhonov1952,Vasileva1963}, regardless of the space dimension. This result however does not extend to discontinuous solutions of  \eqref{eqn:fo-bz}, for which roots $v_i$ of \eqref{eqn:vi-bz} can instantaneously get lost. For this reason our focus in this paper is to study the limit $\eps \to 0$ of solutions to \eqref{eqn:so} at the {\em onset} of the jump discontinuity for model \eqref{eqn:fo-bz}. The limit turns out to be subtle and for a better exposition we build the tools in stages, by starting with a toy problem in one dimension. The one-dimensional problem captures the essential features of the limiting process, in particular the $\eps$-dependent time scales in which the dynamics through a jump occurs.

The zero inertia limit of solutions to second order models has recently been investigated in the context of continuum/PDE models for collective behaviour. In \cite{FeSu2015} the authors study the continuum analogue of \eqref{eqn:so} (with $g\equiv1$, i.e., the isotropic version) and show that its solutions converge as $\eps \to 0$ to solutions of the  PDE counterpart of the first-order model \eqref{eqn:fo-nobz}. Furthermore, \cite{FeSuTa2016} considers a generalization, where an additional alignment term is included in the equation for velocity.  In this sense, the present paper complements these works, by investigating the zero inertia limit at the discrete/ODE level, and also with the caveat of allowing the interactions to be anisotropic.

The summary of the paper is as follows. Section \ref{sec:prelim} presents the two-dimensional anisotropic model and provides motivation for the studies in this paper.  In Section \ref{sec:oned} we introduce a one-dimensional model which allows us to discuss some key elements of the analysis in a  simpler setting. The main result in this section is Theorem \ref{thm: conv k ell 1D}, which  represents an extension of the Tikhonov's theorem \cite{Tikhonov1952,Vasileva1963} to nonsmooth settings. Section \ref{sec:2D-varr} considers the two-dimensional model \eqref{eqn:so}; the main result of the paper is Theorem \ref{thm: conv quad lin 2D R neg}, which establishes the zero-inertia limit of solutions to \eqref{eqn:so} through jump discontinuities of the first-order model \eqref{eqn:fo-nobz}. Finally,  in Section \ref{sec:num} we validate the analytical results with numerical simulations.


\section{Preliminaries and motivation}
\label{sec:prelim}

In this section we summarize briefly the main findings in \cite{EversFetecauRyzhik} that are relevant for the present paper. 

\subsection{First-order model and jump discontinuities.}\label{sec:intro fo} As described in the Introduction, the setup of the anisotropic model  \eqref{eqn:fo-nobz} in \cite{EversFetecauRyzhik} takes visual limitations into consideration. Indeed, consider a reference individual $i$ that interacts with a generic individual $j$, and denote by $\phi_{ij}$  the angle between $x_j-x_i$ and~$v_i$:
 \[
\dfrac{x_i-x_j}{|x_i-x_j|} \cdot \dfrac{v_i}{|v_i|} = - \cos \phi_{ij}.
\]

To model a field of vision, the weights $w_{ij} = g(-\cos \phi_{ij})$ (see \eqref{eqn:weights}) should be the largest ($=1$) for $\phi_{ij}=0$ (the vectors $x_j-x_i$ and $v_i$ are parallel) and
the lowest (possibly $0$) for~$\phi_{ij}=\pi$ ($x_j-x_i$ and $v_i$ anti-parallel) --- see Figure~\ref{fig:vision}(a) for an illustration. A weight function $g$ that captures this behaviour is shown in Figure~\ref{fig:vision}(b); there $g(-\cos\phi)=[\tanh(a(\cos\phi + 1 - b/\pi))+1]/c$, with $c$ a normalization constant such that $g(-1)=1$.

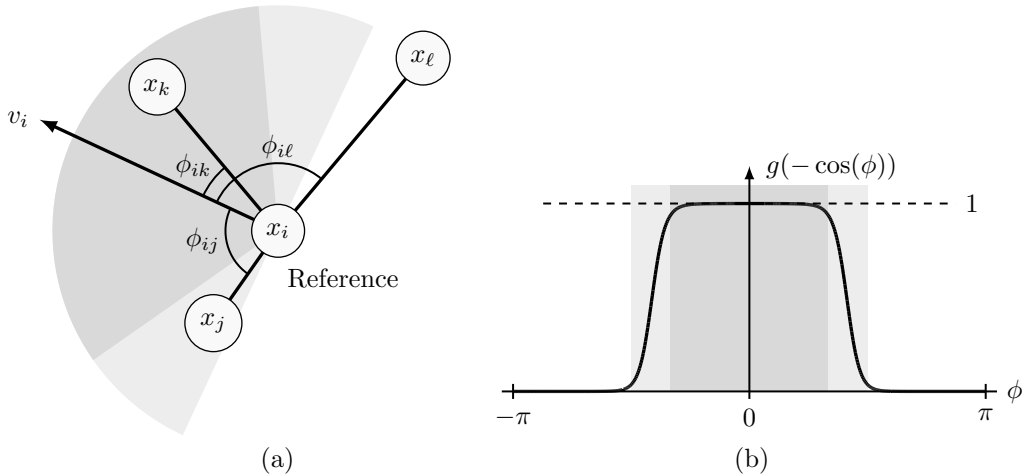
\begin{figure}[ht]%
\centering
                \begin{tikzpicture}[>= latex]

                \begin{scope}[scale=0.25] 

                \begin{scope}[scale=0.4]

                \fill[black!7!white](0,0) -- +(245:30) arc(245:65:30) -- cycle;
                \fill[black!15!white](0,0) -- +(215:30) arc(215:95:30) -- cycle;

                \draw[->, line width=1.25] (0,0) -- +(155:35) node[left] {$v_i$};

                \begin{scope} 
					\tikzstyle{every node} = [draw,circle, fill=gray!5, line width = 0.5]
                \draw[line width=0.75] (155:11) arc (155:130:11);
                \draw[line width=1.25] (0,0) -- +(130:25) node {$x_k$};
                \draw[line width=0.75] (155:7) arc (155:235:7);
                \draw[line width=1.25] (0,0) -- +(235:15)node {$x_j$};
                \draw[line width=0.75] (155:9) arc (155:50:9);
                \draw[line width=1.25] node {$x_i$} (0,0) -- +(50:30) node {$x_\ell$};
       					\end{scope} 
		     	\end{scope}
        
                \draw (-4,-0.5) node {$\phi_{ij}$};
                \draw (-4.5,3.5) node {$\phi_{ik}$};
                \draw (0,4.5) node {$\phi_{i\ell}$};
                \draw (0,-1.5) node[below right] {Reference};       

            \end{scope}

            \end{tikzpicture}$~~~$ 
\begin{tikzpicture}[>= latex]
                \begin{scope}[xscale=1,yscale=2.5]
                	\fill[black!7!white] (-90*2*pi/360,0) rectangle (90*2*pi/360,1.1);
                	\fill[black!15!white] (-60*2*pi/360,0) rectangle (60*2*pi/360,1.1);
                	\draw[line width=0.75, dashed] (-pi+0.4,1) -- (pi-0.4,1) node[right] {$1$};
					\draw[->,line width=0.75] (0,0) -- (0,1.2);
					\draw[line width=0.75] (-pi-0.15,0) -- (pi+0.15,0) node[right] {$\phi$};
					\draw[line width=0.75] (-pi,0.04) -- (-pi,-0.04) node[below] {$-\pi$};
					\draw[line width=0.75] (pi,0.04) -- (pi,-0.04) node[below] {$\pi$};
					\draw[line width=0.75] (0,0) -- (0,-0.04) node[below] {$0$};

					\pgfmathsetmacro\phiStart {-3.14};
                \pgfmathsetmacro\dphi {0.01};
                \pgfmathsetmacro\phiStOne {\phiStart+\dphi};
                 \pgfmathsetmacro\phiStTwo {\phiStart+2*\dphi};
					\foreach \x  in {\phiStOne,\phiStTwo,...,-0.2}	{
					\pgfmathsetmacro\xprev {\x-\dphi};
					\pgfmathsetmacro\a {6};	
					\pgfmathsetmacro\b {4};	
					\pgfmathsetmacro\c {tanh(\a*(2-\b/pi))+1};	
					\pgfmathsetmacro\y {(tanh(\a*(cos(360*\x/(2*pi))+1-\b/pi))+1)/\c};
					\pgfmathsetmacro\yprev {(tanh(\a*(cos(360*\xprev/(2*pi))+1-\b/pi))+1)/\c};	
					
					
					\draw[line width=1.25] (\xprev,\yprev) -- (\x,\y);
					}
					\pgfmathsetmacro\a {6};	
					\pgfmathsetmacro\b {4};
					\pgfmathsetmacro\c {tanh(\a*(2-\b/pi))+1};		
					\pgfmathsetmacro\x {-0.2};		
					\pgfmathsetmacro\y {(tanh(\a*(cos(360*\x/(2*pi))+1-\b/pi))+1)/\c};
					\pgfmathsetmacro\xnext {0.2};	
					\pgfmathsetmacro\ynext {(tanh(\a*(cos(360*\xnext/(2*pi))+1-\b/pi))+1)/\c};				
					\draw[line width=1.25] (\x,\y) -- (\xnext,\ynext);

\pgfmathsetmacro\phiStart {0.2};
                \pgfmathsetmacro\dphi {0.01};
                \pgfmathsetmacro\phiStOne {\phiStart+\dphi};
                 \pgfmathsetmacro\phiStTwo {\phiStart+2*\dphi};
					\foreach \x  in {\phiStOne,\phiStTwo,...,3.14}	{
					\pgfmathsetmacro\xprev {\x-\dphi};
					\pgfmathsetmacro\a {6};	
					\pgfmathsetmacro\b {4};	
					\pgfmathsetmacro\c {tanh(\a*(2-\b/pi))+1};	
					\pgfmathsetmacro\y {(tanh(\a*(cos(360*\x/(2*pi))+1-\b/pi))+1)/\c};
					\pgfmathsetmacro\yprev {(tanh(\a*(cos(360*\xprev/(2*pi))+1-\b/pi))+1)/\c};	
					
					
					\draw[line width=1.25] (\xprev,\yprev) -- (\x,\y);
					}

					\draw (0.1,1.2) node[right] {$g(-\cos(\phi))$};
				\end{scope}
\end{tikzpicture}
\\
 (a) \hspace{0.34\linewidth} (b)
\caption{(a): An illustration of the visual perception of a reference individual $i$: the field of vision (dark grey), a peripheral zone (light grey) and a blind zone (white). Interactions are weighted: $w_{ik}>w_{ij}>w_{i\ell}$. (b): The weight function $g(-\cos\phi)=[\tanh(a(\cos\phi + 1 - b/\pi))+1]/c$. The parameters are: $a=6$ and $b=4$, where $a$ controls the steepness of the graph and $b$ controls its width. The function takes values close to $1$ in a region around $\phi=0$ (field of vision, dark grey), has a steep decay to nearly $0$ in the peripheral zone (light grey), and takes negligible values near $\phi = \pm \pi$ (blind zone, white).}%
\label{fig:vision}%
\end{figure}

A well-posedness theory for solutions to model \eqref{eqn:fo-bz} has been established in \cite{EversFetecauRyzhik}. Up to some technical issues (omitted here), by which certain initial configurations in phase space are excluded, there exist unique local solutions $x_i(t)$, $v_i(t)$ ($i=1,\dots,N$) to \eqref{eqn:fo-bz} that are continuous in time. Most notable for the present work is the root loss of \eqref{eqn:vi-bz} alluded to above, at which such continuous solutions to \eqref{eqn:fo-bz} break. While the well-posedness theory and the breakdown of solutions to model \eqref{eqn:fo-bz} hold in any dimension, the numerical investigations in \cite{EversFetecauRyzhik}, as well the analytical results derived in the present paper, focus specifically on the two-dimensional case. Hence, throughout the paper the anisotropic model \eqref{eqn:fo-bz} is considered in $d=2$ dimensions.

To explain the loss of smoothness, consider a fixed spatial configuration $\{x_1,x_2,\dots,x_N\} \subset \R^2$ and inspect the roots $v_i$ of equation \eqref{eqn:vi-bz}. To this purpose, use the polar coordinate representation $v_i=r_i(\cos\theta_i,\sin\theta_i)^T$ for the velocity $v_i$, and write \eqref{eqn:vi-bz} as
\begin{equation}\label{eqn fixed point polar coordinates}
r_i{\cos\theta_i\choose\sin\theta_i} = -\dfrac1N\displaystyle\sum_{j\neq i}\nabla_{x_i}K(|x_i-x_j|)\,g\left(\dfrac{x_i-x_j}{|x_i-x_j|}\cdot{\cos\theta_i\choose\sin\theta_i}\right).
\end{equation}
By taking the inner product with $(-\sin\theta_i,\cos\theta_i)^T$ and
$(\cos\theta_i,\sin\theta_i)^T$, the vector equation \eqref{eqn fixed
  point polar coordinates} can be written as
 \begin{equation}
 \label{eqn:HRroots}
 H_i(x_1, \dots, x_N,\theta_i)=0, \qquad r_i = R_i(x_1,\dots,x_N,\theta_i),
 \end{equation}
 where the functions $H_i$, $R_i$ ($i=1,\dots,N$) are defined as 
 \begin{subequations}
\label{eqn:FR}
\begin{align}
H_i(x_1,\dots,x_N,\theta)&=-\dfrac1N\displaystyle\sum_{j\neq i}\nabla_{x_i}K(|x_i-x_j|)\cdot{-\sin\theta \choose\cos\theta}\,g\left(\dfrac{x_i-x_j}{|x_i-x_j|}\cdot{\cos\theta \choose\sin\theta}\right),\\
R_i(x_1,\dots,x_N,\theta)&=-\dfrac1N\displaystyle\sum_{j\neq i}\nabla_{x_i}K(|x_i-x_j|)\cdot{\cos\theta \choose\sin\theta}\,g\left(\dfrac{x_i-x_j}{|x_i-x_j|}\cdot{\cos\theta \choose\sin\theta}\right).
\end{align}
\end{subequations}
The advantage of using polar coordinates for $v_i$ is that the first equation in \eqref{eqn:HRroots} is a {\em scalar} equation to be solved for $\theta_i$, while the second equation yields $r_i$ explicitly in terms of $\theta_i$. Note that for a root $\theta_i$ of $H_i$ to be admissible, one needs $R_i$ evaluated at $\theta_i$ to be non-negative.  

The spatial configuration $\{x_1, x_2, \dots, x_N\}$ changes in time, as it evolves according to \eqref{eqn:fo-bz} and hence, the solutions $\theta_i$, $r_i$ of \eqref{eqn:HRroots} evolve in time as well, along with the configuration. Consequently,  jumps in $\theta_i$, $r_i$ can occur at certain spatial configurations through the dynamical evolution. We illustrate this point with an example.  

Consider the four particle ($i=1,\dots,4$) run presented in  \cite{EversFetecauRyzhik} --- see Figure~\ref{fig:trajectories zoom r theta}(a). For the purpose of this discussion, it is enough to consider the evolution of the root $\theta_1$ of $H_1$, corresponding to particle $1$ (top left particle in Figure~\ref{fig:trajectories zoom r theta}(a)). The solid black line in Figure \ref{fig:HR} shows the plot of $H_1$ as a function of $\theta$ at the initial time. Note that in general, the functions $H_i$ can have several roots, but once a {\em simple} root is selected at the initial time, there exists a (local) continuous solution to \eqref{eqn:fo-bz} that starts from that phase space configuration \cite{EversFetecauRyzhik}. For the run presented here, $\theta_1 \approx -1.00$ 
at the initial time. As the spatial configuration evolves in time,  $H_1$ changes its profile and the root $\theta_1$ evolves as well (see the dashed and dash-dotted black lines in Figure \ref{fig:HR}). The critical time is when $\theta_1$ becomes a {\em double} root of $H_1$ (see the dash-dotted line in Figure \ref{fig:HR}; the double root $\thetas_1$ is denoted by the filled circle); had the spatial configuration continued to evolve past this time, in the direction of the current velocity, the root $\theta_1= \thetas_1$ would be instantaneously lost.  To extend the dynamics of  \eqref{eqn:fo-bz} beyond breakdown, a jump  in $\theta_1$ (more precisely in the velocity $v_1$) would need to be enforced.

\begin{figure}%
\centering
\includegraphics[height=0.45\textwidth]{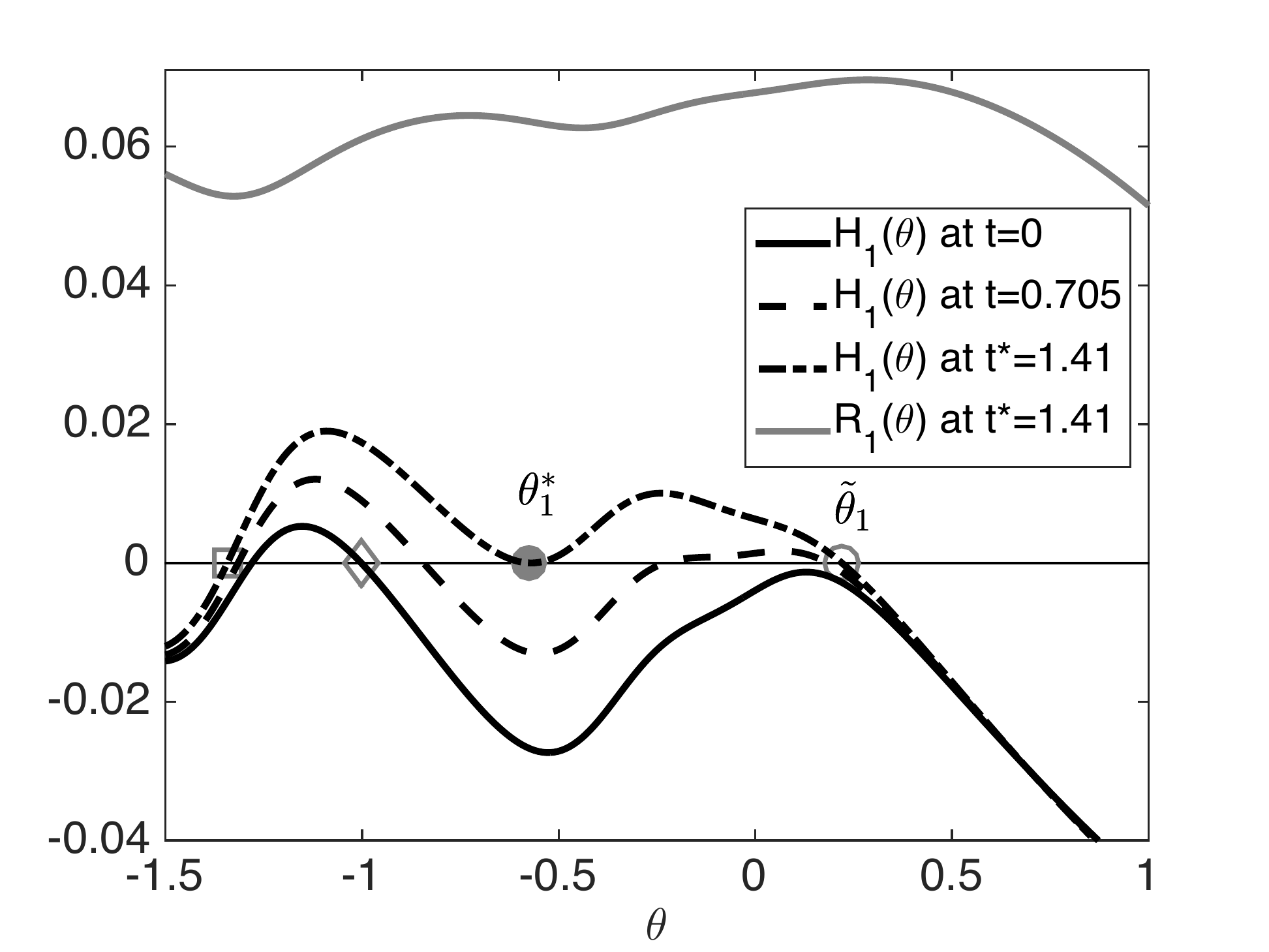}
\caption{The function $H_1$ associated to particle 1, drawn as a function of $\theta$ at three time instances: the initial time $t=0$, the breakdown time $t=t^*=1.41$ and the time halfway in between $t=t^*/2=0.705$ -- see the four particle simulation in Figure \ref{fig:trajectories zoom r theta}. At time $t=0$ (solid black line) a simple root is present at $\theta \approx -1.00$, indicated by a diamond. As time evolves, so does the function $H_1$. At time $t=t^*$ (dash-dotted line) the simple root has evolved into a double root at $\theta=\thetas_1\approx-0.57$ (filled circle). If the evolution is extended after breakdown in accordance with solutions of the relaxation model \eqref{eqn:so} --see Section \ref{sec:intro relax}-- then $\theta_1$ jumps to $\theta=\thetat_1\approx 0.22$, indicated by the open circle. Note that, at $t=t^*$, we have that $R_1(\thetat_1)>0$. The square at $\theta\approx -1.34$ indicates another simple root of $H_1$ at $t=t^*$.}
\label{fig:HR}%
\end{figure}

Breakdown times, characterized by double-root loss as in Figure \ref{fig:HR}, are the central point of discussion for the present work. Throughout the paper we denote these breakdown times generically by $t^\ast$, and we also add a superscript $^\ast$ to refer to the phase space configuration at $t^\ast$. It is important to note that at a breakdown time there are typically several roots of 
\eqref{eqn:vi-bz} that the velocity can jump to. For instance, in the simulation presented above, $H_1$ at breakdown (see the dash-dotted line in Figure \ref{fig:HR}) has several simple roots that $\theta_1$ can jump to (the ones within the domain of the plot are indicated by the open circle and the square). To offer a consistent, biologically meaningful mechanism to select a jump, the authors in \cite{EversFetecauRyzhik} propose the relaxation model \eqref{eqn:so}.

\subsection{Relaxation model} \label{sec:intro relax}
In \cite{EversFetecauRyzhik} it has been demonstrated numerically that the relaxation model \eqref{eqn:so} can be used to capture discontinuous solutions to \eqref{eqn:fo-bz}. Before a breakdown, the approximation of solutions to \eqref{eqn:fo-bz} with solutions to \eqref{eqn:so} is validated in fact (under certain assumptions on the initial phase space configuration) by the classical analytical results of Tikhonov \cite{Tikhonov1952,Vasileva1963}. Our interest here is the behaviour of solutions to  \eqref{eqn:so} upon approaching a breakdown time $t^\ast$ of \eqref{eqn:fo-bz}.

\begin{figure}%
\centering
\vspace{-4cm}
\includegraphics[width=0.8\textwidth]{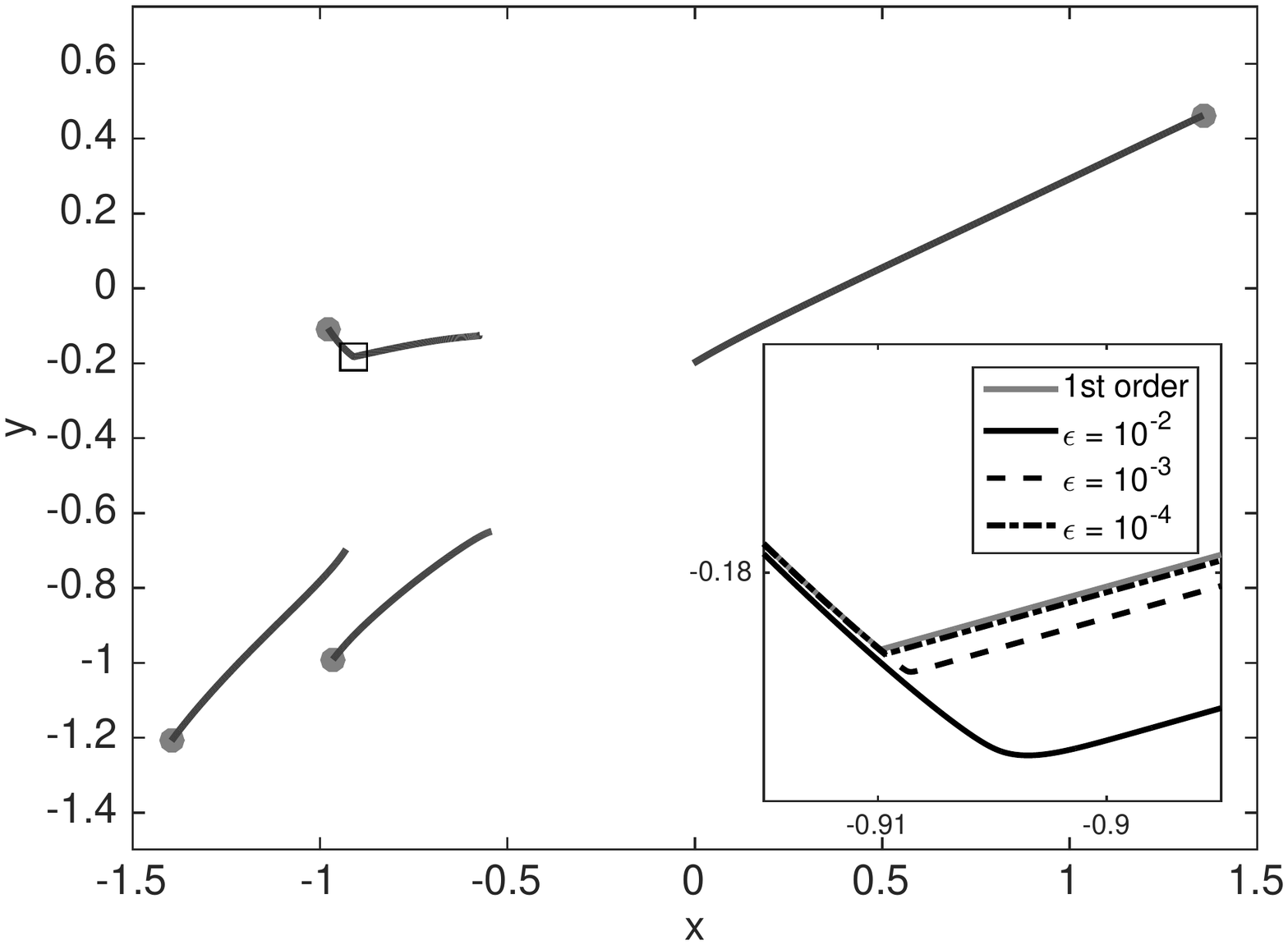}\\ \vspace{-3.7cm}
(a) \\ \vspace{0.25cm}
\includegraphics[height=0.35\textwidth]{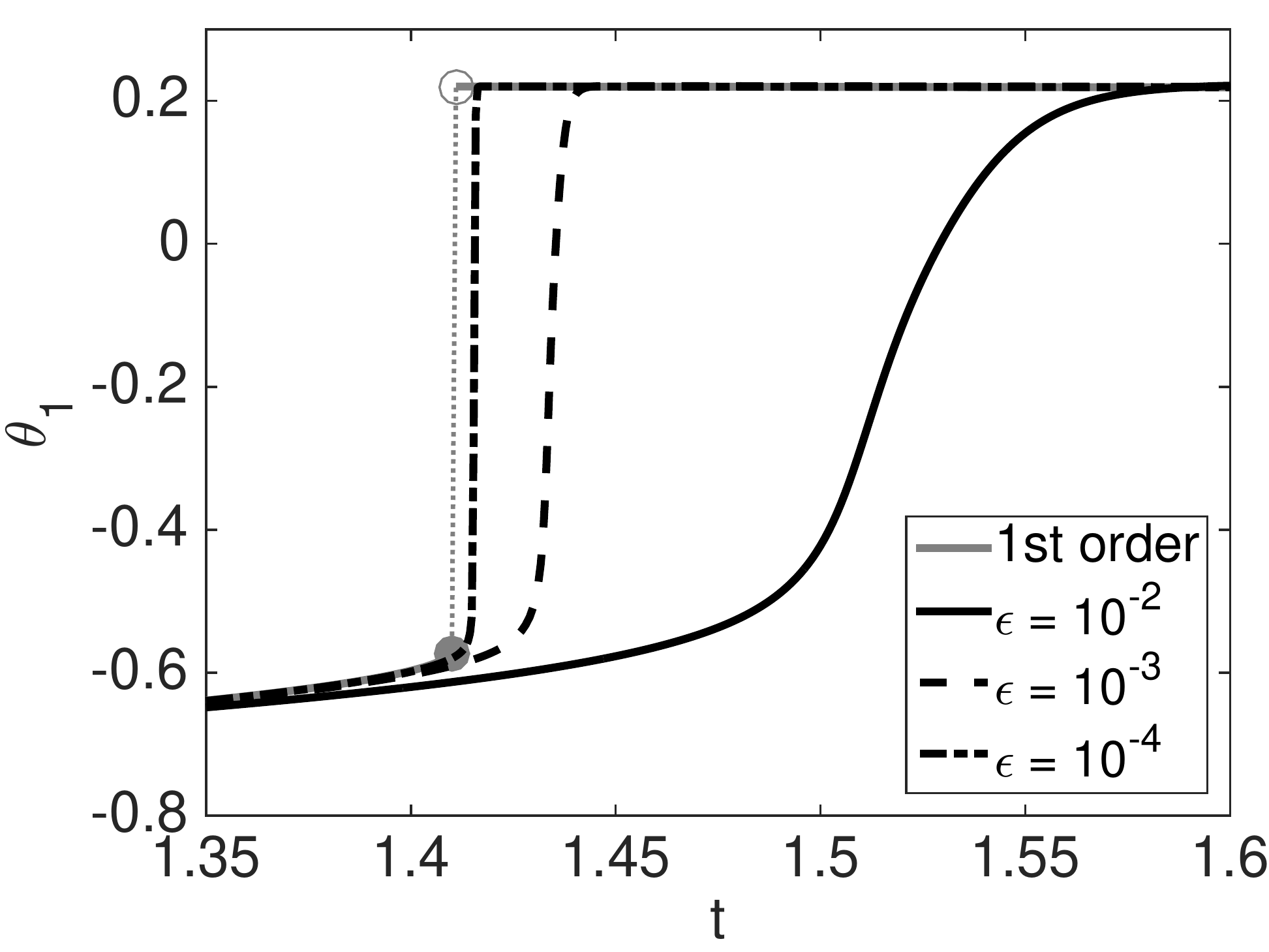}
$~~~$
\includegraphics[height=0.35\textwidth]{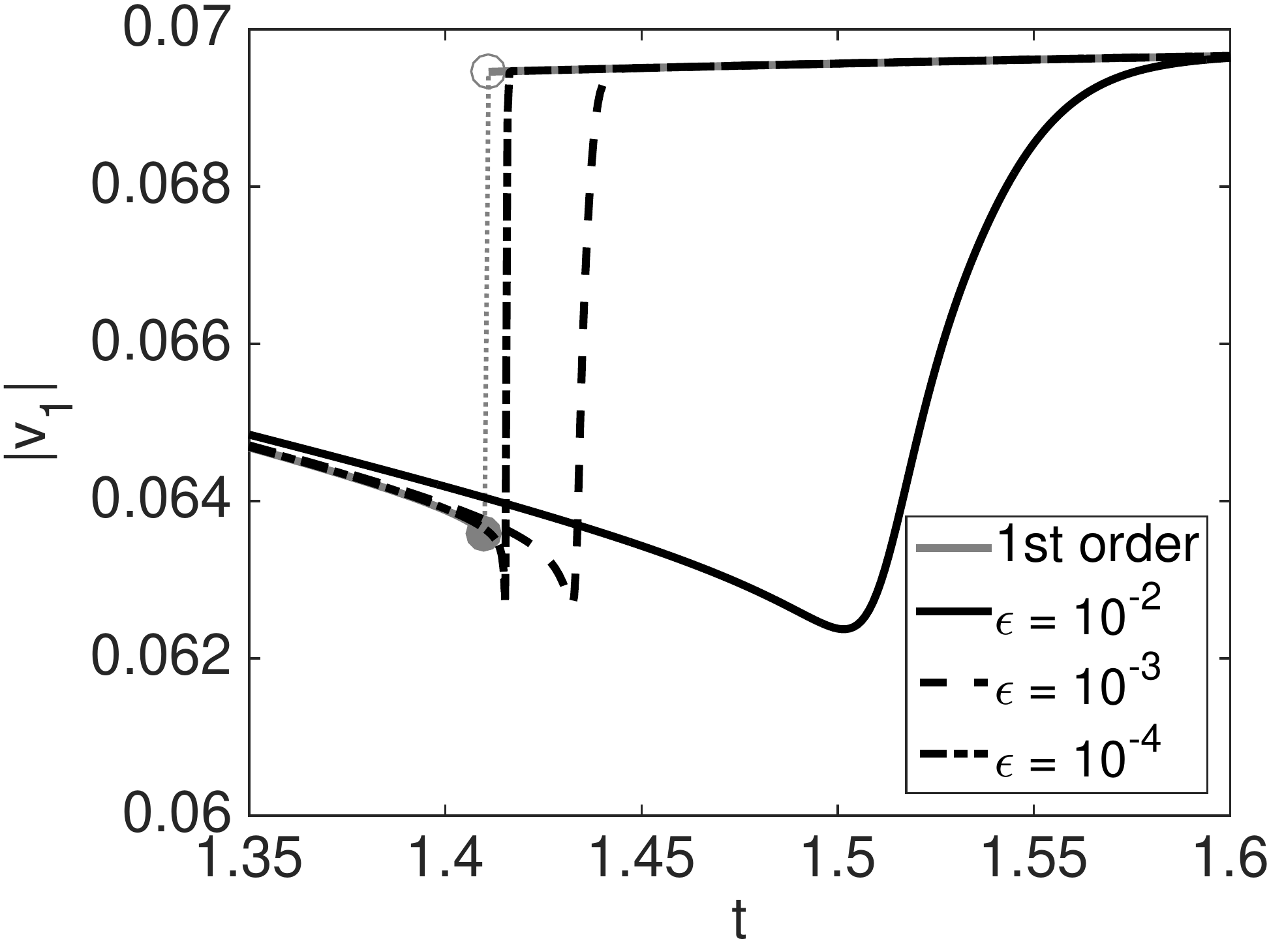}\\
(b) \hspace{0.45\textwidth} (c)
\caption{Time evolution of four particles. (a) The solid line in the main plot represents the solution of the anisotropic first-order model \eqref{eqn:fo-bz} starting from the filled circles. Breakdown occurs in the velocity of particle 1, which is the top left particle. The square indicates the position at which breakdown happens. To extend the evolution after breakdown, we enforce a jump in $\theta_1$, in accordance with solutions of the relaxation model \eqref{eqn:so}. Insert: Zoomed image near the breakdown time of model \eqref{eqn:fo-bz}. On top of the trajectory from the main plot (solid grey) we graph the solution of the relaxation model \eqref{eqn:so} for three values of $\eps$: $\eps=10^{-2},10^{-3}$, and $10^{-4}$. Note how the $\eps$-model \eqref{eqn:so} captures the discontinuities in velocity, as well as approximates solutions of \eqref{eqn:fo-bz} away from the jumps. (b) and (c) Evolution of the angle $\theta_1$ and the magnitude of the velocity $|v_1|$, respectively, around the breakdown time: solutions of \eqref{eqn:fo-bz} and of \eqref{eqn:so} for $\eps=10^{-2},10^{-3}$, and $10^{-4}$. The second-order model captures the sharp transitions in $\theta_1$ and $|v_1|$.}%
\label{fig:trajectories zoom r theta}%
\end{figure}

Figure \ref{fig:trajectories zoom r theta}(b) illustrates this behaviour; in this numerical simulation  \eqref{eqn:so} has been initialized with the same phase space configuration as the initial data for the first-order model. While before breakdown it remains within approximation error to the solution of  \eqref{eqn:fo-bz}, the solution to \eqref{eqn:so} steepens upon approaching $t^\ast$ and approaches, via a fast layer, a new root of \eqref{eqn:vi-bz}. Specifically, through a steep time layer, the direction of particle 1 transitions from $\thetas_1 \approx -0.57$, the double root of $H_1$ (filled circle), to $\thetat_1 \approx 0.22$, the root of $H_1$ indicated by an open circle -- see Figure \ref{fig:HR}. Once a jump selection has been identified, the first-order model can be reinitialized in the new direction $\thetat_1$ and hence, its time evolution can be continued through a breakdown.

In the present paper we provide a rigorous framework and proof for the jump selection process that occurs through the relaxation model.  Using polar coordinates for $v_i$, system \eqref{eqn:so}  reduces to:
\begin{subequations}
\label{eqn:so-polar}
\begin{align}
    \frac{\dx_i}{\dt} &= r_i (\cos \theta_i, \sin \theta_i)^T, \\
    \eps \dfrac{\dtheta_i}{\dt} &= \dfrac1r_i\, H_i(x_1,\dots,x_N,\theta_i),\label{eqn:so-polar b}
  \\
  \eps \dfrac{\dr_i}{\dt} &= -r_i + R_i(x_1,\dots,x_N,\theta_i).\label{eqn:so-polar c}
      \end{align}
\end{subequations}

We first simplify model \eqref{eqn:so-polar} by focusing on the particle $i$ that undergoes a jump in velocity at breakdown. We fix the locations of the other particles $x_j$, with $j \neq i$, and initialize \eqref{eqn:so-polar} at the onset $t^\ast$ of the discontinuity. For convenience of notations we drop the index $i$; the location $x_i$ is now denoted by $\bx = (x,y)$. With these simplifications, system \eqref{eqn:so-polar} initialized at breakdown reads:
\begin{equation}\label{eqn:syst x polar 2D}
\left\{
  \begin{array}{l}
  	\dfrac{\rm d}{\dt}\left(\begin{array}{c} x\\y \end{array}\right) = r \, \left(\begin{array}{c} \cos\theta\\\sin\theta \end{array}\right) ,\\
  \\
   \eps \dfrac{\dtheta}{\dt} = \dfrac1r\, H(x,y,\theta), \\
  \\
    \eps  \dfrac{\dr}{\dt} = -r + R(x,y,\theta) \,,
	\\ \\
	\bx(0)=\bxs, \theta(0)=\thetas, r(0)=\rs.
  \end{array}
\right.
\end{equation}

A typical profile of $H(\bxs,\theta)$ is shown in Figure \ref{fig:HR} (dash-dotted line). Most relevant for the analysis, it has a double root at $\thetas$ and another root $\thetat$ such that $H(\bxs,\theta)>0$ for $\theta \in (\thetas,\thetat)$. As suggested by numerics, and as proved in Theorem \ref{thm: conv quad lin 2D R neg} below, the solution to \eqref{eqn:syst x polar 2D} converges, within a transition layer that scales with $\eps^{2/3}$, to $(\bxs,\thetat,\rt)$, with $\rt = R(\bxs,\thetat)$. Indeed, initiated at $t^\ast$, the profile of $H$ ``rises" and the solution $\theta(t)$ (see in particular the equation for $\theta$ in \eqref{eqn:syst x polar 2D}) is expected to escape, through a bottleneck, to $\thetat$. As seen in the proof of Theorem \ref{thm: conv quad lin 2D R neg}, making these ideas precise require a lot of care, in particular since one has to rule out the possibility of a strong bottleneck effect, in which the solution would be trapped near $\thetas$ long enough for the profile $H$ to ``lower" and gain new roots.

To set the main ideas used in the proof, we start by discussing a toy problem in one dimension. Then, we extend the toy problem to a one-dimensional model with a more general right-hand side. Since proving the desired convergence result for this general 1D problem is quite involved, we choose to simply list the main results and present instead a full proof for the main model in 2D. In the Appendix we revisit the 1D problem and discuss how a proof in one dimension can be obtained for very general right-hand sides. The major merit of the result in one dimension (Theorem~\ref{thm: conv k ell 1D}) is that it constitutes an extension of the Tikhonov's theorem \cite{Tikhonov1952,Vasileva1963} in nonsmooth/discontinuous settings. 


\section{One-dimensional problem}
\label{sec:oned}

\subsection{One-dimensional toy problem.}\label{sec:toy} In this section we use the following simple initial value problem in one dimension with a fixed $\eps$ as an illustration of the main idea of our analysis for the 2D problem:
\begin{equation} \label{eqn:syst x 1D toy}
\left\{
  \begin{array}{l}
  	\dfrac{\dx}{\dt} = v,\\
  	\\
	\eps \dfrac{\dv}{\dt} = -h(v)\cdot (v-\vs)^2(v-\vt) + (x-\xs)=:\F(x,v),\\
	\\
	x(0)=\xs,  \quad v(0)=\vs,
  \end{array}
\right.
\end{equation}
where the two constants $\vs, \vt$ satisfy $0<\vs<\vt$ and there exists a constant $h_0 > 0$ such that $h(v)> h_0$ for all $v\in[\vs,\vt]$. The right-hand side $\F(x,v)$ evaluated at the initial configuration $\xs$  has a double root at the initial velocity $\vs$ and a simple root at $\vt$. Such profile $\F(\xs,v)$ is qualitatively similar to the plot of $H_1(\theta)$ at the breakdown time $t^*$ in model \eqref{eqn:fo-bz}; see the dash-dotted line in Figure \ref{fig:HR}.

In what follows, we outline three stages that system~\eqref{eqn:syst x 1D toy} undergoes when it makes the transition from $\vs$ to the neighbourhood of $\vt$. These three stages are: formation of the bottleneck, escape from the bottleneck, and convergence to $\vt$. 

\subsubsection*{\underline{Formation of the bottleneck}} 

Since the initial velocity $\vs>0$, it holds that $x(t)>\xs$ for $t>0$ sufficiently small. Consequently,  the profile $\F(x(t),\cdot)$ ``rises", loses its root at $\vs$, and becomes strictly positive near $\vs$. Hence, $v(t)>\vs$ for $t>0$ small, as $\dv/\dt$, immediately after being initialized at $0$, becomes positive and remains so in a short time interval. Note though that in this short initial interval, $\F(x(t),\cdot)$ is very small near $v=\vs$, and a bottleneck situation has been created. 

Define 
\begin{equation}
\label{eqn:eta1d}
\eta(t) := \F(x(t),v(t)) - \F(\xs,v(t)) = x(t)-\xs.
\end{equation} 
By referring to Figure \ref{fig:eta} for an illustration, we note that $\eta$ is the increment associated to the bottleneck. Tracking the evolution of such a quantity is an important idea used later in the analysis. In the toy problem considered here, the evolution of $\eta$ is simply governed by $\deta/\dt=v$.

\begin{figure}[ht]%
\centering
\begin{tikzpicture}[>= latex]
                \begin{scope}[scale=3]
                \pgfmathsetmacro\vt {1};
                \pgfmathsetmacro\vEnd {\vt+0.2};
                \pgfmathsetmacro\vStart {-0.5};
                \pgfmathsetmacro\dv {0.02};
                \pgfmathsetmacro\vStNext {\vStart+\dv};
                \pgfmathsetmacro \up {0.3};
					
					\draw[line width=0.75] (-1,0) -- (2,0) node[right] {$v$};
					\draw[line width=0.75] (0,0.04) -- (0,-0.04) node[below] {$v^*$};
					\draw[line width=0.75] (\vt,0.04) -- (\vt,-0.04)  node[below left] {$\tilde{v}$};
					
					\foreach \v  in {\vStart,\vStNext,...,\vEnd}	{
					\pgfmathsetmacro\vprev {\v-\dv};
					\pgfmathsetmacro\eta {-2*\v*\v*(\v-\vt)};
					\pgfmathsetmacro\etaprev {-2*\vprev*\vprev*(\vprev-\vt)};	
					\draw[line width=1.25] (\vprev,\etaprev) -- (\v,\eta);
					
					\pgfmathsetmacro\etaUp {-2*\v*\v*(\v-\vt)+\up};
					\pgfmathsetmacro\etaUpprev {-2*\vprev*\vprev*(\vprev-\vt)+\up};	
					\draw[line width=1.25] (\vprev,\etaUpprev) -- (\v,\etaUp);
					}
					
					\draw[<->,line width=0.75] (0,0) -- (0,\up) node[pos=.5,right] {$\eta(t)$};
					
					\draw (\vStart,1.2)  node[left] {$\mathcal{F}(x(t),\,\cdot\,)$};
					\draw (\vStart,0.6)  node[left] {$\mathcal{F}(x^*,\,\cdot\,)$};

				\end{scope}
\end{tikzpicture}
\caption{Illustration of $\eta(t)$ defined in \eqref{eqn:eta1d};  $\eta$ quantifies the initial bottleneck that occurs in the dynamics of $v$ -- see \eqref{eqn:syst x 1D toy} and \eqref{eqn:syst x 1D toy-eta}.
}
\label{fig:eta}%
\end{figure}
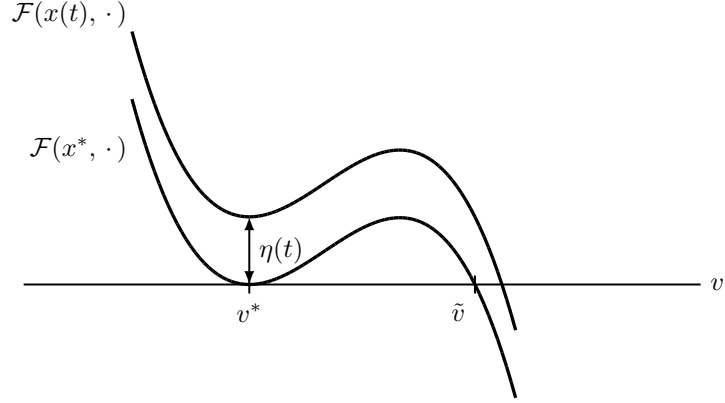

By a change of independent variables $(x, v) \to (\eta, v)$, \eqref{eqn:syst x 1D toy} becomes
\begin{equation} \label{eqn:syst x 1D toy-eta}
\left\{
  \begin{array}{l}
  	\dfrac{\deta}{\dt} = v,\\
  	\\
	\eps \dfrac{\dv}{\dt} = \F(\xs,v) + \eta.\\
	\\
	\eta(0)=0, \quad v(0)=\vs.
  \end{array}
\right.
\end{equation}
Since $v(0)=\vs>0$, there is a (short) initial interval on which $v>0$, hence $\deta/\dt=v>0$ and thus $\eta>0$, since $\eta(0)=0$. Moreover $\F(\xs,v)\geqs0$ for any $v$ in a neighbourhood of $\vs$. So there is a time interval on which
\begin{align*}
   0 \leq \F(\xs, v) \leq c_0 \,,
   \qquad \eta>0 \,, \qquad\text{ and thus}
\qquad
\frac{\dv}{\dt} > 0, 
\end{align*}
where $c_0>0$ is some constant. Consequently, $v(t) \in [\vs, \vt]$ for all $t$ in this interval, and thus
\begin{equation*}
\vs \leq \frac{\deta}{\dt} \leq \vt.
\end{equation*}
Therefore, on this time interval, 
\begin{align*}
    \vs t \leq \eta \leq \vt t \,,
\qquad
   0 \leq v - \vs \lesssim \frac{t + t^2}{\eps} \,,
\qquad
  v(t) \in [\vs, \vt] \,.
\end{align*}
This shows that for a fixed $\eps$,
\begin{align*}
   \F(\xs, v) = - h(v) (v - \vs)^2 (v - \vt)
= \frac{1}{\eps^2}\BigO(t^2) 
\ll \eta(t)
\qquad
\text{near $t = 0$.}
\end{align*}
Consequently, immediately after initialization, $\eta$ is the dominant term in the $v$-equation. As time evolves however, we note that 
\begin{equation} \label{eqn:lower bnd v-vs}
\eps\,\dfrac{\dv}{\dt} \geqs \eta \geqs  \vs t
\qquad\Rightarrow\qquad
v(t)-\vs \geqs\frac{\vs}{2\eps} t^2,
\end{equation}
and hence, $\F(\xs, v) \sim (v-\vs)^2$ may grow to be comparable to $\eta$ for $t$ large enough. We define the end time of the bottleneck as the moment when
\begin{align*}
   \vpran{t^2/\eps}^2 \sim t \,,
\end{align*}
which gives an interval of size $\BigO(\eps^{2/3})$. Therefore, by this definition of the end time, the bottleneck extends for a time period of order $\BigO(\eps^{2/3})$.


\subsubsection*{\underline{Escaping the bottleneck}} Denote the end of the bottleneck time as $t_0$. Then by~\eqref{eqn:lower bnd v-vs} the velocity satisfies
\begin{align*}
    v(t_0) - \vs \gtrsim \frac{\vs}{2 \eps} \cdot \eps^{4/3}
    = \frac{\vs}{2} \eps^{1/3} \,.
\end{align*}
Starting with such initial data and going beyond the bottleneck interval of $\BigO(\eps^{2/3})$, the evolution of $v$ is mainly driven by the quadratic nonlinearity $\F(\xs,v) \sim (v-\vs)^2$ such that 
\begin{equation*}
\eps \dfrac{\dv}{\dt} \gtrsim (v-\vs)^2,
\end{equation*}
where we used that $\eta>0$ for as long as $v$ increases. Solving this differential inequality gives 
\begin{align*}  
   v(t) - \vs 
\gtrsim
   \dfrac{\eps^{1/3}}{\dfrac{2}{\vs} - \dfrac{t - t_0}{\eps^{2/3}}} \,.
\end{align*}
The dynamics above then guarantees that $v$ moves an $\Ord(1)$ distance away from $\vs$ within another $\BigO(\eps^{2/3}) - \BigO(\eps)$ time interval, thus escaping the bottleneck. 


\subsubsection*{\underline{Converging to $\vt$}} In the third stage where $v$ starts at an $\BigO(1)$ distance from $\vs$, the dominant term in the $v$-equation becomes $\F \sim -(v - \vt)$. The main dynamics is thus driven by 
\begin{align*}
    \eps \frac{\dv}{\dt} \gtrsim - (v - \vt),
\end{align*}
with the initial data less than $\vt$. Therefore, within a time interval of order (arbitrarily close to) $\BigO(\eps)$ the velocity $v$ is attracted to $\vt$. Note that through the combined three stages, the total time interval is of size $\Ord\bigl(\eps^{2/3}\bigr)$. Hence $x(t)$ has moved an $\Ord\bigl(\eps^{2/3}\bigr)$ distance from $\xs$. Consequently, the root of $\F(x,\cdot)$ that is being reached asymptotically is located within $\Ord\bigl(\eps^{2/3}\bigr)$ distance from $\vt$. 
 To conclude, within a time interval of size $\Ord\bigl(\eps^{2/3}\bigr)$,  solutions to \eqref{eqn:syst x 1D toy-eta} make the transition from $v=\vs$ to an $\Ord\bigl(\eps^{2/3}\bigr)$-neighbourhood of $\vt$.  
 
From an asymptotic/formal analysis point of view, the arguments above are fairly standard. In fact, a modified version of \eqref{eqn:syst x 1D toy-eta}  is textbook material \cite[Section 6.4]{Holmes_book}; the problem studied in \cite{Holmes_book} is 
 \begin{align*}
     \eps \dfrac{\dv}{\dt} = (v-\vs)^2(v-\vt) + t \,,
\end{align*} 
where the term $t$ is the driving mechanism behind the loss of the root at $v=\vs$. By the methods developed in this paper, such a formal 
argument can be made rigorous however.


\subsection{One-dimensional problem with general right-hand side.} In this part we present an extension of the toy problem \eqref{eqn:syst x 1D toy-eta} that allows for general right-hand sides in the $v$-equation. 
As mentioned before, this result in one dimension 
extends Tikhonov's theorem \cite{Tikhonov1952,Vasileva1963} in nonsmooth/discontinuous settings. 

%
The extended model has the form
\begin{equation}\label{eqn:syst x 1D general}
\left\{
  \begin{array}{l}
  	\dfrac{\dx}{\dt} = v,\\
  	\\
	\eps \dfrac{\dv}{\dt} = \F(x,v),\\
	\\
	x(0)=x_0, \quad v(0)=v_0.
  \end{array}
\right.
\end{equation}
The hypotheses on the right-hand side $\F$ are more general than in the toy problem of Section \ref{sec:toy}. We assume that $x_0$ is ``close" to some $\xs$ for which the function $\F(\xs,\cdot)$ resembles the dash-dotted line in Figure \ref{fig:eta}. However, instead of a \textit{double} root, we assume that $\F(\xs,\cdot)$ has a root at $\vs$ of \textit{even} multiplicity $2k$, for arbitrary $k\geqs1$. In particular, this implies that $\F(\xs,\cdot)$ does not cross the horizontal axis at $\vs$. Moreover, we assume that $\F(\xs,\cdot)$ has another root at $\vt>\vs$, where $\F(\xs,\cdot)$ \textit{does} cross the horizontal axis. Thus, this root is of \textit{odd} multiplicity $2\ell-1$, for some $\ell\geq1$, generalizing the \textit{single} root appearing in \eqref{eqn:syst x 1D toy}.

These properties of $\F(\cdot,\cdot)$ are made precise in the following assumption. Recall that $\xs,\, \vs,\,\vt\in\R$ are given, with $\vs<\vt$. 
%
%
\begin{assumption}\label{ass:F 1d k ell}
Let the integers $k, \ell \geq 1$ be given. Assume that $\F\in C^{\max\{2k,2\ell-1\}}(\R^2;\R)$ is such that
\begin{itemize}
\item $\partial^j_v\F(\xs,\vs)=0$ for all $j\in\{0,\ldots,2k-1\}$, \; $\partial^{2k}_v\F(\xs,\vs)>0$, 
\smallskip

\item $\partial^{j}_v\F(\xs,\vt)=0$ for all $j\in\{0,\ldots,2(\ell-1)\}$, \;  $\partial^{2\ell-1}_v\F(\xs,\vt)<0$, 
\smallskip

\item $\F(\xs,u)>0$ for all $u\in(\vs,\vt)$.
\end{itemize}
\end{assumption}

\noindent Assumption~\ref{ass:F 1d k ell} implies that there is a function $h \in C(\R)$, strictly positive on $[\vs,\vt]$, such that
\begin{equation}\label{eqn:F factors}
   \F(\xs,v) 
   = -h(v)\cdot(v-\vs)^{2k}(v-\vt)^{2\ell-1},
\qquad
   h(v) \geq h_0 > 0 \,,
\end{equation}
where $h_0$ is a constant independent of $\eps$. 

\begin{theorem}\label{thm: conv k ell 1D}
Suppose $\F$ satisfies the conditions in Assumption \ref{ass:F 1d k ell}. Suppose 
\begin{align*}
   \partial_x\F(\xs,\vs)\cdot \vs>0 \,.
\end{align*} 
Let $(x,v)$ be the solution to \eqref{eqn:syst x polar 2D}, with initial conditions $x(0)=x_0$, $v(0)=v_0$. Let $\eps>0$ be given. Suppose there exist constants $a_1,\,a_2>0$ such that
\begin{align*}
   |x_0-\xs|\leqs a_1\,\eps \,,
\qquad
   |v_0 - \vs|\leqs a_2\,\eps \,.
\end{align*} 
Then for all $k\geqs1$ and $\ell\geqs1$ and $\eps>0$ sufficiently small, there are constants $C$, $c$ and $\nu$ (independent of $\eps$) and a time $\tau\leqs C\,\eps^{2k/(4k-1)}$ such that 
\begin{align*}
    |v(\tau)-\vt|\leqs c\,\eps^{\nu} \,,
\end{align*}
where
\begin{align*}
\nu=\frac{2k}{4k-1} \;\;\text{ if }\ell=1,\quad  \text{and} \quad \nu = \frac{2k-1}{(4k-1)(2\ell-1)} \;\;\text{ if $\ell>1$.}
\end{align*}
\end{theorem}
The proof of Theorem~\ref{thm: conv k ell 1D} resembles that for the two-dimensional anisotropic model in the case where $r$ is a constant; see Theorem \ref{thm: conv quad lin 2D R neg}. We only briefly sketch its proof in Appendix~\ref{app:proof thm 1D gen ic}.

The special case that we obtain by choosing $k=\ell=1$ in Theorem \ref{thm: conv k ell 1D} corresponds to the toy problem presented in Section \ref{sec:toy}.

Theorem \ref{thm: conv k ell 1D} extends Tikhonov's theorem  \cite{Tikhonov1952,Vasileva1963} in one spatial dimension, to the case where the right-hand side $\F$ has a root (at $v=\vs$) that gets lost. Note that such situation might come into existence dynamically: we may have a time interval of smooth evolution before we encounter root loss -- see Section \ref{sec:intro fo}, where we discussed this phenomenon in two dimensions. On this initial ``smooth" interval before root loss, Tikhonov's theorem makes sure that the solutions of \eqref{eqn:syst x 1D general} and its first-order counterpart are $\Ord(\eps)$ close (possibly except for some initial layer). This is exactly why we allow for $\Ord(\eps)$ variations in the initial conditions of Theorem \ref{thm: conv k ell 1D}: it is as if we reinitialize our system at the time of root loss, while by Tikhonov's theorem we may expect the deviations from $\xs$ and $\vs$ to be at most $\Ord(\eps)$.

\section{Two-dimensional relaxation model}
\label{sec:2D-varr}

In this section we investigate the anisotropic relaxation model \eqref{eqn:so-polar}. We will first treat Equation~\eqref{eqn:syst x polar 2D} where only one particle is considered moving, while the general model \eqref{eqn:so-polar} for the evolution of $N$ particles is addressed in Remark~\ref{rem:N moving}.

We first consider model~\eqref{eqn:syst x polar 2D} for a single particle. In general, we do not expect the system to start exactly at the breakdown state as in~\eqref{eqn:syst x polar 2D}. Instead, we relax the initial condition in~\eqref{eqn:syst x polar 2D} and assume that the initial data is close to $(\bxs, \thetas, \rs)$. More specifically, let $(\bx_0, \theta_0, r_0)$ be the initial data that such that 
\begin{align} \label{cond:initial}
   |\bxo-\bxs|\leqs a_1\,\eps,
\qquad
   |\thetao-\thetas|\leqs a_2\,\eps,
\qquad
   |\ro-\rs|\leqs a_3\,\eps,
\end{align} 
for some constants $a_1, a_2, a_3 > 0$. 

Denote 
\begin{align} \label{def:F}
   \F(\bx, \theta, r) := \frac{1}{r} H(\bx, \theta),
\end{align}
which represents the right-hand-side of the equation for $\theta$ in \eqref{eqn:syst x polar 2D}. Note that we use the same notation $\F$ as in the one-dimensional problem (see \eqref{eqn:syst x 1D toy}), to better parallel with the considerations made in Section \ref{sec:oned}. There should be no danger of confusion, in particular since in two dimensions $\F$ is a function of four scalar variables (as opposed to a function of two variables in one dimension).

Also, as an analogue of~\eqref{eqn:eta1d}, and preserving the same notation symbol,  let $\eta=\eta(t)$ be given by
\begin{equation}
\begin{aligned} \label{eqn:def eta 2d general}
   \eta(t)
   := \F(\bx(t),\theta(t),r(t))-\F(\bxs,\theta(t),r(t))
   = \frac{1}{r(t)} H(\bx(t),\theta(t)) - \frac{1}{r(t)} H(\bxs,\theta(t)).
\end{aligned}
\end{equation}
As in Section \ref{sec:oned} (see for instance Figure \ref{fig:eta}), $\eta$ is used to quantify the bottleneck.
Now reformulate~\eqref{eqn:syst x polar 2D} as
\begin{equation} \label{eqn:syst eta 2D general}
\left\{
  \begin{array}{l}
  	\!\dfrac{\deta}{\dt} 
	= \overbrace{\nabla_\bx\F(\bx,\theta,r)\cdot \bv}^{=:A} 
	  +\! \overbrace{\vpran{\partial_\theta\F(\bx,\theta,r)-\partial_\theta\F(\bxs,\theta,r)} \dfrac{\dtheta}{\dt}}^{=:B} 
\!+ \overbrace{\vpran{\partial_r\F(\bx,\theta,r)-\partial_r\F(\bxs,\theta,r)}\dfrac{\dr}{\dt}}^{=:C}, \\
  	\\
	\eps \dfrac{\dtheta}{\dt} = \dfrac1r H(\bxs,\theta)+\eta,\\ 
	\\
  \eps \dfrac{d r}{dt} = -r + R(\bx,\theta),\\
	\\
	(\eta(0), \theta(0), r(0)) = (\eta_0, \theta_0, r_0) \,,
  \end{array}
\right.
\end{equation}
where we have denoted
\begin{align*} \label{def:eta-0}
   \eta_0 = \F(\bx_0,\theta_0,r_0)-\F(\bxs,\theta_0,r_0) \,.
\end{align*}
Note that the spatial variable $\bx$ appears explicitly in the dynamics for $\eta$;  its evolution is governed by ${\rm d}\bx/\dt=\bv=r\,(\cos\theta,\sin\theta)^T$.

The main assumptions for $H$ are
\begin{assumption}\label{ass:H 2d quad lin}
Assume that $H\in C^2(\R^2\times\R;\R)$. Suppose there exist two constants $\thetat > \thetas > 0$ such that $H$ satisfies
\begin{itemize}
\item $H(\bxs,\thetat)=0$, \quad $H(\bxs,\thetas)=\partial_\theta H(\bxs,\thetas)=0$,
\smallskip

\item $\partial^2_\theta H(\bxs,\thetas)>0$, \quad
$\partial_\theta H(\bxs,\thetat)<0$,
\smallskip

\item $H(\bxs,\phi)>0$ for all $\phi\in(\thetas,\thetat)$.
\end{itemize}
\end{assumption}
\noindent Assumption~\ref{ass:H 2d quad lin} guarantees that there is a function $h \in C(\R)$ strictly positive on $[\thetas,\thetat]$ such that
\begin{equation} \label{eqn:H poly theta}
   H(\bxs,\phi) = -h(\phi)\cdot(\phi-\thetas)^2(\phi-\thetat) \,,
\qquad
   h(\phi) \geq h_0 > 0 \,,
\qquad
   \phi \in [\thetas, \thetat] 
\end{equation}
for some 
constant $h_0 > 0$. The motivation for such an assumption on $H$ was presented in Section \ref{sec:intro fo} -- see for example the dash-dotted line in Figure \ref{fig:HR}.

The assumptions we make on the function $R$ are very unrestrictive -- see the statement of Theorem \ref{thm: conv quad lin 2D R neg}. In particular, we assume that it is a bounded function and also, that $R(\bxs,\cdot)$ is strictly positive at $\thetas$ and $\thetat$.

Our main goal for this paper is to justify the jump discontinuity of the velocity of the particle as $\eps \to 0$. Although the velocity now has two components $r$ and $\theta$, 
the main dynamics of~\eqref{eqn:syst eta 2D general} is driven by the $\theta$-equation. Hence, we first show the jump of $\theta$ from $\thetas$ to $\thetat$ within a time interval that vanishes with $\Eps$ and then show the jump of $r$ induced by the jump of $\theta$. 

To show the transition of $\theta$ we follow a similar approach as for the toy problem~\eqref{eqn:syst x 1D toy}. Taken into account the extra terms in system~\eqref{eqn:syst eta 2D general} compared to~\eqref{eqn:syst x 1D toy}, we will divide the whole transition process into four stages:
\begin{itemize}
\item Interval I: Within this interval, a bottleneck will form for $\theta$. At the end of the interval, we have that $\theta - \thetas = \BigO(\Eps^{1/3})$. The term $A$ is strictly positive and dominates $B$. The main driving force in the dynamics for $\theta$ is $\eta$.

\item Interval II: The second interval is a preparation for escaping the bottleneck. More specifically, at the end of the second interval, $\theta$ will reach the state where $\theta - \thetas = \BigO(\Eps^{1/6})$. During this interval $A + B \geq 0$. 
The main driving force for $\theta$ is $\frac{1}{r} H\sim (\theta-\thetas)^2$.

\item Interval III: At the end of the third interval, $\theta$ will escape from the bottleneck, because the initial state prepared by Interval II is large enough. Hence, at the end of Interval III we will have $\theta - \thetas = \BigO(1)$. 

\item Interval IV: In this last time interval, the main driving force for $\theta$ is $\frac{1}{r} H\sim -(\theta-\thetat)$. Consequently, $\theta$ is attracted to the stable equilibrium point $\thetat$.
\end{itemize}
\smallskip
The characteristics of the four intervals are represented schematically in Figure \ref{fig:intervals}. There we also add the orders of magnitudes of the intervals themselves. 

\begin{figure}[t]%
\centering
\begin{tikzpicture}[>= latex]
                \begin{scope}[scale=1]

					\draw[line width=1.25] (0,0) -- (13,0);
					
					\draw[line width=1.25] (0,-0.1) -- (0,0.1) node[above] {$0$};
					\draw[line width=1.25] (4,-0.1) -- (4,0.1) node[above] {$\tau_1$};
					\draw[line width=1.25] (8,-0.1) -- (8,0.1) node[above] {$\tau_2$};
					\draw[line width=1.25] (11,-0.1) -- (11,0.1) node[above] {$\tau_3$};
					\draw[line width=1.25] (13,-0.1) -- (13,0.1) node[above] {$\tau_4$};
				
					\begin{scope} 
					\tikzstyle{every node} = [draw,rectangle, fill=gray!5, line width = 0.5]
					\draw (2,1.5)  node[below] {$I$};
					\draw (6,1.5)  node[below] {$II$};
					\draw (9.5,1.5)  node[below] {$III$};
					\draw (12,1.5)  node[below] {$IV$};
					\end{scope}

            \draw[decorate, decoration={brace,mirror,amplitude=10pt}, xshift=0pt, yshift=-5pt](0.1,0) -- (3.9,0) node[pos=.5,sloped, below, yshift=-10pt] {$\Ord(\eps^{2/3})$};
            \draw[decorate, decoration={brace,mirror,amplitude=10pt}, xshift=0pt, yshift=-5pt](4.1,0) -- (7.9,0) node[pos=.5,sloped, below, yshift=-10pt] {$\Ord(\eps^{2/3})-\Ord(\eps^{5/6})$};		
			 \draw[decorate, decoration={brace,mirror,amplitude=10pt}, xshift=0pt, yshift=-5pt](8.1,0) -- (10.9,0) node[pos=.5,sloped, below, yshift=-10pt] {$\Ord(\eps^{5/6})-\Ord(\eps)$};	
			\draw[decorate, decoration={brace,mirror,amplitude=10pt}, xshift=0pt, yshift=-5pt](11.1,0) -- (12.9,0) node[pos=.5,sloped, below, yshift=-10pt] {$\Ord(\eps^{1-\lambda})$};

\draw[->,line width=1, color=gray!70] (0,-0.2) -- (0,-2) node[below,color=black] {$\theta-\thetas=\Ord(\eps)$};
\draw[->,line width=1, color=gray!70] (4,-0.2) -- (4,-2) node[below,color=black] {$\theta-\thetas=\Ord(\eps^{1/3})$};
\draw[->,line width=1, color=gray!70] (8,-0.2) -- (8,-2) node[below,color=black] {$\theta-\thetas=\Ord(\eps^{1/6})$};
\draw[->,line width=1, color=gray!70] (11,-0.2) -- (11,-3) node[below,color=black] {\begin{minipage}{0.25\textwidth}\begin{center}
$\theta-\thetas=\Ord(1)$\\
$\theta-\thetat=\Ord(1)$
\end{center}
\end{minipage}};
\draw[->,line width=1, color=gray!70] (13,-0.2) -- (13,-2) node[below,color=black] {$\theta-\thetat=\Ord(\eps^{2/3})$};
\end{scope}
\end{tikzpicture}
\caption{Orders of magnitude of the four time intervals that constitute the transition layer identified in Theorem \ref{thm: conv quad lin 2D R neg}. The total length of the four intervals is $\Ord(\eps^{2/3})$. Note that in Interval IV, the parameter $\lambda>0$ is arbitrarily small. The evolution of $\theta$ is also indicated. It moves from an $\eps$-neighbourhood of $\thetas$ at time $t=0$ to an $\eps^{2/3}$-neighbourhood of $\thetat$ at time $t=\tEndIV$.} %
\label{fig:intervals}%
\end{figure}
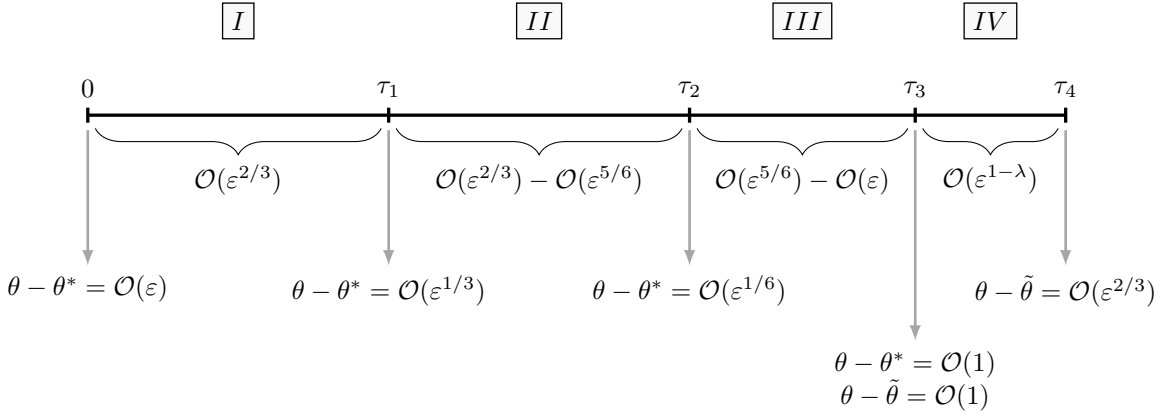

To account for possibly different characteristics of the profile $R(\xs,\cdot)$ between $\thetas$ and $\thetat$, we distinguish between the following two cases:
\begin{enumerate} 
\item[\RP] $R(\bxs,\phi)>0$ for all $\phi\in[\thetas,\thetat]$. The scenario in Figure \ref{fig:HR} falls into this category; the grey line is strictly positive. In this case, define 
\begin{align} \label{def:R-0-RP}
   R_0:=\dfrac12\,\min_{\phi\in[\thetas,\thetat]}R(\bxs,\phi)>0 \,.
\end{align}
\item[\RN] There exists $\phi \in [\thetas, \thetat]$ such that $R(\bxs,\phi)=0$. In this case define
\begin{align}
\nn   \thetatwo := \inf\{\phi\in[\thetas,\thetat]& :\, R(\bxs,\phi)=0\},
\qquad
   \thetathree := \sup\{\phi\in[\thetas,\thetat]:\, R(\bxs,\phi)=0\} \,,
\\
  &\thetatwos=(\thetas+\thetatwo)/2 \,,
\qquad
  \thetathrees=(\thetathree+\thetat)/2 \,. \label{eqn:def theta2s theta3s}
\end{align}
If $\thetatwo\neq\thetathree$, then $R(\xs,\cdot)$ may become negative in the intermediate interval and may even have multiple subsequent positive and negative intervals. The first zero of $R(\xs,\cdot)$ is at $\thetatwo$, the last zero is at $\thetathree$. Also note that $\thetatwos$ and $\thetathrees$ do not depend on $\eps$; these are quantities that are $\Ord(1)$ away from the zeroes of $R(\xs,\cdot)$ -- see Figure \ref{fig:R neg} for an illustration. 
We also define\footnote{Since cases {\RP } and {\RN } are mutually disjoint, using the same notation $R_0$ in \eqref{def:R-0-RP} and \eqref{def:R-0-RN} is consistent, and in fact very convenient for the presentation.
} 
\begin{equation}  \label{def:R-0-RN}
R_0 := \dfrac12\,\min_{\phi \in [\thetas,\thetatwos]\cup[\thetathrees,\thetat]} R(\bxs,\phi)>0.
\end{equation}
\end{enumerate}

\begin{figure}[t]%
\centering
\begin{tikzpicture}[>= latex]
\begin{scope}[scale=1]
					
\draw[line width=0.75] (-4.5,0) -- (6,0) ;
\draw[line width=0.75] (-3,0.1) -- (-3,-0.1) node[below] {$\thetas$};
\draw[line width=0.75] (-1,-0.1) -- (-1,0.1) node[above] {$\thetatwo$};
\draw[line width=0.75] (2,-0.1) -- (2,0.1) node[above] {$\thetathree$};
\draw[line width=0.75] (5,0.1) -- (5,-0.1)  node[below] {$\thetat$};
\draw[line width=0.75] (-2,0.1) -- (-2,-0.1)  node[below] {$\thetatwos$};
\draw[line width=0.75] (3.5,0.1) -- (3.5,-0.1)  node[below] {$\thetathrees$};

\draw[line width=1.25] (-4,2.5) to [out=0, in=135] (-1,0) to [out=-45, in=-115] (0,0) 
to [out=65,  in=120] %
(1.2,0) to [out=-60, in=210] (2,0) to [out=30, in=190] (5.5,1.9);

\draw (-4,2.5)  node[left] {$R(\mathbf{x}^*,\,\cdot\,)$};
\end{scope}
\end{tikzpicture}
\caption{Schematic illustration of $R(\bxs,\cdot)$ in the case \RN. The function $R(\bxs,\cdot)$ is strictly positive at $\thetas$ and $\thetat$. The first and last zero of $R(\bxs,\cdot)$ within the interval $[\thetas,\thetat]$ are denoted by $\thetatwo$ and $\thetathree$, respectively. In general, $\thetatwo=\thetathree$ may happen, when $R(\bxs,\cdot)$ only touches the horizontal axis. If $\thetatwo\neq\thetathree$, then $R(\bxs,\cdot)$ could be negative everywhere in $(\thetatwo,\thetathree)$, or have additional zeros. The picture shows the case in which $R(\bxs,\cdot)$ has multiple zeroes between $\thetatwo$ and $\thetathree$. The point $\thetatwos$ is halfway between $\thetas$ and $\thetatwo$, and likewise the point $\thetathrees$ is halfway between $\thetathree$ and $\thetat$.
} %
\label{fig:R neg}%
\end{figure}

Our main theorem states
\begin{theorem}\label{thm: conv quad lin 2D R neg}
Assume that $H, R \in C^2(\R^2\times\R;\R)$ and $H$ satisfies Assumption~\ref{ass:H 2d quad lin}. Assume moreover that 
\begin{align} \label{def:As-rs}
   \As = \nabla_\bx\F(\bxs,\thetas,\rs)\cdot \bvs>0 \,, 
\qquad 
   \rs=R(\bxs,\thetas)>0 \,,
\qquad
   R(\bxs,\thetat)>0 \,.
\end{align}
Let $(\eta, \theta, r)$ be the solution to~\eqref{eqn:syst eta 2D general} whose initial data satisfy~\eqref{cond:initial}. 
Let there be a constant $\bar{R}\geqs 0$ large enough such that 
\begin{align} \label{bound:bar-R}
   0 < r_0 \leq \bar R \,,
\quad \text{ and } \quad
   |R(\mathbf{z},\phi)| \leqs \bar{R} 
\quad
  \text{for all $(\mathbf{z},\phi)\in\R^2\times\R$} \,.
\end{align} 
Then, for all $\eps>0$ sufficiently small, there exist three positive constants $\bar{C}_1, \bar{C}_2$ and $c$ (independent of $\eps$) and a time ${\bar C}_1 \Eps^{2/3} \leqs \tau\leqs \bar{C}_2\,\eps^{2/3}$ such that 
\begin{align*}
  |\theta(\tau)-\thetat|\leqs c\,\eps^{2/3} \,.
\end{align*}
\end{theorem}
\begin{proof}
The proof is divided into four steps according to the four intervals described above. The common procedures within each interval are:
\begin{quote}
\hspace{-1cm}  \fbox{Estimate of $\abs{\bx - \bxs}, \abs{r - \rs}$, and the bounds for $r$}
$\Rightarrow$ \fbox{Estimate of $\eta$}
$\Rightarrow$ \fbox{Estimate of $\theta - \thetas$}. 
\end{quote}
We only distinguish between the cases {\RP} and {\RN} in the fourth interval. Hence, unless otherwise noted, the two cases are treated simultaneously, and notation $R_0$ can refer to either \eqref{def:R-0-RP} or \eqref{def:R-0-RN}. 
\medskip

\Ni \underline{\it{Step 1: Interval I (Formation of the bottleneck). }}
Compared with setting the initial data exactly at $(\xs, \thetas, \rs)$, the main extra difficulty that the initial perturbation introduces is that $\eta$ may not be positive near $t = 0$. Hence $\theta$ may not be increasing starting from $t=0$. We want to show that after an initial interval of order $\BigO(\Eps^{2/3})$, $\theta - \thetas$ becomes positive and of order $\BigO(\Eps^{1/3})$. 

First we introduce some constants before explicitly defining Interval I. 
Let 
\begin{align}
   \bar{B}:=\bar{B}(\bxs,1)\times\bar{B}(\thetas,1) \label{eqn:def ball B bar}
\end{align} 
and 
\begin{align} \label{eqn:def D8 2d}
   D:= \frac{3}{R_0} \vpran{\max_{\bar{B}}|\partial_\theta R| + 1} \,,
\qquad
   C_0:= \frac{8D}{\As} + \frac{1}{D} \,,
\end{align}
where $R_0$ is defined as in~\eqref{def:R-0-RP} or \eqref{def:R-0-RN} and $\As$ is given by~\eqref{def:As-rs}. Note that $D$ is strictly positive and hence, $C_0$ is well-defined. 
Let
\begin{equation} \label{eqn:def sigma1 2d}
    \sigma_1 
:= \sup\{ t \,:\,  \abs{\theta(s)-\thetas} \leqs \eps^{1/3} \,\, \text{for all $s \in [0, t]$}\}.
\end{equation}
Define the end-time $\tau_1$ of Interval I as 
\begin{align} \label{eqn:def tEndI 2d}
    \tEndI := \min\{ \sigma_1, C_0\,\eps^{2/3}\} \,.
\end{align}
Note that since initially we have $\abs{\theta_0 - \thetas} = \BigO(\Eps)$, the length of Interval I is strictly positive as long as $\Eps$ is small enough. 

Our main goal in Step 1 is to show that at the end-time $t = \tau_1$,
\begin{align} \label{goal:interval-1}
   \theta(\tau_1) - \thetas = \eps^{1/3} 
\qquad \text{and} \qquad
   r(t) > 0
\quad \text{for all} \quad
   t \in [0, \tEndI] \,.
\end{align}
We start by showing the following bounds on $[0, \tEndI]$ (recall upper bound $\bar R$ introduced in \eqref{bound:bar-R}):
\begin{align} \label{bound:interval-I-1}
   |\bx(t)-\bxs|
\leqs  D_1\,\eps^{2/3}, 
\qquad
     0 < R_0 \leq r(t) \leq \bar R \,,
\qquad
   |r(t)-\rs|
\leqs 
  D_2\,\eps^{1/3} \,, 
\end{align}
where $D_1, D_2$ are some constants independent of $\Eps$ that will be defined later.  

The upper bound for $r$ can be derived from the $r$-equation. Indeed, the ODE for $r$ in \eqref{eqn:syst eta 2D general} yields the inequality
\begin{equation*} \label{eqn:ODE r diff ineq Rbar}
    \eps\,\dfrac{{\rm d}|r|}{\dt} 
\leqs 
   -|r| + \bar{R} \,.
\end{equation*}
Hence, by the assumption for the initial data $r_0$ in~\eqref{bound:bar-R}, we have
\begin{align} \label{bound:r-all-time}
   \abs{r(t)}\leqs\bar{R} 
\end{align} 
for all $t$ such that the solution exists. 
Therefore on Interval I,
\begin{equation} \label{eqn:bnd x-xs eps int1 2d}
  |\bx(t)-\bxs|
\leqs 
  |\bx(t)-\bxo| + |\bxo-\bxs| \leqs \bar{R}\cdot C_0\,\eps^{2/3} 
  + a_1\,\eps \leqs D_1\,\eps^{2/3} \,,
\end{equation}
where the constant $D_1$ is chosen as $D_1 = 2 \bar R C_0$. By \eqref{eqn:def sigma1 2d} and \eqref{eqn:bnd x-xs eps int1 2d}, we know that $(\bx(t),\theta(t))\in\bar{B} $ for $t\in[0,\tEndI]$ and $\Eps$ sufficiently small; cf.~\eqref{eqn:def ball B bar}. Thus for $\Eps$ small enough,
\begin{equation} \label{bound:R-Rs-perturb}
   |R(\bx(t),\theta(t))-R(\bxs,\thetas)| 
\leqs 
   \max_{\bar{B}}|\nabla_\bx R| \cdot D_1 \,\eps^{2/3} 
   + \max_{\bar{B}}|\partial_\theta R| \cdot \, \eps^{1/3}
\leqs
   D_2 \Eps^{1/3} \,,
\end{equation}
where 
\begin{align} \label{def:D-2}
   D_2 = \max_{\bar{B}}|\partial_\theta R| + 1 \,.
\end{align}
Since $R(\bxs,\thetas)\geqs 2\,R_0$, the bound~\eqref{bound:R-Rs-perturb} implies that for $\eps$ sufficiently small, 
\begin{equation*}
   R(\bx(t),\theta(t))\geqs R_0 \,,
\qquad
   t\in[0,\tEndI] \,.
\end{equation*}
We also have $r_0 \geqs R_0$ since
\begin{align*} 
   |\ro-R(\bxs,\thetas)|=|\ro-\rs|\leqs a_3\,\eps \,.
\end{align*}
Therefore $R_0$  is a subsolution to the $r$-equation which gives
\begin{align*}
    r(t) \geqs R_0 \,,
\qquad
   t\in[0,\tEndI] \,.
\end{align*}

Now we estimate the size of $|r(t)-\rs|$. 
Observe that by~\eqref{eqn:syst eta 2D general} and \eqref{bound:R-Rs-perturb},
\begin{align} \label{bound:r-rs-Inter-I}
   \Eps \frac{\rm d }{\dt} \abs{r - \rs}
\leqs
   - \abs{r - \rs} + |R(\bx(t),\theta(t))-R(\bxs,\thetas)|
\leqs
   - \abs{r - \rs} + D_2 \Eps^{1/3}
\end{align}
with the initial data satisfying
\begin{align*}
   |\ro-\rs|
\leqs 
   a_3\,\eps
\leqs
   D_2 \Eps^{1/3}
\qquad
  \text{for $\Eps$ small enough} \,.
\end{align*}
Therefore $D_2\, \Eps^{1/3}$ is a supersolution to~\eqref{bound:r-rs-Inter-I}, which gives
\begin{equation*} \label{eqn:bnd int1 r-rs 2d}
   |r(t)-\rs| 
\leqs 
   D_2 \Eps^{1/3} \,,
\qquad
   t\in[0,\tEndI] \,.
\end{equation*}
We thus have verified all the bounds in~\eqref{bound:interval-I-1}. 
\smallskip

Next we study the evolution of $\eta$ on Interval I by estimating the three terms $A, B, C$ on the right-hand side of the $\eta$-equation. We will show that on Interval I the term $A$ is dominant. First by the continuity of $A$ in $\bx$, $\theta$ and $r$ together with~\eqref{bound:interval-I-1} and~\eqref{def:As-rs}, we have that 
\begin{align}  \label{bound:Interval-I-A}
    A(t)\geqs \As/2 > 0
\qquad
   \text{for all $t \in [0, \tau_1]$.}
\end{align}
Meanwhile, on Interval I, $H$ and $\eta$ satisfy
\begin{equation} \label{eqn:Taylor H}
\begin{aligned} 
   |H(\bx,\theta)| 
\leqs
  &\, \underbrace{|H(\bxs,\thetas)|}_{=0} 
      + |\nabla_\bx H(\bxs,\thetas)|\cdot |\bx-\bxs| 
      + \underbrace{|\partial_\theta H(\bxs,\thetas)|}_{=0}\cdot |\theta-\thetas|
\\
  &\, + \left(\max_{\alpha_1+\alpha_2+\alpha_3=2}\,
                 \max_{\bar{B}} |\partial^{\alpha_1}_x\partial^{\alpha_2}_y                              
                 \partial^{\alpha_3}_\theta H| \right)
            \cdot 
               \sum_{\alpha_1+\alpha_2+\alpha_3=2} (x-\xs)^{\alpha_1}(y- 
               \ys)^{\alpha_2}(\theta-\thetas)^{\alpha_3}
\end{aligned}
\end{equation}
and
\begin{equation} \label{bound:eta-I-1}
\begin{aligned}
  \abs{\eta(t)}
= \frac{1}{r(t)} 
       \abs{H(\bx(t),\theta(t)) - H(\bxs,\theta(t))}
\leq 
    \frac{1}{R_0} \max_{\bar B} \abs{\nabla_{\bx} H}
    \abs{\bx - \bxs} \,.
\end{aligned}
\end{equation}
Since $|\bx(t)-\bxs|\leqs D_1\,\eps^{2/3}$ and $|\theta(t)-\thetas|\leqs\eps^{1/3}$ on Interval I, it follows from~\eqref{eqn:Taylor H} and~\eqref{bound:eta-I-1} that 
\begin{equation} \label{eqn:Taylor H-1}
  \abs{\eta(t)} \leq D_3 \, \Eps^{2/3} \,,
\qquad
  |H(\bx(t),\theta(t))|\leqs D_4\,\eps^{2/3} \,,
\qquad
  t \in [0, \tau_1] \,,
\end{equation}
and
\begin{equation} \label{eqn:bnd dvH}
   |\partial_\theta H(\bx(t),\theta(t))-\partial_\theta H(\bxs,\theta(t))| \leqs  \max_{\bar{B}}|\nabla_\bx \partial_\theta H|\cdot |\bx(t)-\bxs| \leqs D_5\,\eps^{2/3}
\qquad
  t \in [0, \tau_1] \,,
\end{equation}
where $D_3, D_4, D_5$ are positive constants that are independent of $\Eps$. 
We can then estimate $|B(t)|$ using \eqref{bound:interval-I-1}, \eqref{eqn:Taylor H-1} and \eqref{eqn:bnd dvH} and obtain that
\begin{equation}
\begin{aligned} \label{eqn:bnd B int1 2d}
    |B(t)|
& \leqs
   \abs{\dfrac{1}{r(t)} \, \partial_\theta H(\bx(t),\theta(t))
           -\dfrac{1}{r(t)} \, \partial_\theta H(\bxs,\theta(t))}
  \cdot 
    \dfrac1\eps \, \abs{\dfrac{1}{r(t)} \, H(\bxs,\theta(t)) + \eta(t)}
\\
& \leqs 
    \dfrac{1}{\eps\, R_0^2}\, (D_5\,\eps^{2/3}) \, 
    \vpran{D_4 + R_0 D_3 } \eps^{2/3} 
  \leqs 
    \dfrac{D_5 \vpran{D_4 + R_0 D_3 }}{R_0^2} \eps^{1/3} \,.
\end{aligned}
\end{equation}
Hence, for sufficiently small $\eps$, it holds that $|B(t)|\leqs \As/4$. Together with~\eqref{bound:Interval-I-A} we have 
\begin{align} \label{bound:interval-I-A-B}
    A(t)+B(t)\geqs \As/4 \,,
\qquad
   t \in [0, \tau_1] \,.
\end{align}
Regarding the term $C$, we note that
\begin{align*} 
  C =\bigg(\partial_r\F(\bx,\theta,r)-\partial_r\F(\bxs,\theta,r)\bigg) \dfrac{\dr}{\dt}
    = \bigg(-\dfrac{1}{r^2}\,H(\bx,\theta) + \dfrac{1}{r^2}\, H(\bxs,\theta)\bigg)\dfrac{\dr}{\dt}
     = -\dfrac1r \, \eta \, \dfrac{\dr}{\dt} \,.
\end{align*}
Denote 
\begin{equation} \label{eqn:def G}
G(t):= \dfrac{1}{r(t)}\, \dfrac{\dr}{\dt} \,.
\end{equation}
Then for every $t \in [0, \tau_1]$, by~\eqref{bound:interval-I-1}, \eqref{def:D-2} and the definition of Interval I we have
\begin{align} \label{eqn:upper bnd G 2d}
  |G(t)| 
= &\, \left|\dfrac{1}{r(t)}\, \dfrac{\dr}{\dt} \right| 
= \dfrac1\eps \, \dfrac{1}{r(t)}\cdot \big|-r(t)+R(\bx(t),\theta(t))   
    + \underbrace{\rs - R(\bxs,\thetas)}_{=0}\big|  \nn
\\
\nn \leqs&\, \dfrac{1}{\eps\, R_0}\, \left( |r(t)-\rs| + \max_{\bar{B}}|\nabla_\bx R|\cdot |\bx(t)-\bxs| + \max_{\bar{B}}|\partial_\theta R|\cdot |\theta(t)-\thetas| \right) 
\\
\leq & \, 
  \dfrac{1}{\eps\, R_0}
  \vpran{D_2 \,\Eps^{1/3} + \max_{\bar{B}}|\nabla_\bx R| \cdot D_1\, \Eps^{2/3}
             + D_2\, \Eps^{1/3}}
\leq 
   \frac{3D_2}{R_0} \Eps^{-2/3} \,,
\qquad
  t \in [0, \tau_1] \,.
\end{align}
Combining~\eqref{bound:interval-I-A-B} with~\eqref{eqn:upper bnd G 2d}, we thus have on Interval I,
\begin{align} \label{eqn:diff ineq eta with G int1 2d}
    \dfrac{\deta}{\dt}
\geqs \dfrac14 \As - G\cdot \eta
\geqs \dfrac14 \As - \frac{3D_2}{R_0}\eps^{-2/3} \eta \,.
\end{align}
Meanwhile, the initial data $\eta(0)$ satisfies
\begin{equation}
   |\eta(0)| 
= \left| \dfrac{1}{\ro}\,H(\bxo,\thetao) - \dfrac{1}{\ro}\,H(\bxs,\thetao) \right| 
\leqs 
   \dfrac{1}{\ro}\,\max_{\bx \in \bar{B}(\bxs,1)}|\nabla_\bx H(\bx,\thetao)| \cdot \underbrace{|\bxo-\bxs|}_{\leqs a_1\,\eps} 
\leqs 
    D_6 \,\eps,\label{eqn:bnd eta0 int1 2d gen IC}
\end{equation}
where $D_6 = \frac{a_1}{R_0} \max_{\bar{B}}|\nabla_\bx H|$.  

Recall the definition of $D$ in~\eqref{eqn:def D8 2d} and solve~\eqref{eqn:diff ineq eta with G int1 2d} together with~\eqref{eqn:bnd eta0 int1 2d gen IC}. This yields that
\begin{equation} \label{eqn:lower bnd eta 2d int1 gen IC}
    \eta(t) 
\geqs 
   -D_6\,\eps + \dfrac{\As}{4\,L(\eps)}\,\bigg(1-\exp\big(-L(\eps)\,t\big)\bigg) \,,
\qquad
   L(\Eps) = D \, \Eps^{-2/3} \,.
\end{equation}
%
%
%
%
Note that, by Assumption~\ref{ass:H 2d quad lin}, we have that $H(\bxs,\theta(t))\geqs0$ on Interval I, for $\eps$ sufficiently small. 
Hence
\begin{equation} \label{eqn:diff ineq theta int1 2d}
   \eps\,\dfrac{\dtheta}{\dt} 
   = \eta + \dfrac1r \, H(\bxs,\theta) \geqs \eta \,.
\end{equation}
thus it follows from \eqref{eqn:lower bnd eta 2d int1 gen IC} that
\begin{equation} \label{eqn:lower bnd theta-thetas gen IC}
     \theta(t)-\thetas 
\geqs 
    \underbrace{\thetao-\thetas}_{\geqs-a_2\,\eps} 
    + \left(\dfrac{\As}{4\,\eps\,L(\eps)} - D_6 \right) t 
    + \dfrac{\As}{4\,\eps\,(L(\eps))^2} \bigg(\exp(-L(\eps)\,t)-1\bigg).
\end{equation}
At the end of Interval I, either of the following two cases is true:
\begin{description}
\item[Case I-1] the end of Interval I is reached when $\theta(\tEndI)-\thetas=\eps^{1/3}$. By \eqref{eqn:def sigma1 2d}, we know that the final time must satisfy $\tEndI\leqs C_0\,\eps^{2/3}$.
\item[Case I-2] the end of Interval I is reached when $\theta(\tEndI)-\thetas=-\eps^{1/3}$. Substituting the inequality $\exp(-z)\geqs1-z$ in \eqref{eqn:lower bnd theta-thetas gen IC}, we then find that
\begin{equation*}
-\eps^{1/3} = \theta(\tEndI)-\thetas \geqs -a_2\,\eps - D_6\,\tEndI
\end{equation*}
must hold in this case. This implies, for $\eps$ sufficiently small, that
\begin{equation*} \label{eqn:lower bnd t int 1 2d gen IC}
    \tEndI
\geqs 
   \dfrac{1}{2 D_6}\,\eps^{1/3},
\end{equation*}
which contradicts the assumption that $\tEndI\leqs C_0\,\eps^{2/3}$. Therefore at the end of Interval I it can \emph{not} happen that $\theta(\tEndI)-\thetas=-\eps^{1/3}$.
\item[Case I-3] the end of Interval I is reached when $\tEndI=C_0\,\eps^{2/3}$. 
Note that, by the definition of $L(\Eps)$ in~\eqref{eqn:lower bnd eta 2d int1 gen IC}, we have
\begin{equation*}
   \dfrac{\As}{4\,\eps\,L(\eps)} =  \dfrac{\As}{4 \, D\,\eps^{1/3}}  \,,    
\qquad
   \dfrac{\As}{4\,\eps\,(L(\eps))^2} =  \dfrac{\As\,\eps^{1/3}}{4\, D^2},
\end{equation*}
Therefore
\begin{align*}
   \dfrac{\As}{4\,\eps\,L(\eps)} - D_6 > 0
\qquad
   \text{for $\Eps$ small enough.}
\end{align*}
Hence by~\eqref{eqn:lower bnd theta-thetas gen IC} we have
\begin{equation} \label{eqn:lower bnd theta-thetas case 2 int 1 2d gen ic}
    \theta(\tEndI)-\thetas 
\geqs 
    -a_2\,\eps 
    + \left(\dfrac{\As}{4\,\eps\,L(\eps)}-D_6 \right)\cdot C_0\,\eps^{2/3} 
    - \dfrac{\As}{4\,\eps\,(L(\eps))^2}.
\end{equation}
By the definitions of $C_0$ and $D$ in~\eqref{eqn:def D8 2d}, we have
\begin{align*}
   \dfrac{\As\,C_0}{4\, D} - \dfrac{\As}{4\, D^2}
 = 2 \,.
\end{align*} 
Hence, \eqref{eqn:lower bnd theta-thetas case 2 int 1 2d gen ic} implies that
\begin{align*}
   \theta(\tEndI)-\thetas 
\geqs&\, 
   \left(\dfrac{\As\,C_0}{4\, D} - \dfrac{\As}{4\, D^2}\right)\cdot \eps^{1/3} -a_2\,\eps - D_6  C_0\,\eps^{2/3}
\\
=&\, 2\,\eps^{1/3} -a_2\,\eps - D_6 C_0\,\eps^{2/3}
\geqs \eps^{1/3}.
\end{align*}
The latter inequality holds for $\eps$ sufficiently small. Rather than just $\theta(\tEndI)-\thetas\geqs\eps^{1/3}$, actually the equality $\theta(\tEndI)-\thetas=\eps^{1/3}$ must hold, since otherwise $\tEndI$ would have been reached before $t=C_0\,\eps^{2/3}$.
\end{description}
In conclusion, at the end of Interval I we have that
\begin{equation*}
\tEndI \leqs C_0\,\eps^{2/3} \hspace{0.06\linewidth}\text{ and} \hspace{0.06\linewidth} \theta(\tEndI)-\thetas = \eps^{1/3}.
\end{equation*}
We also have estimates for $\eta$ at the end-time $\tau_1$. First, by~\eqref{eqn:Taylor H-1}, we have
\begin{align*}
   \frac{\dtheta}{\dt} 
\leq
   \frac{1}{\Eps} \abs{\frac{1}{r} H + \eta}
\leq
   \frac{1}{\Eps R_0} \vpran{D_4 + R_0 D_3} \Eps^{2/3}
\leq 
   \frac{1}{R_0} \vpran{D_4 + R_0 D_3} \Eps^{-1/3} \,.
\end{align*}
Integrating from $0$ to $\tau_1$ gives
\begin{align*}
   \Eps^{1/3} 
  = \theta(\tau_1) - \theta^\ast
 \leq 
    \theta_0 - \theta^\ast
    + \frac{1}{R_0} \vpran{D_4 + R_0 D_3} \Eps^{-1/3} \tau_1
 \leq
    a_2 \Eps
    + \frac{1}{R_0} \vpran{D_4 + R_0 D_3} \Eps^{-1/3} \tau_1 \,.
\end{align*}
Therefore for $\Eps$ small enough, we derive that 
\begin{align}\label{eqn:lower bnd tau1}
   \tau_1 
\geq 
   \frac{R_0}{2 \vpran{D_4 + R_0 D_3}} \Eps^{2/3} \,.
\end{align}
It can be shown that the constant in this lower bound is smaller than $C_0$ (details are left to the reader), hence there is no inconsistency with the upper bound $\tEndI \leqs C_0\,\eps^{2/3}$. Substitution in~\eqref{eqn:lower bnd eta 2d int1 gen IC} of the lower bound of $\tau_1$, gives
\begin{align} \label{bound:eta-tau-1}
    \eta(\tau_1) 
\geqs 
   -D_6 \, \eps 
   + \dfrac{\As}{4\,D}\, 
       \vpran{1-\exp\vpran{-\frac{D\,R_0}{2 \vpran{D_4 + R_0 D_3}}}} \Eps^{2/3} 
> 0 
\qquad
   \text{for $\Eps$ small enough.}       
\end{align}

\smallskip

\Ni \underline{\it{Step 2: Interval II (Preparation for escaping the bottleneck). }}
Within this interval, we show that $\theta$ moves further away from $\thetas$. This will provide a large enough initial data for Interval III, during which $\theta$ attains an order $\BigO(1)$ distance from $\thetas$, thus escaping the bottleneck.
The main difference of the analysis between Interval I and II is that now we will use $\frac{1}{r} H$ as the main driving force for the evolution of $\theta$ as opposed to $\eta$ used in Interval I. 

We define the end-time of Interval II as follows. Let $\alpha\geqs 0$ be a `sufficiently small' parameter (details follow later) that is independent of $\eps$, and let
\begin{equation*} \label{eqn:def sigma2 2d}
\sigma_2:= \sup\{ t\geqs \tEndI \,:\, \eps^{1/3}\leqs\theta(s)-\thetas \leqs \alpha\, \eps^{1/6} \,\, \text{for all $s \in [\tau_1, t]$}\} \,.
\end{equation*}
The end-time for Interval II is defined as
\begin{equation} \label{eqn:def tEndII 2d}
\tEndII := \min\{ \sigma_2, \tEndI + C_1\,\eps^{2/3}-C_2\,\eps^{5/6}\} \,,
\end{equation}
where the constants are 
\begin{equation} \label{def:C-1-2-Inter-II}
   C_1:=\dfrac{2\bar{R}}{\left(\min_{\phi\in[\thetas,\thetat]}h(\phi)\right)\cdot (\thetat-\thetas)} \,,
\qquad
   C_2:=\dfrac{C_1}{\alpha}.
\end{equation}
The function $h$ is as given in \eqref{eqn:H poly theta}, and $\bar{R}$ is as in \eqref{bound:bar-R}.  We have from \eqref{bound:eta-tau-1} that $\eta(\tEndI)>0$ and, consequently, $\dtheta/\dt>0$ at time $t=\tEndI$ due to \eqref{eqn:diff ineq theta int1 2d}. Moreover, $\theta(\tEndI)-\thetas=\eps^{1/3}$. Therefore it must hold that $\sigma_2>\tEndI$ and the length of the interval $[\tEndI,\tEndII]$ (i.e.~Interval II) is strictly positive, provided that $\eps$ is sufficiently small.

Our main goal for Interval II is to show that
\begin{align} \label{goal:Interval-II}
    \theta(\tau_2) - \theta^\ast = \alpha\, \Eps^{1/6}
\qquad \text{ and } \qquad
   r(t) > 0 
\quad \text{for all } \quad
   t \in [\tau_1, \tau_2] \,.
\end{align}

The first estimate on Interval II is similar to that in Interval I. In particular, we want to  show that within Interval II:
\begin{align} \label{bound:interval-II-1}
   |\bx(t)-\bxs|
\leqs  D_7\,\eps^{2/3}, 
\qquad
     0 < R_0 \leq r(t) \leq \bar R \,,
\qquad
   |r(t)-\rs|
\leqs 
  D_{8}\,\eps^{1/6} \,, 
\end{align}
where $R_0, \bar R$ are defined in~\eqref{def:R-0-RP} or \eqref{def:R-0-RN} and~\eqref{bound:bar-R} and $D_7, D_{8}$ are some constants independent of $\Eps$. The derivation of the bounds for $\bx(t)$ and $r(t)$ in~\eqref{bound:interval-II-1} is similar to Interval I. In particular, by the bound of $r$ in~\eqref{bound:r-all-time}, given that the combined lengths of the intervals I and II is smaller than $(C_0+C_1) \eps^{2/3}$, we have
\begin{equation}
\label{eqn:bnd x-xs eps int2 2d}
    |\bx(t)-\bxs|\leqs \bar{R} (C_0+C_1)\,\eps^{2/3} 
:= D_7\,\eps^{2/3}.
\end{equation}
%
%
%
Analogously to \eqref{bound:R-Rs-perturb}, we find that 
\begin{align*}
   |R(\bx(t),\theta(t))-\rs|\leqs D_{8}\,\eps^{1/6} \,, 
\qquad
   t \in [\tau_1, \tau_2] \,,
\end{align*} 
where $D_{8}:= D_7\cdot\max_{\bar{B}}|\nabla_\bx R| + \alpha\,\max_{\bar{B}}|\partial_\theta R|$. Subsequently, an argument like \eqref{bound:r-rs-Inter-I}
yields that 
\begin{align*}
   |r(t)-\rs|\leqs D_{8}\,\eps^{1/6} \,.
\end{align*} 
Therefore all the bounds in~\eqref{bound:interval-II-1} hold. These bounds imply, like in~\eqref{bound:Interval-I-A}, that for $\Eps$ small enough, 
\begin{align*} \label{bound:A-Inter-II}
     A(t) \geqs \As/2 \,,
\qquad
     t \in [\tau_1, \tau_2] \,,
\end{align*}

Now we show that $\eta$ remains non-negative on Interval II. Using a Taylor expansion comparable to \eqref{eqn:Taylor H}, we find that 
\begin{equation}
|H(\bx,\theta)|\leqs D_{9}\,\eps^{2/3}    +   D_{10}\,\alpha^2\,\eps^{1/3},\label{eqn:bnd H int2 2d}
\end{equation}
where $D_{10}:= \max_{\bar{B}}|\partial^2_\theta H|$, and $D_{9}$ depends on a number of first- and second-order partial derivatives of $H$. Note that if we restrict ourselves to $\eps\leqs 1$ and $\alpha\leqs1$, then $D_{9}$ can be taken such that it is independent of $\alpha$ and $\eps$.
Moreover, a similar argument as in~\eqref{eqn:bnd dvH} gives
\begin{equation}
|\partial_\theta H(\bx,\theta)-\partial_\theta H(\bxs,\theta)| \leqs  \underbrace{D_7\cdot\max_{\bar{B}}|\nabla_\bx \partial_\theta H|}_{=: D_{11}}\cdot\, \eps^{2/3},\label{eqn:bnd dvH int2 2d}
\end{equation}
Due to \eqref{eqn:bnd H int2 2d} and \eqref{eqn:bnd dvH int2 2d}, we have the following analogue of \eqref{eqn:bnd B int1 2d}:
\begin{align*}
  |B(t)|
&\leqs 
   \dfrac{1}{\eps\,R_0^2}\cdot D_{11}\,\eps^{2/3}\cdot 
   (D_{9}\,\eps^{2/3}+D_{10}\,\alpha^2\,\eps^{1/3}
     +  D_7\,\max_{\bar B} \abs{\nabla_{\bx} H} \eps^{2/3})
\\
& = \frac{D_{11} D_9}{R_0^2} \Eps^{1/3}
      + \frac{D_{11} D_{10}}{R_0^2} \alpha^2
      + \frac{D_{11} D_7 \max_{\bar B} \abs{\nabla_{\bx} H}}{R_0^2} \Eps^{1/3} \,,
\end{align*}
where we also used \eqref{bound:eta-I-1} and  \eqref{eqn:bnd x-xs eps int2 2d} to bound $|\eta|$. Consequently, $|B(t)|\leqs \As/2$ on Interval II, if $\eps$ and $\alpha$ are sufficiently small. The choice of $\alpha$ is 
in particular independent of $\Eps$. Such choice of $\alpha$ gives
\begin{align} \label{bound:A-B-Inter-II}
    A(t)+B(t) \geqs 0 \,, 
\qquad
    t \in [\tau_1, \tau_2] \,.
\end{align}
The lower bound~\eqref{bound:A-B-Inter-II} is not sharp but nevertheless sufficient for the analysis in the sequel.  By~\eqref{bound:A-B-Inter-II}, the function $\eta$ satisfies the inequality
\begin{equation*} 
    \dfrac{\deta}{\dt}
\geqs C(t) 
= -G(t)\cdot \eta \,,
\end{equation*}
where $C(t), G(t)$ are defined in~\eqref{eqn:syst eta 2D general} and \eqref{eqn:def G}. The initial data $\eta(\tau_1)$ is positive due to  \eqref{bound:eta-tau-1}. Therefore,
\begin{equation*}
    \eta(t)
\geqs 
    \eta(\tEndI)\cdot \exp(-\int_{\tEndI}^t G(s)\,ds)
> 0 \,,
\qquad
    t \in [\tau_1, \tau_2] \,.
\end{equation*}

Using the positivity of $\eta$ on Interval II, we can now show the first equality in~\eqref{goal:Interval-II}. Since $\theta$ is arbitrarily close to $\thetas$ for $\eps$ sufficiently small, we have that $H(\bxs,\theta)\geqs0$ on Interval II. Hence,
\begin{equation*}
\eps\,\dfrac{\dtheta}{\dt} = \underbrace{\eta}_{>0} + \dfrac 1r \underbrace{H(\bxs,\theta)}_{\geqs0} > 0,
\end{equation*}
and thus $\theta$ is increasing in Interval II. 
Moreover, for $\Eps$ sufficiently small, we have
\begin{align*}
   \thetat-\theta(t)\geqs (\thetat-\thetas)/2 \,,
\qquad
   t \in [\tau_1, \tau_2] \,.
\end{align*} 
Hence, since $\eta$ is positive and by the definition of $C_1$ in~\eqref{def:C-1-2-Inter-II}, the $\theta$-equation satisfies
\begin{equation} \label{eqn:diff ineq theta int2 2d}
    \eps\,\dfrac{\dtheta}{\dt} 
\geqs 
    \dfrac{1}{\bar{R}}\,\bigg( \underbrace{\min_{\phi\in[\thetas,\thetat]}h(\phi)}_{>0} \bigg)\cdot \dfrac12(\thetat-\thetas) \cdot (\theta(t)-\thetas)^2
= \frac{1}{C_1} (\theta(t)-\thetas)^2 \,.
\end{equation}
This yields the differential inequality
\begin{equation} \label{eqn:diff ineq theta - thetas int2 2d}
     \dfrac{\rm d}{\dt}(\theta-\thetas)
\geqs 
     \dfrac{1}{C_1\,\eps}\,(\theta-\thetas)^2  \,,
\qquad
    \theta(\tEndI)-\thetas=\eps^{1/3} \,.
\end{equation}
Solving~\eqref{eqn:diff ineq theta - thetas int2 2d} then gives
\begin{equation} \label{eqn:lower bnd theta-thetas int2 2d}
      \theta(t)-\thetas 
\geqs 
    \dfrac{\eps}{\eps^{2/3}-  \frac{1}{C_1}\,(t-\tEndI)} \,,
\qquad
   t\in[\tEndI,\tEndII] \,.
\end{equation}
Since $\theta(\cdot)$ is nondecreasing, it cannot leave the interval $[\thetas+\eps^{1/3},\thetas+\alpha\,\eps^{1/6}]$ at its left-hand boundary. At the end of Interval II, one of the following is therefore satisfied:
\begin{description}
\item[Case II-1] $\theta(\tEndII)-\thetas=\alpha\,\eps^{1/6}$. By \eqref{eqn:def tEndII 2d}, we know that the final time must satisfy $\tEndII-\tEndI\leqs C_1\,\eps^{2/3}-C_2\,\eps^{5/6}$.
\item[Case II-2] $\tEndII-\tEndI=C_1\,\eps^{2/3}-C_2\,\eps^{5/6}$. It follows from \eqref{eqn:lower bnd theta-thetas int2 2d} that, 
\begin{equation*}
\theta(\tEndII)-\thetas \geqs \alpha\,\eps^{1/6},
\end{equation*}
due to our choice of $C_1$ and $C_2$. In fact, $\theta(\tEndII)-\thetas$ must be equal to $\alpha\,\eps^{1/6}$, since otherwise $\tEndII$ would have been reached before $t=\tEndI+C_1\,\eps^{2/3}-C_2\,\eps^{5/6}$.
\end{description}
Together with the bounds of $r$ in~\eqref{bound:interval-II-1}, we have shown that~\eqref{goal:Interval-II} holds at the end of Interval II.
\medskip

\Ni \underline{\it{Step 3: Interval III (Escape from the bottleneck).}} In this part we show that within a time interval of length $\BigO\vpran{\Eps^{5/6}}$ after Interval II, $\theta$ will escape from the bottleneck and become order $\BigO(1)$ away from $\thetas$. 

We characterize the distance from $\theta$ to $\thetas$ by a parameter $\beta \in (0, 1)$. More specifically, if {\RN } holds, then let the parameter $0<\beta<1$ be sufficiently small, such that 
\begin{align*}
    \thetas+\beta\,(\thetat-\thetas)\leqs\thetatwos \,.
\end{align*} 
Recall that $\thetatwos$ was defined in \eqref{eqn:def theta2s theta3s}, such that it is smaller than the first zero of $R(\bxs,\cdot)$; cf.~Figure \ref{fig:R neg}. If {\RP } holds, then let $\beta$ be an arbitrary parameter such that $0<\beta<1$. Such choice of $\beta$ guarantees in both cases {\RN} and {\RP}~that 
\begin{align} \label{bound:R-lower-Inter-III}
    R(t) \geq 2R_0 \,,
\qquad
    \text{in Interval III.} 
\end{align}

To define the end-time for Interval III, we first introduce the constants
\begin{equation} \label{def:C-3-4-Inter-III}
   C_3:=\dfrac{C_1}{\alpha\,(1-\beta)} \,, 
\qquad
   C_4:= \dfrac{C_1}{\beta\,(1-\beta)(\thetat-\thetas)},
\end{equation}
with $C_1$ defined in \eqref{def:C-1-2-Inter-II}. Let
\begin{equation} \label{eqn:def sigma3 2d}
    \sigma_3:= \sup\{ t\geqs \tEndII \,:\,  \alpha\,\eps^{1/6}\leqs\theta(s)-\thetas \leqs \beta\,(\thetat-\thetas) \,\, \text{for all $s \in [\tau_2, t]$} \} \,.
\end{equation}
Then the end-time of Interval III is defined as
\begin{equation}
\tEndIII := \min\{ \sigma_3, \tEndII+C_3\,\eps^{5/6}-C_4\,\eps\}.\label{eqn:def tEndIII 2d}
\end{equation}
Observe that by evaluating \eqref{eqn:diff ineq theta - thetas int2 2d} at $t=\tEndII$ and noting that $\theta(\tEndII)-\thetas=\alpha\,\eps^{1/6}$, we obtain that ${\rm d}\theta/\dt>0$ at $t=\tEndII$. Therefore $\sigma_3>\tEndII$ and the length of Interval III, i.e.~the interval $[\tEndII,\tEndIII]$, is strictly positive.

The main goal in this part is to show that
\begin{align} \label{goal:Inter-III}
    \theta(\tEndIII)-\thetas = \beta\,(\thetat-\thetas) \,,
\quad \text{and} \quad
    r(t) > 0 \,,
\qquad
   t \in [\tau_2, \tau_3] \,.
\end{align}

First, a similar argument as for Interval I and II shows that by our choice of $\beta$ and~\eqref{bound:R-lower-Inter-III}, 
\begin{align} \label{bound:Inter-III-1}
   |\bx(t)-\bxs|
\leqs 
   \bar{R}\, (C_0+C_1+C_3)\,\eps^{2/3} \,,
\qquad
   R_0 \leqs r(t) \leqs \bar R \,,
\qquad
   t \in [\tau_2, \tau_3] \,.
\end{align}

To estimate the rate of growth of $\theta$ we need an appropriate bound for $\eta$. Unlike the estimates for $\eta$ in Interval I and II, where we carefully compared the sizes of $A$ and $B$, now we only need a rather crude estimate following directly from the definition in~\eqref{eqn:def eta 2d general}. First, for $\Eps$ small enough, we know that 
\begin{equation*} 
   (\bx(t),\theta(t))\in  \bar{B}'\,,\qquad
   t\in[\tEndII,\tEndIII],
   \end{equation*}
where
   \begin{equation}\label{eqn:def ball B bar prime}
\bar{B}':= \bar{B}(\bxs,1)\times\bar{B}(\thetas,\thetat-\thetas) \,.
\end{equation}
Therefore by definition of $\eta$ in \eqref{eqn:def eta 2d general},  it follows that
\begin{equation} \label{eqn:bnd eta x-xs int3 2d}
     |\eta(t)| 
\leqs 
    \dfrac{1}{R_0}\max_{\bar{B}'}|\nabla_\bx H | \cdot  |\bx(t)-\bxs|.
\end{equation}
Combining \eqref{bound:Inter-III-1} with \eqref{eqn:bnd eta x-xs int3 2d}, we have
\begin{equation}
|\eta(t)| \leqs D_{12}\, \eps^{2/3},\label{eqn:bnd eta eps int3 2d}
\end{equation}
with $D_{12}:=\bar{R}/R_0\cdot (C_0+C_1+C_3)\cdot\max_{\bar{B}'}|\nabla_\bx H|$, which is independent of $\eps$.

Using the bound for $\eta$, now we study the evolution of $\theta - \thetas$. By \eqref{eqn:def sigma3 2d}:
\begin{equation*}
    \thetas+\alpha\,\eps^{1/6} 
\leqs 
    \theta(t) \leqs \thetas + \beta\,(\thetat-\thetas) \,,
\qquad
   t \in [\tau_2, \tau_3] \,.
\end{equation*}
In particular, $\theta(t)\in[\thetas,\thetat]$ and $\thetat-\theta(t)\geqs (1-\beta)\cdot(\thetat-\thetas)$. For the evolution of $\theta$ we have, 
\begin{align}\label{eqn:diff ineq theta int3 2d with eta}
\eps\, \dfrac{{\rm d} (\theta-\thetas)}{\dt} \geqs 
    \dfrac{1}{\bar{R}}\,\bigg( \min_{\phi\in[\thetas,\thetat]}h(\phi) \bigg)\cdot (1-\beta)\,(\thetat-\thetas) \cdot (\theta(t)-\thetas)^2 + \eta(t).
\end{align}
This estimate is similar to \eqref{eqn:diff ineq theta int2 2d}, be it that $\eta$ is no longer nonnegative. Instead, we know that $\eta(t)\geqs -D_{12}\,\eps^{2/3}$, due to \eqref{eqn:bnd eta eps int3 2d}. If we define
\begin{align*}
   D_{13}:=\frac{2(1-\beta)}{C_1},
\end{align*}
with $C_1$ given by~\eqref{def:C-1-2-Inter-II},  then $D_{13}$ is exactly the coefficient in front of $(\theta-\thetas)^2$ in \eqref{eqn:diff ineq theta int3 2d with eta}. Hence, it follows from \eqref{eqn:diff ineq theta int3 2d with eta} that
\begin{align*}
\eps\, \dfrac{{\rm d} (\theta-\thetas)}{\dt} \geqs 
    D_{13} (\theta-\thetas)^2 -D_{12}\,\eps^{2/3}.
\end{align*}
Consequently, the $\theta$-equation satisfies
\begin{align}  \label{eqn:diff ineq theta int3 2d}
   \dfrac{{\rm d}(\theta-\thetas)}{\dt}
&\geqs 
   \dfrac12\,D_{13}\,\eps^{-1}\,(\theta-\thetas)^2 
   + \underbrace{\dfrac12\,D_{13}\,\alpha^2\,\eps^{-2/3} - D_{12}\,\eps^{-1/3}}_{\geqs \, 0 \,\,\text{ for $\eps$ sufficiently small}}
 \geqs \dfrac12\,D_{13}\,\eps^{-1}\,(\theta-\thetas)^2 \,,
\end{align}
for all $t\in[\tEndII,\tEndIII]$. Here, we used that $(\theta-\thetas)\geqs \alpha\,\eps^{1/6}$. The accompanying initial condition is $\theta(\tEndII)-\thetas=\alpha\,\eps^{1/6}$. After solving for the corresponding subsolution, we find that
\begin{equation}
\theta(t)-\thetas \geqs \dfrac{2\,\alpha\,\eps}{2\,\eps^{5/6}- \alpha\, D_{13}\,(t-\tEndII)}.\label{eqn:lower bnd theta-thetas int3 2d}
\end{equation}
Since $\theta$ is increasing within Interval III, at the end of the interval, there are only two possible scenarios:
\begin{description}
\item[Case III-1] $\theta(\tEndIII)-\thetas=\beta\,(\thetat-\thetas)$. By \eqref{eqn:def tEndIII 2d}, we know that the final time must satisfy $\tEndIII-\tEndII\leqs C_3\,\eps^{5/6}-C_4\,\eps$.
\item[Case III-2] $\tEndIII-\tEndII=C_3\,\eps^{5/6}-C_4\,\eps$. 
It follows from \eqref{eqn:lower bnd theta-thetas int3 2d} that
\begin{equation*}
\theta(\tEndIII)-\thetas \geqs \dfrac{2\,\alpha\,\eps}{2\,\eps^{5/6}- \alpha\, D_{13}\,(C_3\,\eps^{5/6}-C_4\,\eps)} = \beta\,(\thetat-\thetas),
\end{equation*}
by the definitions of $C_3, C_4$ in~\eqref{def:C-3-4-Inter-III}. 
This lower bound implies that $\theta(\tEndIII)-\thetas$ must actually be equal to $\beta\,(\thetat-\thetas)$, since otherwise $\tEndIII$ would have been reached before $t=\tEndII+C_3\,\eps^{5/6}-C_4\,\eps$.
\end{description}
It follows from the above considerations that~\eqref{goal:Inter-III} holds on Interval III.
\medskip

\Ni \underline{\it{Step 4: Interval IV (Convergence to $\thetat$). }} In this last interval we show that $\theta$ converges to $\thetat$. The rough definition of the endpoint of Interval IV is the time when $\theta$ falls into a small neighbourhood of $\thetat$. The main idea is that the first term $\frac{1}{r} H$ on the right-hand side of the $\theta$-equation becomes a linear forcing that drives $\theta$ toward $\thetat$. However, unlike in the previous three intervals, we have to treat separately the case (RN), as the forcing term $R$ in the $r$-equation may become negative in Interval IV. This generates a delicate situation regarding the positivity of $r$. 
\medskip

\Ni {\underline{\bf Case (RN)}} In this case $R$ is bounded but is allowed to be negative for $\theta \in [\thetatwos, \thetathrees]$; see Figure \ref{fig:R neg}. We will show however that $r$ remains strictly positive. To this end, we divide Interval IV into two subintervals IV-A and IV-B: during the first subinterval IV-A, $\theta$ evolves from $\theta(\tau_3)$ to $\thetathrees$, thus covering the whole region where $R$ can be negative. This is the region where we need to estimate $r$ carefully and show that it remains strictly positive. During the second subinterval $\theta$ satisfies $\theta \geqs \thetathrees$. Hence, $R$ has a strict lower bound which makes the analysis more straightforward.  

We start by giving the precise definition of the first subinterval IV-A. Let 
\begin{equation}
\sigma_{4A}:= \sup\{ t\geqs \tEndIII \,:\,  \thetas + \beta\,(\thetat-\thetas)\leqs \theta(s) \leqs \thetathrees \,\,\, \text{and} \,\,\, r(s)>0 \,\, \text{for all $s \in [\tau_3, t]$}\},\label{eqn:def sigma4A 2d}
\end{equation}
Define the end-time $\tau_{4A}$ of the first subinterval IV-A as 
\begin{equation}
\tEndIVA := \min\{ \sigma_{4A}, \tEndIII+C_5\,\eps\},\label{eqn:def tEndIVA 2d}
\end{equation}
where
\begin{equation}\label{eqn:def C5 int4A}
C_5 := \dfrac{\bar{R}}{D_{14}}\left(\thetathrees - \thetas-\beta\,(\thetat-\thetas)\right)>0 \,,
\qquad
D_{14}:= \dfrac12\,\min_{\phi\in[\thetas+\beta\,(\thetat-\thetas)\,,\,\thetathrees]} H(\bxs,\phi) > 0.
\end{equation}
Since $\theta(\tEndIII)-\thetas=\beta\,(\thetat-\thetas)$, \eqref{eqn:diff ineq theta int3 2d} implies that $\dtheta/\dt>0$ at $t=\tEndIII$. Moreover, $r(\tEndIII)\geqs R_0>0$. Hence, cf.~\eqref{eqn:def sigma4A 2d}, it holds that $\sigma_{4A}>\tEndIII$ and the length of Interval IV-A is strictly positive. Our main objective for Interval IV-A is to show that there exists a constant $\hat R_0 > 0$ such that 
\begin{align} \label{goal:Inter-IV-A}
    \theta(\tau_{4A}) = \thetathrees \,,
\qquad
 0 <  \hat R_0 \leqs r(t) \leqs \bar R  \,,
\qquad
   t \in [\tau_3, \tau_{4A}] \,.
\end{align}
The upper bound $r \leq \bar R$ holds for all $t\geqs0$, while by the definition of $\tau_{4A}$ we have $r > 0$ for $t \in [\tau_3, \tau_{4A})$. Consequently, 
\begin{equation} \label{eqn:bnd x-xs int4A 2d}
   |\bx(t)-\bxs|
\leqs 
   \bar{R}\, (C_0+C_1+C_3)\,\eps^{2/3} + \bar{R}\, C_5 \, \eps \,,
\qquad
  t\in[\tEndIII,\tEndIVA] \,.
\end{equation}
Define
\begin{equation*}
    \hat{R} 
:= 2\, \underbrace{\max_{\phi\in[\thetas,\thetat]}\left(-R(\bxs,\phi)  \right) }_{\geqs\, 0 \text{ in case \RN} }+ 1
> 0 \,.
\end{equation*}
The reason for adding 1 in the above definition is to make sure that $\hat R > 0$ rather than $\hat R \geqs 0$, which turns out to be useful in the sequel. Then
\begin{align} \label{bound:lower-R-Inter-IV-A}
   R(\bxs,\theta(t))\geqs -\hat{R}/2 \,,
\qquad
   t\in[0,\tEndIVA] \,.
\end{align} 
Due to \eqref{eqn:bnd x-xs int4A 2d}, the difference between $R(\bx(t),\theta(t))$ and $R(\bxs,\theta(t))$ is arbitrarily small (for sufficiently small $\eps$). Hence, it follows from~\eqref{bound:lower-R-Inter-IV-A} that
\begin{equation*} 
    R(\bx(t),\theta(t)) > -\hat{R} \,,
\qquad
    t \in [\tau_3, \tau_{4A}] \,,
\end{equation*}
provided $\eps$ is sufficiently small. 
In order to show the positive lower bound of $r$, we study the rate of change of $r$ with respect to $\theta$. Note that $(\bx(t),\theta(t))\in\bar{B}'$ within Interval IV-A, with $\bar{B}'$ being defined in \eqref{eqn:def ball B bar prime}. By \eqref{eqn:bnd x-xs int4A 2d} and the definition of $\eta$ in~\eqref{eqn:def eta 2d general}, we have
\begin{equation}\label{eqn:bnd eta int4A}
|\eta(t)| \leqs \dfrac{1}{r(t)}\max_{\bar{B}'}|\nabla_\bx H | \cdot  |\bx(t)-\bxs| \leqs \dfrac{1}{r(t)}\max_{\bar{B}'}|\nabla_\bx H | \cdot  \bar{R}\cdot(C_0+C_1+C_3+C_5)\,\eps^{2/3}.
\end{equation}
Furthermore, 
\begin{equation*}
     \dfrac{1}{r(t)}\, H(\bxs,\theta(t))  
\geqs 
    \dfrac{1}{r(t)}\, \min_{\phi\in[\thetas+\beta\,(\thetat-\thetas)\,,\,\thetathrees]} H(\bxs,\phi) >0 \,,
\qquad
  t \in [\tau_3, \tau_{4A})
%
\end{equation*}
with the lower bound being \textit{strictly} positive due to \eqref{eqn:H poly theta}. By \eqref{eqn:bnd eta int4A} $\eta$ is arbitrarily small for $\eps$ sufficiently small, and hence we have (for $\Eps$ small enough) that
\begin{equation} \label{eqn:lower bnd rhs dtheta/dt int4A}
   \Eps \frac{\dtheta}{\dt} 
=
    \dfrac{1}{r (t)}\,H(\bxs,\theta(t)) + \eta(t) 
\geqs  
     \dfrac{D_{14}}{r(t)} > 0 \,,
\qquad
   t \in [\tau_3, \tau_{4A}) \,,
\end{equation}
where the positive constant $D_{14}$ is given by \eqref{eqn:def C5 int4A}.
The evolution of $r$, on the other hand, satisfies
\begin{equation} \label{bound:eqn-r-IV-A}
\eps\,\dfrac{\dr}{\dt}=-r+R(\bx,\theta)\geqs -(r+\hat{R}) 
\end{equation}
with $r + \hat R > 0$. Combining~\eqref{eqn:lower bnd rhs dtheta/dt int4A} and~\eqref{bound:eqn-r-IV-A}, we have
\begin{equation*}
   \dfrac{\dr}{\dtheta}
=  \dfrac{\eps\,\dfrac{\dr}{\dt}}{\eps\,\dfrac{\dtheta}{\dt}} 
\geqs 
    \dfrac{-(r+\hat{R})}{\eps\,\dfrac{\dtheta}{\dt}}
\geqs
   \frac{- r (r + \hat R)}{D_{14}},
\end{equation*}
where the right-hand side is strictly negative. Hence we have
\begin{equation*}
\dfrac{\dr}{r (r+\hat{R})}=\dfrac{1}{\hat{R}}\,\left(\dfrac1r   -\dfrac{1}{r+\hat{R}}  \right)\,\dr \geqs -\dfrac{\dtheta}{D_{14}}.
\end{equation*}
Integration of this differential inequality from $\tEndIII$ to some $t\in[\tEndIII,\tEndIVA]$ yields
\begin{equation*}
\ln\left(\dfrac{r(t)}{r(\tEndIII)}\right)   -\ln\left(\dfrac{r(t)+\hat{R}}{r(\tEndIII)+\hat{R}} \right) \geqs -\dfrac{\hat{R}}{D_{14}} \, (\theta(t)-\theta(\tEndIII))
\end{equation*}
from which it follows that
\begin{equation} \label{bound:lower-r-explicit-IV-A}
   r(t) 
\geqs 
   \dfrac{\hat{R}}{\left( 1+ \hat{R}/r(\tEndIII) \right)\cdot\exp\left( \hat{R} \cdot (\theta(t)-\theta(\tEndIII))/D_{14}\right) -1} \,,
\qquad
   t \in [\tau_3, \tau_{4A}) \,.
\end{equation}
By~\eqref{eqn:lower bnd rhs dtheta/dt int4A}, we know that $\theta$ is strictly increasing on Interval IV-A. Hence $\theta(t) \geqs \theta(\tau_3)$ for $t \in [\tau_3, \tau_{4A}]$ and the right-hand side of~\eqref{bound:lower-r-explicit-IV-A} is strictly positive. Furthermore, by~\eqref{goal:Inter-III} and~\eqref{bound:Inter-III-1}, we know that $R_0\leqs r(\tEndIII)\leqs \bar{R}$,
$\theta(\tEndIII)=\thetas+\beta\,(\thetat-\thetas)$; also, $\theta(t)\in[\thetas+\beta\,(\thetat-\thetas)\,,\,\thetathrees]$ on Interval IV-A. Hence, 
\begin{equation}\label{eqn:lower bnd r int4A} 
   r(t) 
\geqs 
   \hat R_0 
 := \dfrac{R_0\,\hat{R}}{R_0+\hat{R}}\cdot\exp\left( -\hat{R} \, (\thetathrees-\thetas)/D_{14}\right)
> 0 \,,
\qquad
   t \in [\tau_3, \tau_{4A}] \,.
\end{equation}
Therefore, the second bound in~\eqref{goal:Inter-IV-A} holds and the end of Interval IV-A cannot be due to $r$ reaching~zero.

To show the first equality in~\eqref{goal:Inter-IV-A}, we derive from~\eqref{eqn:lower bnd rhs dtheta/dt int4A} that
\begin{equation}\label{eqn:lower bnd theta fnc of t int4A}
    \theta(t) 
\geqs 
   \thetas+\beta\,(\thetat-\thetas)+\dfrac{D_{14}\,(t-\tEndIII)}{\eps\,\bar{R}} \,,
\qquad
    t\in[\tEndIII,\tEndIVA]
\end{equation}
Since $\theta$ is increasing, the end of Interval IV-A is reached in either of the following two cases:
\begin{description}
\item[Case IV-1] $\theta(\tEndIVA)=\thetathrees$. By \eqref{eqn:def tEndIVA 2d}, we know that the final time must satisfy $\tEndIVA-\tEndIII\leqs C_5\,\eps$.
\item[Case IV-2] $\tEndIVA-\tEndIII=C_5\,\eps$. Therefore, due to \eqref{eqn:lower bnd theta fnc of t int4A}, it holds at time $\tEndIVA$ that
\begin{equation*}
\theta(\tEndIVA) \geqs \thetas+\beta\,(\thetat-\thetas)+\dfrac{D_{14}\,C_5}{\bar{R}} = \thetathrees,
\end{equation*}
where the last equality follows from the definition of $C_5$ in \eqref{eqn:def C5 int4A}. Thus, $\theta(\tEndIVA)=\thetathrees$ must hold, since otherwise the end of Interval IV-A would have been reached earlier than at $\tEndIVA$.
\end{description}
We thereby have shown that~\eqref{goal:Inter-IV-A} holds. 


\medskip

Next, we define the second subinterval IV-B.  Let $\lambda$ be an arbitrary constant such that $0<\lambda\leqs 1/3$ and fix $n \in \N$ such that $n\,\lambda\geqs 2/3$. Define
\begin{equation}
\sigma_{4B}:= \sup\{ t\geqs \tEndIVA \,:\, \thetathrees \leqs \theta(s) \leqs \thetat-c\,\eps^{2/3}\,\, \text{for all } s \in [\tau_{4A}, t]\},\label{eqn:def sigma4B 2d}
\end{equation}
The end-time of the subinterval IV-B is defined by
\begin{equation}
\tEndIVB := \min\{ \sigma_{4B}, \tEndIVA+C_6\,\eps^{1-\lambda}\},\label{eqn:def tEndIVB 2d}
\end{equation}
with
\begin{equation}\label{eqn:C6 and c}
C_6 := \dfrac{\bar{R}\,(n!\,(\thetat-\thetathrees))^{1/n}}{(\thetathrees-\thetas)^2\cdot \min_{\phi\in[\thetas,\thetat]}h(\phi)},
\qquad
c=1+\dfrac{\bar{R}^2\, (C_0+C_1+C_3+C_5+C_6) \cdot \max_{\bar{B}'}|\nabla_\bx H|}{\tilde{R}_0\,(\thetathrees-\thetas)^2\cdot \min_{\phi\in[\thetas,\thetat]}h(\phi)},
\end{equation}
where $\tilde{R}_0 = \min\{R_0, \hat R_0\}$ with $\hat R_0$ being the lower bound of $r$ on IV-A defined in~\eqref{eqn:lower bnd r int4A}.
Since $\theta(\tEndIVA)=\thetathrees$ and $\dtheta/\dt>0$ at time $\tEndIVA$, it holds that $\sigma_{4B}>\tEndIVA$. Consequently, $\tEndIVB>\tEndIVA$ and the time-length of the subinterval IV-B is strictly positive. The introduction of the arbitrary number $\lambda$ will be clear from the analysis below.

Our main goal on Interval IV-B is to show that
\begin{align} \label{goal:IV-B}
    \theta(\tau_{4B}) = \thetat - c\, \Eps^{2/3} \,,
\qquad
    r(t) \geq \tilde{R}_0 = \min\{R_0, \hat R_0\} \,,
\qquad
   t \in [\tau_{4A}, \tau_{4B}] \,.
\end{align}

Similar as in all the previous intervals, the upper bound \eqref{bound:r-all-time} holds for all $t\geqs0$. Consequently,
\begin{equation} \label{eqn:bnd x-xs int4B 2d}
  |\bx(t)-\bxs|
\leqs 
   \bar{R}\cdot(C_0+C_1+C_3+C_5+C_6)\,\eps^{2/3} \,,
\qquad
   t\in[0,\tEndIVB]
\end{equation}
where we use that $\lambda\leqs1/3$. Due to \eqref{def:R-0-RN} and \eqref{eqn:def sigma4B 2d} we have that $R(\bxs,\theta(t))\geqs 2\,R_0$ for all $t\in[\tEndIVA,\tEndIVB]$, and therefore
\begin{align*}
    r(t) \geqs \tilde R_0 = \min\{R_0, \hat R_0\} > 0 \,,
\qquad
    t\in[0,\tEndIVB] \,.
\end{align*}
Compare the way in which we derived the lower bounds on $r$ in 
\eqref{bound:interval-I-1}, \eqref{bound:interval-II-1}, \eqref{bound:Inter-III-1} and \eqref{goal:Inter-IV-A}. Hence the lower bound for $r$ in~\eqref{goal:IV-B} holds.

In order to show the first equality in~\eqref{goal:IV-B}, we again use the definition of $\eta$ and obtain its bound as 
\begin{equation} \label{eqn:bnd eta eps int4B 2d}
    |\eta(t)|  \leqs D_{15}\,\eps^{2/3} \,,
\qquad
    t\in[\tEndIVA,\tEndIVB] \,.
\end{equation}
where the constant $D_{15}$ is given by 
\begin{align}\label{eqn:const D15 eta 4B}
   D_{15}:= \bar{R}/\tilde{R}_0\cdot (C_0+C_1+C_3+C_5+C_6) \cdot \max_{\bar{B}'}|\nabla_\bx H| \,.
\end{align}
Using the above bound of $\eta$, we have
\begin{equation} \label{eqn:diff ineq theta-thetat int4B 2d}
   \eps \dfrac{{\rm d}(\theta-\thetat)}{\dt} 
\geqs 
   -D_{16}\, (\theta(t)-\thetat) - D_{15}\,\eps^{2/3} \,,
\end{equation}
where 
\begin{equation}\label{eqn:def const D16 theta eqn 4B}
D_{16}:=(\thetathrees-\thetas)^2/\bar{R}\cdot \min_{\phi\in[\thetas,\thetat]}h(\phi).
\end{equation}
Note that \eqref{eqn:C6 and c}, \eqref{eqn:const D15 eta 4B} and \eqref{eqn:def const D16 theta eqn 4B} imply that $c=1+D_{15}/D_{16}$. The initial condition $\theta(\tEndIVA)-\thetat = \thetathrees-\thetat<0$ is an $\Ord(1)$-quantity, while $\thetat-\theta(t)\geqs c\,\eps^{2/3}$ on Interval IV-B due to \eqref{eqn:def sigma4B 2d}. Therefore, for any $\eps$ sufficiently small, the right-hand side of \eqref{eqn:diff ineq theta-thetat int4B 2d} is strictly positive for all $t$ in Interval IV-B, and thus $\theta$ is increasing. Moreover, the subsolution of \eqref{eqn:diff ineq theta-thetat int4B 2d} corresponding to the initial condition $\theta(\tEndIVA)-\thetat$ leads to the lower bound:
\begin{equation}
\theta(t)-\thetat \geqs -\dfrac{D_{15}}{D_{16}}\,\eps^{2/3} + \left(\dfrac{D_{15}}{D_{16}}\,\eps^{2/3}- (\thetat-\thetathrees) \right)\cdot \exp(-D_{16}\,(t-\tEndIVA)/\eps).\label{eqn:lower bnd theta-thetat int4B 2d}
\end{equation}
Since $\theta$ is increasing, the end of Interval IV-B is reached in either of the following two cases:
\begin{description}
\item[Case IV-3] $\theta(\tEndIVB)=\thetat-c\,\eps^{2/3}$. By \eqref{eqn:def tEndIVB 2d}, we know that the final time must satisfy $\tEndIVB-\tEndIVA\leqs C_6\,\eps^{1-\lambda}$.
\item[Case IV-4] $\tEndIVB-\tEndIVA=C_6\,\eps^{1-\lambda}$. The lower bound in \eqref{eqn:lower bnd theta-thetat int4B 2d} implies that
\begin{align}
\theta(\tEndIVB)-\thetat \geqs&\, -\dfrac{D_{15}}{D_{16}}\,\eps^{2/3} + \left(\dfrac{D_{15}}{D_{16}}\,\eps^{2/3}- (\thetat-\thetathrees) \right)\cdot \exp(-D_{16}\,C_6\,\eps^{-\lambda})\\
\geqs&\, -\dfrac{D_{15}}{D_{16}}\,\eps^{2/3} - (\thetat-\thetathrees) \cdot \exp(-D_{16}\,C_6\,\eps^{-\lambda}).\label{eqn:bnd theta int4B}
\end{align}
For each $w\geqs0$ and $\gamma\in\N$, the inequality $\exp(w)\geqs 1 + w^\gamma/\gamma!$ holds. Hence,
\begin{equation}
\exp(-D_{16}\,C_6\,\eps^{-\lambda})\leqs \dfrac{1}{1+(D_{16}\,C_6)^n\,\eps^{-n\lambda}/n!} = \dfrac{n!\,\eps^{n\lambda}}{n!\,\eps^{n\lambda}+(D_{16}\,C_6)^n} \leqs \dfrac{n!\,\eps^{2/3}}{(D_{16}\,C_6)^n},\label{eqn:bnd exp int4}
\end{equation}
where we used in the last inequality that $n\lambda\geqs 2/3$. Note that $D_{16}\,C_6=  (n!\,(\thetat-\thetathrees))^{1/n}$, by \eqref{eqn:C6 and c} and \eqref{eqn:def const D16 theta eqn 4B}. Consequently, \eqref{eqn:bnd theta int4B} and \eqref{eqn:bnd exp int4} combine into
\begin{equation*}
\theta(\tEndIVB)-\thetat \geqs -\left(1+\dfrac{D_{15}}{D_{16}}\right)\,\eps^{2/3} = -c\,\eps^{2/3},
\label{eqn:bnd theta final int4B}
\end{equation*}
since $c=1+D_{15}/D_{16}$, as stated before. Thus, $\theta(\tEndIVB)-\thetat =-c\,\eps^{2/3}$ must hold, since otherwise the end of Interval IV-B would have been reached earlier than at $\tEndIVB$.
\end{description}
We thereby have shown that $|\theta(\tEndIVB)-\thetat|=c\,\eps^{2/3}$, while $\tEndIVB\leqs (C_0+C_1+C_3+C_5+C_6)\,\eps^{2/3}$. Furthermore, $\tEndIVB\geqs\tEndI\geqs \bar{C}_1\,\eps^{2/3}$, where the constant $\bar{C}_1$ is provided by \eqref{eqn:lower bnd tau1}. Hence the statement of the theorem is proved in the case \RN.
\medskip

\Ni {\underline{\bf Case (RP)}} This case is a repetition of the analysis in Inteval IV-B for the (RN) case, since now we automatically have 
\begin{align*}
    r(t) \geq R_0 > 0 \,.
\end{align*}
Therefore it is \textit{not} necessary to subdivide Interval IV into two subintervals. We nevertheless present some details for the convenience of the reader. 

Let $\lambda$ be an arbitrary constant such that $0<\lambda\leqs 1/3$ and let $n \in \N$ be such that $n\lambda\geqs 2/3$. Define
\begin{equation}
\sigma_4:= \sup\{ t\geqs \tEndIII \,:\,  \thetas+ \beta\,(\thetat-\thetas) \leqs \theta(s) \leqs \thetat-c\,\eps^{2/3} \,\, \text{for all } s \in [\tau_{3}, t]\},\label{eqn:def sigma4 2d}
\end{equation}
and the end-time $\tau_4$ of Interval IV as
\begin{equation} \label{eqn:def tEndIV 2d}
   \tEndIV 
:= \min\{ \sigma_4, \tEndIII+C_7\,\eps^{1-\lambda}\} \,,
\end{equation}
with
\begin{equation*}
    C_7 
  := \dfrac{C_1\,(n!\,(1-\beta)\,(\thetat-\thetas))^{1/n}}{2\,\beta^2\,(\thetat-\thetas)} \,,
\qquad
c := 1 + \dfrac{C_1\left(D_{12}+  C_7\,\bar{R}/R_0\cdot\max_{\bar{B}'}|\nabla_\bx H|\right)}{2\,\beta^2\,(\thetat-\thetas)}.
\end{equation*}
Together with $\theta(\tEndIII)-\thetas=\beta\,(\thetat-\thetas)$, \eqref{eqn:diff ineq theta int3 2d} implies that ${\rm d}\theta/\dt>0$ at $t=\tEndIII$. Hence, it holds that $\sigma_4>\tEndIII$.
Using similar arguments as before, we can show that $0<R_0\leqs r(t) \leqs \bar{R}$ is satisfied within Interval IV. Moreover, $|\bx(t)-\bxs|$ has an upper bound of order $\mathcal{O}(\eps^{2/3})$; this can be proved analogously to e.g.~\eqref{eqn:bnd x-xs int4B 2d}. Consequently, we have due to \eqref{eqn:bnd eta x-xs int3 2d} that
\begin{equation*}
|\eta(t)|\leqs D_{17}\,\eps^{2/3}\label{eqn:bnd eta eps int4 2d}
\end{equation*}
for all $t\in[\tEndIII,\tEndIV]$. Here, $D_{17}:= D_{12}+ \bar{R}/R_0\cdot C_7 \cdot \max_{\bar{B}'}|\nabla_\bx H|$. The estimate of $\eta$ is very similar to \eqref{eqn:bnd eta eps int3 2d}, be it that the constant $D_{17}$ takes into account the appropriate estimate for $|\bx(t)-\bxs|$ in this interval.

By definition of $\sigma_4$, we have that $\theta(t)-\thetas\geqs \beta\,(\thetat-\thetas)$ on Interval IV. Hence, like in \eqref{eqn:diff ineq theta-thetat int4B 2d}, we have that
\begin{equation}
\eps \dfrac{{\rm d}(\theta-\thetat)}{\dt} \geqs -D_{18}\, (\theta(t)-\thetat) - D_{17}\,\eps^{2/3}\label{eqn:diff ineq theta-thetat int4 2d}
\end{equation}
holds in Interval IV, with $D_{18}=2\beta^2\,(\thetat-\thetas)/C_1$. The initial condition is $\theta(\tEndIII)-\thetat = -(1-\beta)\cdot(\thetat-\thetas)$.
Now, we follow the same steps as in Interval IV-B for the case \RN. The end of Interval IV is reached in either of the following two cases:
\begin{description}
\item[Case 1] $\theta(\tEndIV)=\thetat-c\,\eps^{2/3}$. By \eqref{eqn:def tEndIV 2d}, we know that the final time must satisfy $\tEndIV-\tEndIII\leqs C_7\,\eps^{1-\lambda}$.
\item[Case 2] $\tEndIV-\tEndIII=C_7\,\eps^{1-\lambda}$. Due to the same arguments as in Interval IV-B of \RN, and using the definitions of $C_7 = (n!\,(1-\beta)\,(\thetat-\thetas))^{1/n}/D_{18}$ and $c=1+D_{17}/D_{18}$, we obtain that
\begin{equation*}
\theta(\tEndIV)-\thetat \geqs  -c\,\eps^{2/3}.
\end{equation*}
Thus, $|\theta(\tEndIV)-\thetat| =c\,\eps^{2/3}$ must hold, since otherwise the end of Interval IV would have been reached earlier than at $\tEndIV$.
\end{description}
We have  shown that $|\theta(\tEndIV)-\thetat|=c\,\eps^{2/3}$, while $\tEndIV\leqs (C_0+C_1+C_3+C_7)\,\eps^{2/3}$. Furthermore, $\tEndIVB\geqs\tEndI\geqs \bar{C}_1\,\eps^{2/3}$, where the constant $\bar{C}_1$ is provided by \eqref{eqn:lower bnd tau1}. Hence the statement of the theorem is proved also for the case \RP.
\end{proof}

\begin{remark}

Let $\zeta$ be an $\Ord(1)$ quantity such that $R(\bxs,\cdot)$ is strictly positive on $(\thetas-\zeta,\thetas]$. Such $\zeta$ exists by the continuity of $R(\bxs, \cdot)$ and the positivity of $R(\bxs, \thetas)$. Then for each fixed $\eps$, the well-posedness of~\eqref{eqn:syst eta 2D general} over the time period when $\theta \in (\thetas-\zeta,\thetat)$ immediately follows from the bounds (both lower and upper) of $r$ and $\theta$ shown in the proof of Theorem~\ref{thm: conv quad lin 2D R neg}. 
\end{remark}


Next we show that, provided $R$ is not a constant on $[\thetas, \thetat]$, the speed $r$ also experiences a finite jump; compare Figure \ref{fig:trajectories zoom r theta}(c). The jump in $r$ is generated by the jump of $\theta$ shown in the previous theorem. The main idea here is: $R$ undergoes a finite jump that is induced by the jump of $\theta$. This further translates into the jump of $r$. Back to the particle description, this indicates that both the speed and direction of the particle make an instantaneous change as $\Eps \to 0$. This is the procedure that picks up the correct state for the first-order model.  

\begin{theorem}[Finite jump in $r$]\label{thm:r jump 2D}
Let the assumptions for $H, R$ in Theorem \ref{thm: conv quad lin 2D R neg} hold and let $\tau$ be the end-time of the fourth interval in Theorem~\ref{thm: conv quad lin 2D R neg}. Then 
\begin{align} \label{cond:jump-r}
    \sup_{t \in [0, \tau]} |r(t)-\rs| = \BigO(1) 
\end{align} 
if and only if $R(\bxs,\cdot)$ is \textbf{\emph{not}} a constant on $[\thetas,\thetat]$. 
\end{theorem}

\begin{proof}
Note that by construction $\tau$ depends on $\Eps$. In the proof of Theorem \ref{thm: conv quad lin 2D R neg}, $\tau$ is named $\tEndIVB$ or $\tEndIV$. 

\Ni First we consider the case where $R(\bxs,\cdot)$ is constant on $[\thetas,\thetat]$. In this case we show that~\eqref{cond:jump-r} does not hold. 
Note that by assumption of the theorem, $R(\bxs,\thetas)>0$. Hence, if $R(\bxs,\cdot)$ is constant on $[\thetas,\thetat]$, then the case {\RP } must hold.

By assumption \eqref{bound:bar-R}, 
the upper bound $\bar{R}$ is 
large enough such that $\ro\leqs \bar{R}$. As argued before, then $r(t)\leqs \bar{R}$ must hold for all $t\in[0,\tau]$. This implies
there is a constant $\bar{D}\geqs0$ such that 
\begin{align*}
    |\bx(t)-\bxs|\leqs |\bx(t)-\bxo|+|\bxo-\bxs|
\leqs 
    \bar{D}\,\eps^{2/3} \,.
\end{align*} 
where we have used that $\tau=\Ord(\eps^{2/3})$ and $|\bxo-\bxs|\leqs a_1\,\eps$. Moreover, by the proof of Theorem \ref{thm: conv quad lin 2D R neg}, we have 
\begin{align*}
    \theta(t)\in [\thetas-\eps^{1/3},\thetat] \,,
\qquad
    t \in [0,\tau] \,.
\end{align*} 
Therefore $(\bx(t),\theta(t))\in \bar{B}'$ holds, with $\bar{B}'$ as defined in \eqref{eqn:def ball B bar prime}. Let $\tau_1$ be the end-time of Interval I for Theorem~\ref{thm: conv quad lin 2D R neg}. Then for all $t\in[0,\tEndI]$ we have that $|\theta(t)-\thetas|\leqs \eps^{1/3}$. Thus,
\begin{align}
\nonumber |R(\bx(t),\theta(t))-R(\bxs,\thetas)|\leqs&\, |R(\bx(t),\theta(t))-R(\bx(t),\thetas)| + |R(\bx(t),\thetas)-R(\bxs,\thetas)|\\
\leqs&\, \max_{\bar{B}'}|\partial_{\theta} R| \cdot \eps^{1/3} + \max_{\bar{B}'}|\nabla_\bx R|\cdot \bar{D}\,\eps^{2/3}.\label{eqn:diff R-R int1, R const}
\end{align}
Furthermore, for each $t\in[\tEndI,\tau]$, we have that $\theta(t)\in[\thetas,\thetat]$. Hence,
\begin{equation}
|R(\bx(t),\theta(t))-R(\bxs,\thetas)|=|R(\bx(t),\theta(t))-R(\bxs,\theta(t))|, \quad \text{for all } t\in[\tEndI,\tau],\label{eqn:diff R-R int234, R const}
\end{equation}
since $R(\bxs,\cdot)$ is constant on $[\thetas,\thetat]$. Moreover,
\begin{equation}
|R(\bx(t),\theta(t))-R(\bxs,\theta(t))|\leqs \max_{\bar{B}'}|\nabla_\bx R| \cdot \bar{D}\,\eps^{2/3}.\label{eqn:diff R-R eps}
\end{equation}
It follows from \eqref{eqn:diff R-R int1, R const}, \eqref{eqn:diff R-R int234, R const} and \eqref{eqn:diff R-R eps} that there is a $D\geqs0$ such that
\begin{equation*}\label{eqn:bnd R if fnct R const}
|R(\bx(t),\theta(t))-R(\bxs,\thetas)|\leqs D\,\eps^{1/3},
\end{equation*}
for all $t\in[0,\tau]$.

Since $r$ satisfies the equation
\begin{equation*}
\eps\,\dfrac{\dr}{\dt} = -r + R(\bx,\theta),
\end{equation*}
the following inequalities are satisfied
\begin{equation*}\label{eqn:diff ineq if fnct R const}
-r + (R(\bxs,\thetas)-D\,\eps^{1/3}) \leqs \eps\,\dfrac{\dr}{\dt} \leqs -r + (R(\bxs,\thetas)+D\,\eps^{1/3}).
\end{equation*}
These differential inequalities have a constant subsolution $r(t)\equiv R(\bxs,\thetas)-D\,\eps^{1/3}$ and a constant supersolution $r(t)\equiv R(\bxs,\thetas)+D\,\eps^{1/3}$ for $t\in[0,\tau]$. Hence, for any initial condition $r(0)$ satisfying $R(\bxs,\thetas)-D\,\eps^{1/3}\leqs r(0) \leqs R(\bxs,\thetas)+D\,\eps^{1/3}$, the corresponding solution remains within these bounds for all $t\in[0,\tau]$. In our case, the initial condition $r(0)=\ro$ satisfies $\rs-a_3\,\eps\leqs \ro \leqs \rs+a_3\,\eps$ by assumption of the theorem, and $\rs=R(\bxs,\thetas)$. Consequently, if $\eps$ is sufficiently small, then
\begin{equation*} \label{eqn:r-rs eps if fnct R const}
    |r(t)-\rs|\leqs D\,\eps^{1/3}  \,,
\qquad
    t \in [0,\tau] \,.
\end{equation*}
which shows~\eqref{cond:jump-r} does not hold.\\

Now we assume that $R$ is not a constant on $[\thetas,\thetat]$. In this case we want to show \eqref{cond:jump-r}: that is, there exists a $Q > 0$ independent of $\Eps$, and a $t \in [0, \tau]$ such that
\begin{align*}
    \abs{r(t) - \rs} \geq Q > 0
\end{align*}
for all $\Eps$ sufficiently small. This jump in $r$ will occur in Interval III which is defined in the proof of Theorem \ref{thm: conv quad lin 2D R neg} as the time interval where $\theta$ escapes the bottleneck and reaches an $\BigO(1)$ distance from $\thetas$.

First we choose proper subintervals to study. Recall that in the case \RP, the constant $0 < \beta < 1$ can be chosen arbitrarily (with the theorem's statements remaining true). 
We now fix $\beta$ such that 
\begin{align*}
    R(\bxs,\thetas+\beta\,(\thetat-\thetas))
\neq 
   R(\bxs,\thetas) \,,
\end{align*} 
which is possible since $R(\bxs,\cdot)$ is not constant on $[\thetas,\thetat]$. In particular, we choose $\beta$ independently of $\eps$ and denote 
\begin{align*}
   \varphi_2:=\thetas+\beta\,(\thetat-\thetas) \,.
\end{align*}
Hence in the case \RP, we have 
\begin{align*}
    R(\bxs, \varphi_2)
\neq 
   R(\bxs,\thetas),
\end{align*} 

In the case \RN, the proof of Theorem \ref{thm: conv quad lin 2D R neg} still works if $\thetatwos$ is taken different from $(\thetas+\thetatwo)/2$, as long as $\thetatwos$ is $\mathcal{O}(1)$ away from both $\thetas$ and $\thetatwo$. Take such $\thetatwos\in(\thetas,\thetatwo)$, independent of $\eps$, such that 
\begin{align} \label{cond:thetas-2}
    R(\bxs,\thetatwos)\neq R(\bxs,\thetas).
\end{align} 
This is possible since $R(\bxs,\cdot)$ is continuous and $R(\bxs,\thetatwo) = 0 < R(\bxs,\thetas)$. 
Next, choose $0<\beta<1$ independent of $\eps$ such that 
\begin{align*}
   \varphi_2:=\thetas+\beta\,(\thetat-\thetas)\leqs \thetatwos
\qquad \text{and} \qquad
    R(\bxs,\varphi_2)\neq R(\bxs,\thetas).
\end{align*} 
The former is required in the proof of Theorem \ref{thm: conv quad lin 2D R neg}. The latter is possible by~\eqref{cond:thetas-2}.

In both cases {\RP} and \RN, let 
\begin{align*}
   \varphi_1:= \thetas + \gamma\,(\thetat-\thetas)
\end{align*} for some $0<\gamma<\beta$ independent of $\eps$ such that 
\begin{align} \label{cond:varphi-1-2}
   \thetas < \varphi_1 < \varphi_2 \,,
\qquad
   R(\bxs,\varphi_1)\neq R(\bxs,\varphi_2) \,.
\end{align} 
It is possible to find such $\gamma$ since $R(\bxs,\thetas)\neq R(\bxs,\varphi_2)$. Note that the distance between $\varphi_1$ and $\varphi_2$ is of order $\mathcal{O}(1)$, and both quantities are $\mathcal{O}(1)$ away from $\thetas$ and $\thetat$. Moreover, within Interval III, the quantity $\theta$ moves from $\thetas+\alpha\,\eps^{1/6}$ to $\varphi_2$, so for $\eps$ sufficiently small, we have:
\begin{align*}
   [\varphi_1,\varphi_2] \subseteq \, \text{Interval III},
\end{align*}
where by an abuse of notation, we denoted by $[\varphi_1,\varphi_2]$ the time interval in which $\theta$ changes from $\varphi_1$ to $\varphi_2$.
%

We now study the rate of change of $r$ with respect to $\theta$ on the subinterval $[\varphi_1, \varphi_2]$. Dividing the ODEs for $r$ and $\theta$, we obtain
\begin{align*}
\dfrac{\dr}{\dtheta}=\dfrac{-r+R(\bx,\theta)}{\dfrac1r H(\bxs,\theta)+\eta}.
\end{align*}
On Interval III, by~\eqref{bound:Inter-III-1} and~\eqref{eqn:bnd eta eps int3 2d}, we have
\begin{align*}
   R_0 \leqs r(t) \leqs \bar{R} \,, 
\qquad
   |\eta(t)|\leqs D_{12}\,\eps^{2/3} \,,
\qquad
   |R(\bx(t),\theta(t))-R(\bxs,\theta(t))|\leqs K_1\,\eps^{2/3}
\end{align*} 
for some $K_1>0$ independent of $\Eps$. 
Consequently, for $\Eps$ sufficiently small, 
\begin{align}\label{eqn:dr dtheta upper bnd  RP jump r}
    \left|\dfrac{\dr}{\dtheta}\right|
= \dfrac{|-r+R(\bx,\theta)|}{\left|\eta + H(\bxs,\theta)/r\right|} 
\leqs
    \dfrac{2\bar{R}}{({K_2}/{\bar{R})} - D_{12}\,\eps^{2/3}}
\leqs 
    \dfrac{2\bar{R}}{{K_2}/(2\bar{R})}=\dfrac{4\bar{R}^2}{K_2} \,,\
\qquad
   \theta \in [\varphi_1, \varphi_2] \,.
\end{align}
where we have defined
\begin{align*}
    K_2:=\min_{\phi\in[\varphi_1,\varphi_2]}H(\bxs,\phi) > 0 \,.
\end{align*} 
The strict positivity of $K_2$ is guaranteed by Assumption~\ref{ass:H 2d quad lin}.
Furthermore,
\begin{align}\label{eqn:dr dtheta lower bnd  RP jump r}
     \left|\dfrac{\dr}{\dtheta}\right|
= \dfrac{|-r+R(\bx,\theta)|}{|\eta + H(\bxs,\theta)/r|} 
\geqs
    \dfrac{|-r+R(\bx,\theta)|}{({K_3}/{R_0}) +D_{12}\,\eps^{2/3}}
\geqs 
     \dfrac{|-r+R(\bxs,\theta)|-K_1\,\eps^{2/3}}{{2K_3}/{R_0}}
\end{align}
where $K_3:=\max_{\phi\in[\thetas,\thetat]}|H(\bxs,\phi)| > 0$. The last inequality holds if $\eps$ is sufficiently small.

Assume that $r=r(\theta)$ changes at most $o(1)$ on $[\varphi_1,\varphi_2]$. 
Then by~\eqref{cond:varphi-1-2},  at least one of the quantities $|-r(\varphi_1)+R(\bxs,\varphi_1)|$ or $|-r(\varphi_2)+R(\bxs,\varphi_2)|$ must be $\mathcal{O}(1)$. By the uniform bound of $\dr/\dtheta$ in~\eqref{eqn:dr dtheta upper bnd  RP jump r}, there must therefore be an interval $\mathcal{I}$ of $\BigO(1)$ length that is either of the form $\mathcal{I}=[\varphi_1,\bar{\varphi}_1]$ or of the form $\mathcal{I}=[\bar{\varphi}_2,\varphi_2]$, on which it holds that 
\begin{align*}
   |-r(\theta)+R(\bxs,\theta)|\geqs K_4 \,,
\qquad
   \theta \in \mathcal{I} 
\end{align*} 
for some $K_4>0$ independent of $\Eps$. The lower bound \eqref{eqn:dr dtheta lower bnd  RP jump r} subsequently implies that 
\begin{align*}
    \abs{\dr/\dtheta} \geqs K_5>0 \,,
\qquad
    \theta \in \mathcal{I}
\end{align*} 
for some $K_5$ independent of $\Eps$ on the interval $\mathcal{I}$. Since the length of $\mathcal{I}$ is of order $\mathcal{O}(1)$, $r$ must have an $\mathcal{O}(1)$ change in $\mathcal{I}$, which contradicts the assumption that $r$ changes at most $o(1)$ on $[\varphi_1,\varphi_2]$.

Since there must be an order $\mathcal{O}(1)$ change in $r$ on $[\varphi_1,\varphi_2]$, it is possible to find a $\bar{\varphi}\in[\varphi_1,\varphi_2]$ and a constant $Q > 0$ such that 
\begin{align*}
   |r(\bar{\varphi})-\rs|\geqs Q > 0 \,,
\end{align*}  If $t$ is the corresponding time in Interval III at which $\theta(t)=\bar{\varphi}$, then $|r(t)-\rs|\geqs Q$ and the statement of the theorem is proved.
\end{proof}



\begin{remark}\label{rem:N moving}
The statements of Theorems \ref{thm: conv quad lin 2D R neg} and \ref{thm:r jump 2D} still hold if we allow all particles to move, that is, if we consider \eqref{eqn:so-polar} rather than \eqref{eqn:syst x polar 2D}. Note that in \eqref{eqn:so-polar} the only coupling between the particles is via their positions; see \eqref{eqn:so-polar b} and \eqref{eqn:so-polar c}. By assuming a uniform upper bound $\bar R$ for \emph{all} functions $R_i$, one immediately obtains a uniform bound for the speeds $r_i$. Hence, within a time interval of $\Ord(\Eps^{2/3})$, all the positions $x_i$ change by at most $\Ord(\Eps^{2/3})$. The analysis that follows is then similar to the case in which only one particle is moving.
\end{remark}


\section{Numerics}
\label{sec:num}
In this section, we first investigate the scenario presented in Figure \ref{fig:trajectories zoom r theta}, and also in Figure \ref{fig:HR}. We call this Run~1. This is a situation where the function $R_1$ remains strictly positive throughout our interval of interest; compare the case {\RP } in the proof of Theorem \ref{thm: conv quad lin 2D R neg}. For several values of $\eps$, we initialize the system just before the double root of $H_1$ gets lost, and we examine the sharp transition of the velocity of particle 1. In particular, we measure the time it takes for $\theta_1$ to reach an order $\Ord(\eps^{2/3})$ distance from $\thetat_1$, and we start measuring at root loss. Theory predicts that this transition should take an $\Ord(\eps^{2/3})$ amount of time; cf.~the statement of Theorem \ref{thm: conv quad lin 2D R neg}. Note that the values of $\thetas_1$ and $\thetat_1$ can be calculated \textit{a priori}, once we have the function $H_1(\bxs_1,\cdot)$ at root loss. 

In the numerical calculations, $K$ is taken to be the Morse potential. This is the same potential as in \cite{EversFetecauRyzhik} and we used the same values for the parameters. The function $g$ is taken as in Figure \ref{fig:vision}(b), with parameters $a=5$ and $b=\pi$.

We consider two different cases: (1) the case in which the positions of all particles are fixed, except for the one that undergoes the change in velocity; and (2) the case in which all four particles are evolving during the transition interval. The results are presented in Figure \ref{fig:fit R pos}. 
In both cases, a linear fit shows that, up to a small numerical error, the amount of time needed for $\theta_1$ to make the transition scales as $\eps^{2/3}$. For the case where only particle 1 is moving, this is in agreement with exactly what was proven in Theorem \ref{thm: conv quad lin 2D R neg}. In Remark \ref{rem:N moving} we argued that the statement of Theorem \ref{thm: conv quad lin 2D R neg} also holds if \textit{all} particles are allowed to evolve. The linear fit through the diamonds in Figure \ref{fig:fit R pos} supports this result and provides the correct scaling.

\begin{figure}[h]%
\centering
\includegraphics[height=0.35\textwidth]{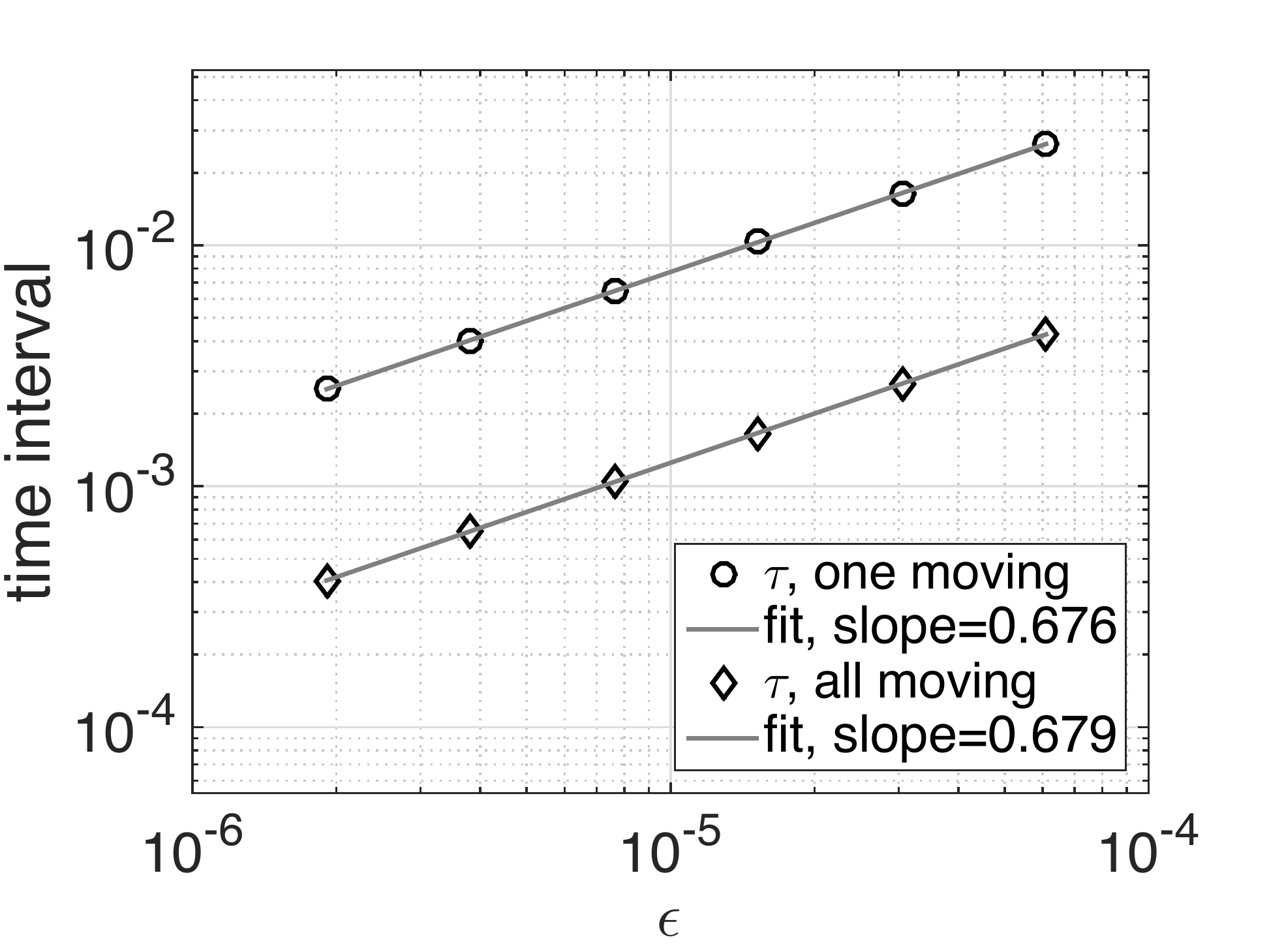}
\caption{The amount of time $\tau$ it takes in Run~1 from root loss until reaching $|\theta_1-\thetat_1|=\eps^{2/3}$, plotted as a function of $\eps$. Circles: Numerical calculations corresponding to \eqref{eqn:syst x polar 2D}. That is, we focus on particle 1 whose velocity undergoes the fast transition. All other particles have their positions fixed. The slope of the fitted line is approximately $\frac23$, which is in agreement with Theorem \ref{thm: conv quad lin 2D R neg}, where the transition time interval is proved to be $\Ord(\eps^{2/3})$. Diamonds: Numerical calculations corresponding to \eqref{eqn:so-polar}, where again particle 1 undergoes a fast transition in velocity, but all particles' positions evolve during the transition interval. The slope of the fitted line is approximately $\frac23$, which supports the claim in Remark \ref{rem:N moving} that Theorem \ref{thm: conv quad lin 2D R neg} predicts the correct length of the transition interval also when \textit{all} particle positions evolve.\\
Because the two datasets are nearly overlapping, in the plot the circles (and the corresponding linear fit) have been translated upwards; it is only the slope that we are interested in.}%
\label{fig:fit R pos}%
\end{figure}

Next, we take the same number of particles (i.e.~4) and the same parameters, but a different initial configuration. Call this Run~2. We consider a scenario that leads to root loss in one of the particles (in this case, particle 2), while the function $R_2(\bxs_2,\cdot)$ is negative in a subinterval of $[\thetas_2,\thetat_2]$; cf.~the case {\RN} in the proof of Theorem \ref{thm: conv quad lin 2D R neg}. 
The spatial coordinate $\bxs_2$ is such that $H_2(\bxs_2,\cdot)$ has a double root. In particular, this double root occurs at $\theta=\thetas_2$, while a simple root is present at $\thetat_2>\thetas_2$, where the derivative of $H_2(\bxs_2,\cdot)$ is negative. See Figure \ref{fig:HR R neg}; the insert is included to show that $R_2(\bxs_2,\thetas_2)>0$.

\begin{figure}[h]%
\centering
\includegraphics[height=0.45\textwidth]{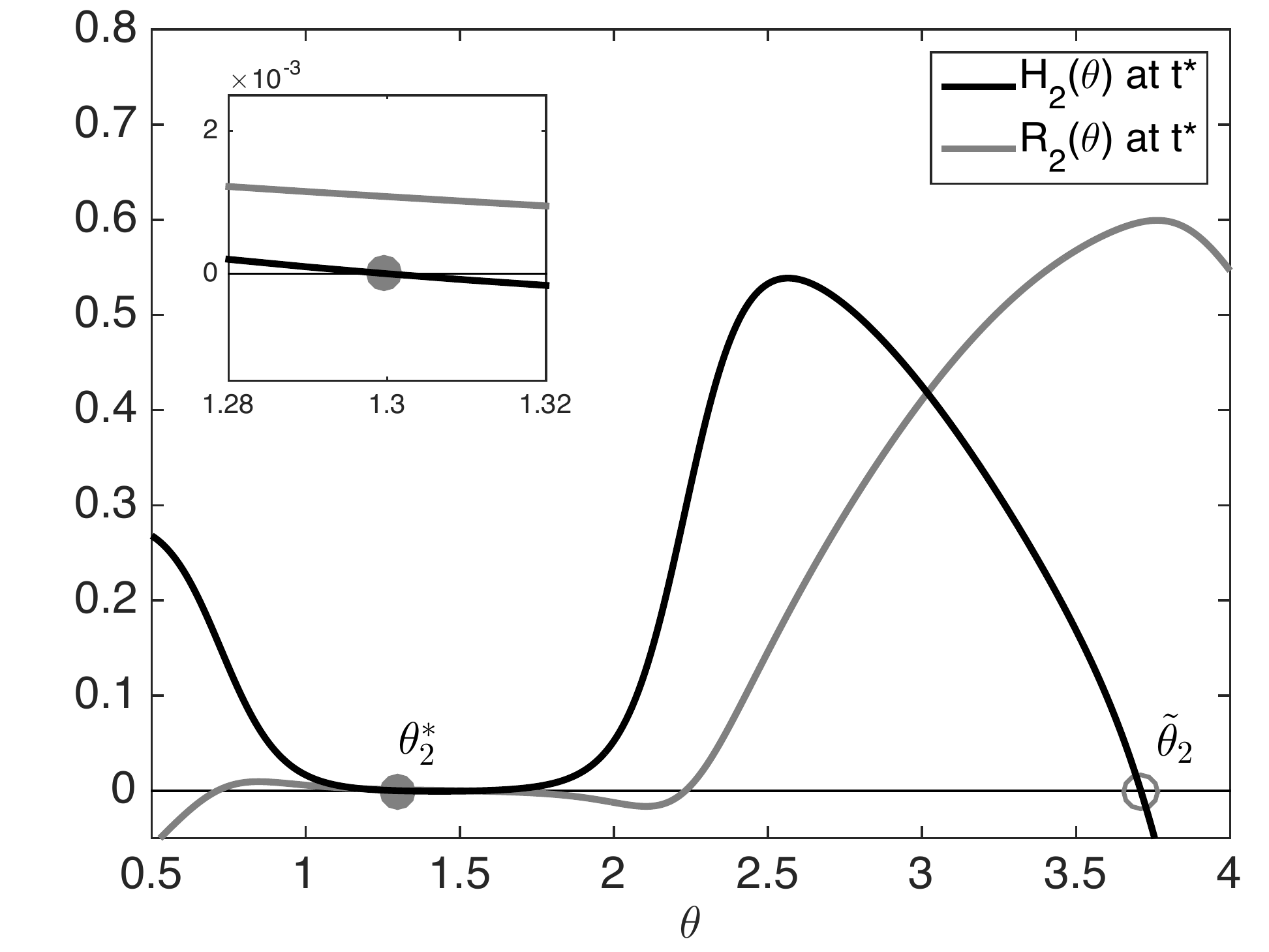}
\caption{The function $H_2$ associated to particle 2 (solid black line), drawn as a function of $\theta$ at the breakdown time $t=t^*$ of Run~2. A double root is present at $\thetas_2\approx1.30$ (filled circle). The open circle indicates $\thetat_2\approx 3.71$, where $H_2$ has a single root and the derivative is negative. Note that $R_2$ is negative in a subinterval of $[\thetas_2,\thetat_2]$, but $R_2(\thetat_2)>0$.
The insert shows the graph near $\thetas_2$ in the \textit{last} numerical timestep in which a root of $H_2$ is present there. The root is still a simple root; the actual double root is very hard to capture numerically, in particular because the graph of $H_2$ is very shallow in that region. The aim of the insert, however, is to show that $R_2$ is strictly positive (be it small) near $\thetas_2$.
}
\label{fig:HR R neg}%
\end{figure}

In Figure \ref{fig:r theta R neg} we show that the angle and the absolute value of the velocity of particle 2 both make a sharp transition in Run~2. We present numerical results for $\eps=10^{-2}, 10^{-3}$, and $10^{-4}$. For decreasing $\eps$, the plots exhibit similar behaviour as in Figures \ref{fig:trajectories zoom r theta}(b) and \ref{fig:trajectories zoom r theta}(c). Note that in Figure \ref{fig:r theta R neg}(b) a logarithmic scale is used on the vertical axis.

\begin{figure}[h]%
\centering
\includegraphics[height=0.35\textwidth]{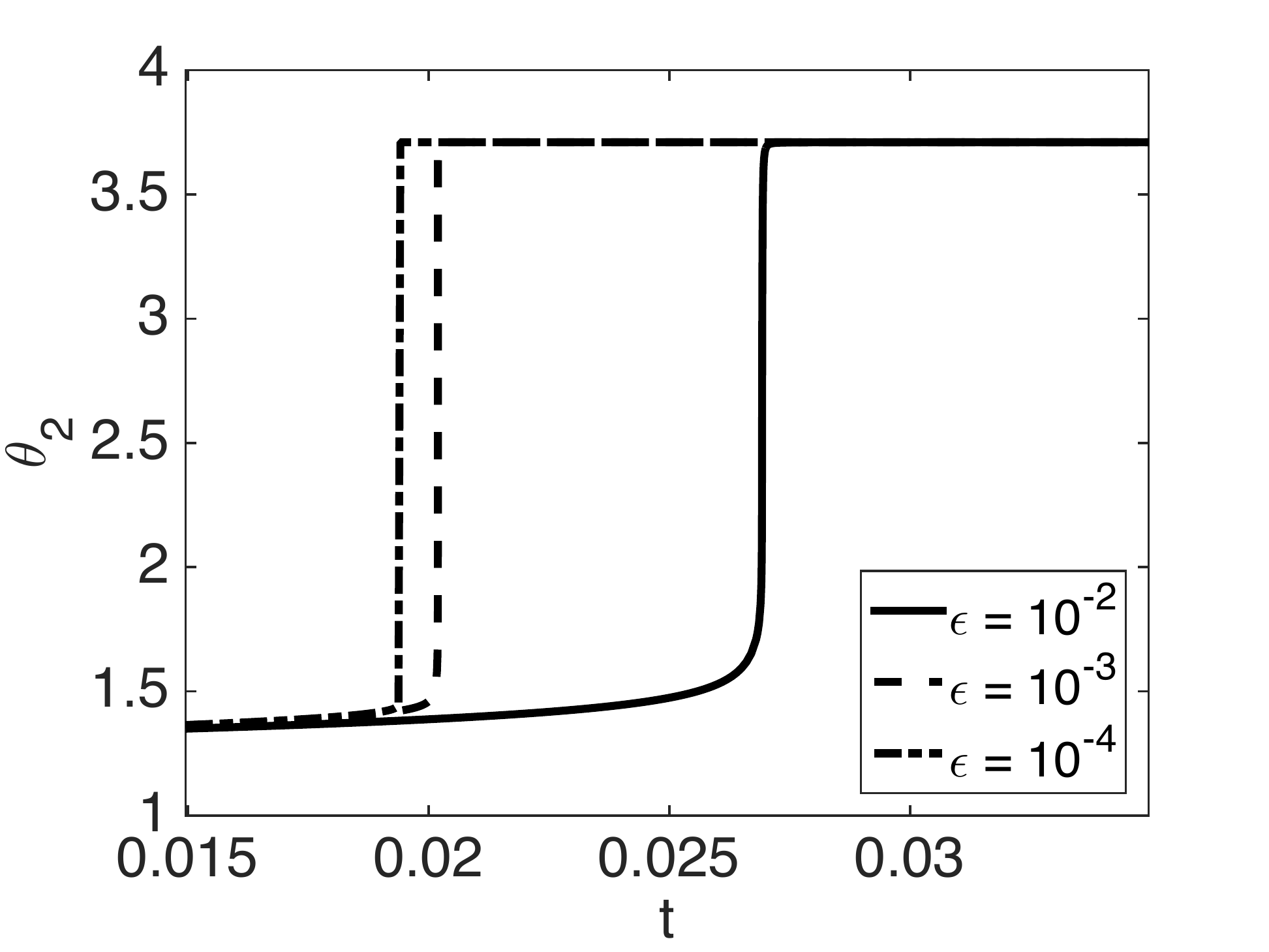}
$~~~$
\includegraphics[height=0.35\textwidth]{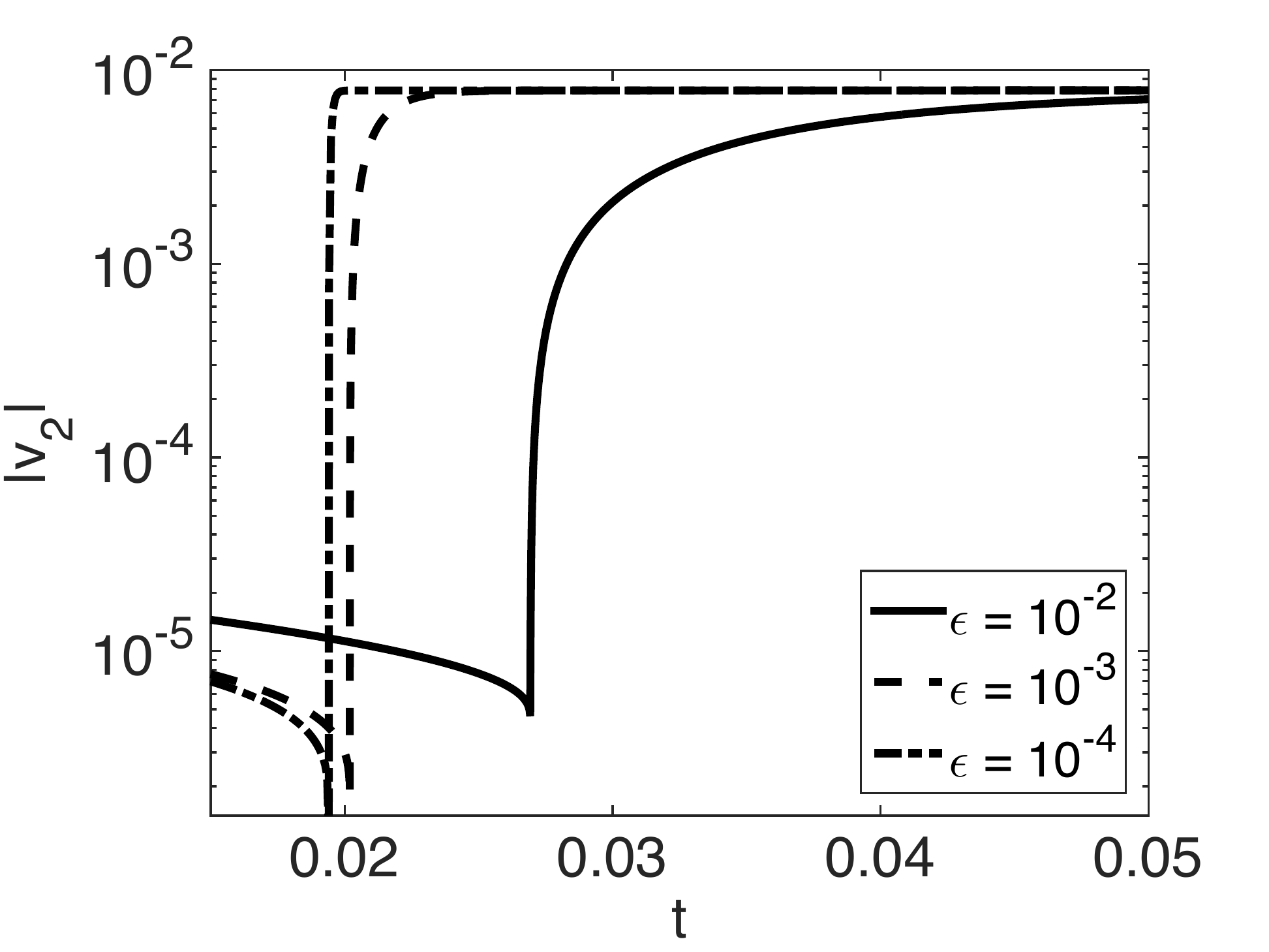}\\
(a) \hspace{0.45\textwidth} (b)
\caption{Time evolution of the velocity of particle 2 in Run~2. This particle is a member of a group of four particles. (a) and (b) Evolution of the angle $\theta_2$ and the magnitude of the velocity $|v_2|$, respectively: solutions of \eqref{eqn:so} for $\eps=10^{-2},10^{-3}$, and $10^{-4}$ capturing sharp transitions in $\theta_2$ and $|v_2|$.}%
\label{fig:r theta R neg}%
\end{figure}

For Run~2, we investigate in Figure \ref{fig:fit R neg} how the length of the transition interval scales with $\eps$. As we did for Run~1, we measure the time from root loss until $\theta_2$ reaches an $\Ord(\eps^{2/3})$ distance from $\thetat_2$. Numerical calculations were done for the case in which only particle 2 moves, as well as for the case in which all four particles move. In both cases, the linear fit in Figure \ref{fig:fit R neg} shows that $\theta_2$ needs an order $\Ord(\eps^{2/3})$ amount of time to make the transition. This is in agreement with Theorem \ref{thm: conv quad lin 2D R neg} and Remark \ref{rem:N moving}. 


\begin{figure}[h]%
\centering
$~~~$
\includegraphics[height=0.35\textwidth]{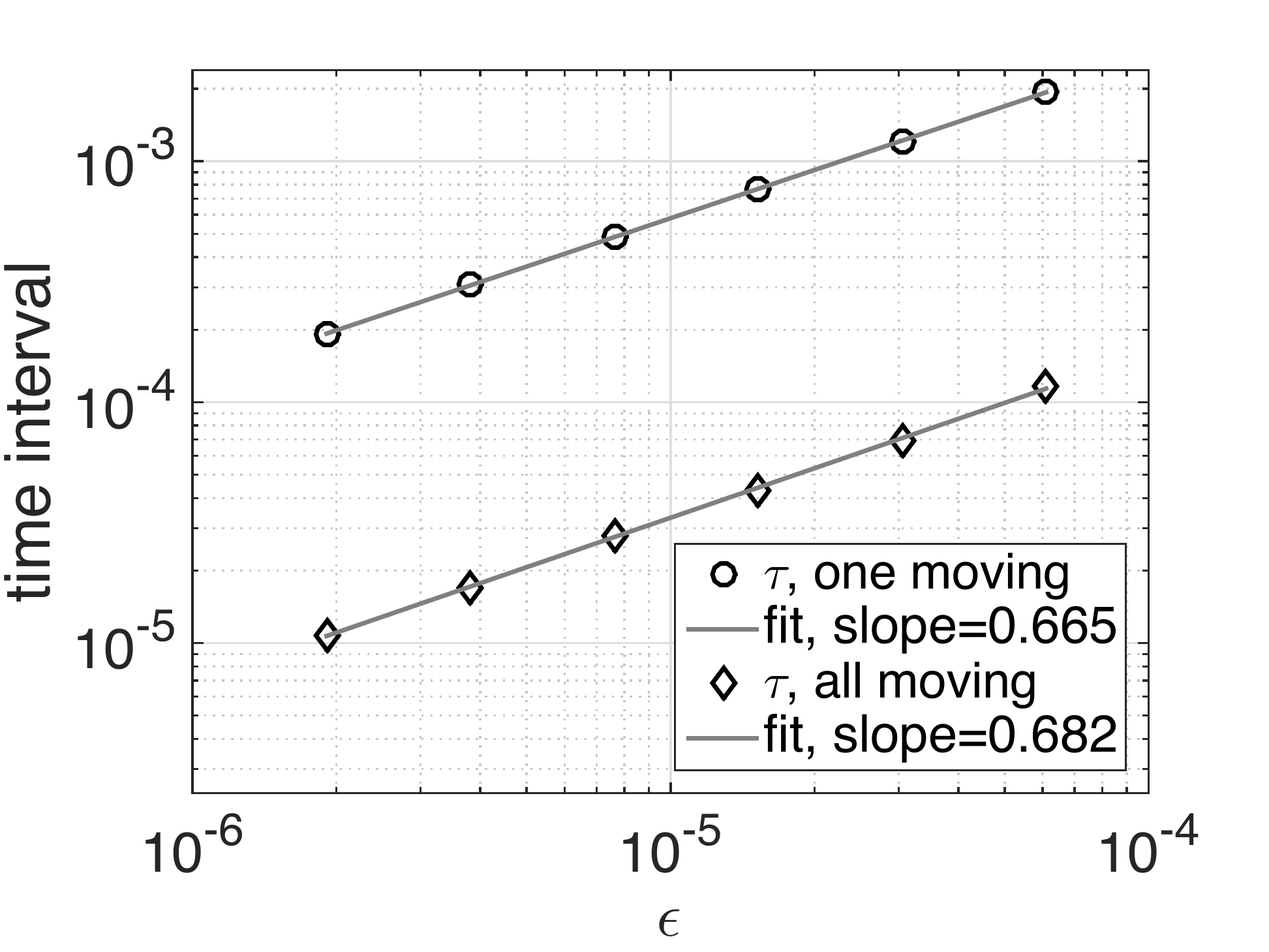}
\caption{The amount of time $\tau$ it takes from root loss until reaching $|\theta_2-\thetat_2|=\eps^{2/3}$, plotted as a function of $\eps$. This graph is for Run~2 where the value of $R_2(\bxs,\cdot)$ is negative in some part of the interval $[\thetas_2,\thetat_2]$. Circles: Numerical calculations corresponding to \eqref{eqn:syst x polar 2D}. The velocity of particle 2 undergoes the fast transition. All other particles have their positions fixed. The slope of the fitted line is \textit{approximately} $\frac23$, which is in agreement with Theorem \ref{thm: conv quad lin 2D R neg}. Diamonds: Numerical calculations corresponding to \eqref{eqn:so-polar}, where again particle 2 undergoes a rapid change in velocity, but all particles' positions evolve during the transition interval. The slope of the fitted line is approximately $\frac23$, which supports the claim that Theorem \ref{thm: conv quad lin 2D R neg} predicts the correct length of the transition interval also when \textit{all} particle positions evolve.}%
\label{fig:fit R neg}%
\end{figure}


\begin{appendix}

\section{Sketch of Proof of Theorem~\ref{thm: conv k ell 1D}}\label{app:proof thm 1D gen ic}
In this appendix we briefly explain the proof of Theorem~\ref{thm: conv k ell 1D}. As mentioned before, the proof parallels that for Theorem~\ref{thm: conv quad lin 2D R neg} with $r$ constant. 

Similar as before, we use the quantity $\eta(t):=\F(x(t),v(t))-\F(\xs,v(t))$ to make a change of variable and rewrite system \eqref{eqn:syst x 1D general} as
\begin{equation}\label{eqn:syst eta 1D general}
\left\{
  \begin{array}{l}
  	\dfrac{\deta}{\dt} = \underbrace{\partial_x\F(x,v)\cdot v}_{=:A} +\underbrace{\dfrac{1}{\eps} \bigg(\partial_v\F(x,v)-\partial_v\F(\xs,v)\bigg)\cdot\F(x,v)}_{=:B},\\
  	\\
	\eps \dfrac{\dv}{\dt} = \F(\xs,v)+\eta,\\ 
	\\
	\eta(0)=0, v(0)=\vs.
  \end{array}
\right.
\end{equation}
Although not written explicitly, the variable $x$ evolves according to $\dx/\dt=v$, like before.

There are two powers of $\eps$ in Theorem~\ref{thm: conv k ell 1D}: the length of the time interval for the transition from $\vs$ to $\vt$ to occur and the final distance of $v$ to $\vt$. 

\subsubsection*{\underline{Interval I}} The overall time scale is determined by Interval I, where $A$ dominates $B$ and $\F(\xs, v)$ becomes comparable to $\eta$. In this interval, we have
\begin{align*}
   \abs{x(t) - \xs} = \BigO(t) \,,
\qquad
   \eta (t) = \BigO(t) \,.
\end{align*}
Moreover, $\eta$ is the main driving force of $v$. Hence,
\begin{align*}
    \abs{v - \vs} = \BigO\vpran{\frac{t^2}{\eps}} \,.
\end{align*}
The balance of $\F(\xs,  v)$ with $\eta$ then gives
\begin{align}\label{eqn:sim end int1}
    (v-\vs)^{2k} \sim t
\quad\Rightarrow\quad
    \vpran{t^2/\eps}^{2k} \sim t
\quad\Rightarrow\quad
   t = \BigO(\eps^{2k/(4k-1)}) \,,
\end{align}
which gives the overall transition time scale. Denote by time $\tEndI$ the end of Interval I. Then we have
\begin{align*}
v(\tEndI)-\vs = \Ord\left( \eps^{1/(4k-1)}\right),
\end{align*}
Note that in the case where $k = 1$, that is, we have the double-root loss as in Section~\ref{sec:toy} or Section~\ref{sec:2D-varr}, we recover the time scale $\BigO(\eps^{2/3})$. Moreover, $v(\tEndI)-\vs = \Ord\left( \eps^{1/3}\right)$, cf.~\eqref{goal:interval-1}.

\subsubsection*{\underline{Interval II}} In Interval II we use that $A$ is still positive and dominates $B$, which implies that $\eta$ remains positive. In Interval II we have moreover that $\F(\xs,v)\sim (v-\vs)^{2k}$. Therefore,
\begin{align*}
\eps\,\dfrac{{\rm d}(v-\vs)}{\dt} \gtrsim (v-\vs)^{2k},
\end{align*}
with initial condition
\begin{align*}
v(\tEndI)-\vs = \Ord\left( \eps^{1/(4k-1)}\right).
\end{align*}
It follows that 
\begin{equation*}
v(t)-\vs \gtrsim \left(\dfrac{\eps}{\eps^{2k/(4k-1)}-  C\,(t-\tEndI)}\right)^{1/(2k-1)}, 
\end{equation*}
for some constant $C$; compare \eqref{eqn:lower bnd theta-thetas int2 2d}. By taking $t:=\tEndII=\tEndI +\Ord(\eps^{2k/(4k-1)})-\Ord(\eps^{(4k+1)/(2(4k-1))})$, we find that
\begin{equation}\label{eqn:v - vs t2 sim}
v(\tEndII)-\vs \gtrsim \eps^{(4k-3)/(2(2k-1)(4k-1))}.
\end{equation}
This distance is (slightly) larger than the one at the end of Interval I. It turns out that $v(\tEndII)$ is actually far enough from $\vs$ to guarantee that $v$ will not remain trapped in the bottleneck. We will see this in the next interval.

For the case $k=1$, the orders of magnitude are $\tEndII=\tEndI+\Ord(\eps^{2/3})-\Ord(\eps^{5/6})$ and $v(\tEndII)-\vs\gtrsim\eps^{1/6}$. These correspond to \eqref{eqn:def tEndII 2d} and the first equation in \eqref{goal:Interval-II}.

\subsubsection*{\underline{Interval III}} In Interval III we allow $A+B$ (and hence $\eta$) to become negative. This is the interval in which $v$ reaches an $\Ord(1)$ distance from $\vs$.

In this interval $\eta$ may be negative, but it is still small: $|\eta| \lesssim \eps^{2k/(4k-1)}$. Consequently,
\begin{equation*}
\eps\,\dfrac{{\rm d}(v-\vs)}{\dt}\gtrsim (v-\vs)^{2k} - \Ord(\eps^{2k/(4k-1)}),
\end{equation*}
with initial condition satisfying \eqref{eqn:v - vs t2 sim}. Due to this initial condition, $(v-\vs)^{2k}$ initially dominates the $\Ord(\eps^{2k/(4k-1)})$ term. It follows that $v$ increases throughout the interval \textit{and} the term $(v-\vs)^{2k}$ remains the dominant one. Therefore, the second term can be absorbed into the first one, to obtain
\begin{equation}
\dfrac{{\rm d}(v-\vs)}{\dt} \gtrsim \eps^{-1}\,(v-\vs)^{2k}.
\end{equation}
Compare this inequality, for the case $k=1$, to \eqref{eqn:diff ineq theta int3 2d}. This is the same estimate as in Interval II, but the initial condition is different. Taking \eqref{eqn:v - vs t2 sim} into account, we find
\begin{equation}
v(t)-\vs \gtrsim \left(\dfrac{\eps}{\eps^{1-(4k-3)/(2(4k-1))}-  C\,(t-\tEndII)}\right)^{1/(2k-1)},
\end{equation}
for some constant $C$. This guarantees that in a time interval after $\tEndII$ of $\Ord(\eps^{(4k+1)/(2(4k-1))})-\Ord(\eps)$ length, $v$ moves an $\Ord(1)$ distance away from $\vs$. We remark that if $k=1$ we have the same timescales as in \eqref{eqn:def tEndIII 2d}.

\subsubsection*{\underline{Interval IV}} The final distance of $v$ to $\vt$ is determined in Interval IV, where $v$ starts $\BigO(1)$ away from both $\vs$ and $\vt$ and the main driving force for $v$ is $\F(\xs, v)$. More precisely, it is the term $(v - \vt)^{2\ell-1}$ in $\F(\xs, v)$; see \eqref{eqn:F factors}. 
Over this interval, we have
\begin{align*}
   \abs{x(t) - \xs} = \BigO\vpran{\eps^{2k/(4k-1)}} \,,
\qquad
   \eta (t) = \BigO\vpran{\eps^{2k/(4k-1)}} \,.
\end{align*}
Consequently,
\begin{align}\label{eqn: diff ineq sim int4}
   \eps \frac{{\rm d} (v-\vt)}{\dt} 
\gtrsim 
   - (v(t)-\vt)^{2\ell-1} - \BigO\vpran{\eps^{2k/(4k-1)}} \,.
\end{align}
Note that $- (v(t)-\vt)^{2\ell-1}$ is positive, since $(v(t)-\vt)^{2\ell-1}$ is an odd power of a negative quantity, hence negative. If $k=\ell=1$, then \eqref{eqn: diff ineq sim int4} is the analogue of \eqref{eqn:diff ineq theta-thetat int4 2d} in the proof of Theorem \ref{thm: conv quad lin 2D R neg}.

Since $(v-\vt)$ is $\Ord(1)$ initially, we can absorb the $-\BigO\vpran{\eps^{2k/(4k-1)}}$ term into the positive term $-(v-\vt)^{2\ell-1}$, \textit{until} they balance. That point in time marks the end of our interval of consideration, and it is the size of $\eta$ that determines $|v-\vt|$ at the end.

The driving mechanism in this interval is
\begin{align*}
   \eps \frac{{\rm d} (v-\vt)}{\dt} 
\gtrsim 
   - (v-\vt)^{2\ell-1} \,.
\end{align*}
The subsequent step is based on differential inequalities (like before), but it requires a separate approach for $\ell=1$ and for $\ell>1$. The resulting inequalities yield that at the end of Interval IV, we have
\begin{align*}
    \abs{v - \vt} = \BigO\vpran{\eps^{(2k)/(4k-1)}} \;\text{if }\ell=1,\quad  \text{and} \quad \abs{v - \vt} = \BigO\vpran{\eps^{(2k-1)/((4k-1)(2\ell-1))}} \;\text{if $\ell>1$.}
\end{align*}

The powers of $\eps$ in terms of $k, \ell$ show that in general the larger $k$ is, the more time it takes the system to leave the bottleneck that is present around $\vs$. Moreover, if $\ell$ increases, then the graph of $\F$ becomes ``flatter" near $\vt$. In the theorem, this is expressed by the fact that $\nu$ decreases and thus the final distance between $v$ and $\vt$ is larger. If $k=\ell=1$ then we have exactly $|v - \vt| = \BigO\vpran{\eps^{2/3}}$ as we have in Theorem \ref{thm: conv quad lin 2D R neg} for $|\theta-\thetat|$. The order of magnitude of Interval IV is arbitrarily close to $\Ord(\eps)$. That is, the length is $\Ord(\eps^{1-\lambda})$ for arbitrarily small $\lambda>0$, similar to Interval IV in the proof of Theorem \ref{thm: conv quad lin 2D R neg}. We omit further details here. 

\end{appendix}

\bibliographystyle{alpha}
\bibliography{bibliography}

\end{document}